\documentclass[11pt]{article}
\usepackage{amsmath,amssymb, amsfonts, cancel,fancyhdr,sectsty,color, enumitem, amsthm, graphicx, framed, multirow, mathrsfs, eucal, threeparttable, rotating, bm}
\usepackage[title, titletoc]{appendix}
\usepackage[justification=centering]{caption}
\usepackage{setspace}
\usepackage[default]{harvard}
\usepackage{hyperref}
\hypersetup{pdfborder={0 0 0} }

\newtheorem{theorem}{Theorem}
\newtheorem{corollary}{Corollary}

\newtheorem{lemma}{Lemma}
\newtheorem{assumption}{Assumption}

\theoremstyle{definition}
\newtheorem{example_tmp}{Example}
\newenvironment{example}
	{ \begin{example_tmp} 	}
	{ 
		
		\qed 
		\end{example_tmp} 
	}
\theoremstyle{definition}
\newtheorem{remark_tmp}{Remark}
\newenvironment{remark}
	{ \begin{remark_tmp} 	}
	{ 
		
		\qed 
		\end{remark_tmp} 
	}

\allowdisplaybreaks[1]

	\newcommand\inde{\protect\mathpalette{\protect\independenT}{\perp}}
	\def\independenT#1#2{\mathrel{\rlap{$#1#2$}\mkern2mu{#1#2}}}

	\newcommand{\indicatortmp}{{\rm 1\hspace*{-0.84ex}%
	\rule{0.06ex}{1.37ex}\hspace*{0.9ex}}}
	\newcommand{\indicator}{\ensuremath{\indicatortmp\!}}
	
	\newcommand{\sumi}{\sum_{i=1}^n}

	\DeclareMathOperator*{\argmin}{arg\,min}

	\newcommand{\E}{\mathbb{E}}
	\newcommand{\En}{\mathbb{E}_n}
	
	\renewcommand{\P}{\mathbb{P}}

	\newcommand{\Eqref}[1]{Eqn.\ \eqref{#1}}
	
	\newcommand{\vertiii}[1]{{\left\vert\kern-0.25ex\left\vert\kern-0.25ex\left\vert #1  \right\vert\kern-0.25ex\right\vert\kern-0.25ex\right\vert_{2,1}}}

	\newcommand{\by}{b^y_{\t,i}}
	\newcommand{\bias}{b_{s}}

	\newcommand{\D}{D} %Treatment category, ie dose
	\renewcommand{\d}{d_i} %Realized dose for unit i
	\newcommand{\dt}{\d^{\t}} %realized indicator indicator for Treatment = specific level
	\newcommand{\T}{\mathcal{T}} %highest treatment level
	\newcommand{\Tbar}{\overline{\mathcal{T}}} %number of treatment levels
	\renewcommand{\t}{t} %generic treatment level

	\newcommand{\N}{\mathbb{N}}

	\newcommand{\NT}{\N_\T}		%used for all but the first treatment, for mlogit
	\newcommand{\NTbar}{\overline{\N}_\T}		%used for all treatments

	\newcommand{\Np}{\N_p}

	\newcommand{\It}{\mathbb{I}_\t}
	\newcommand{\sumt}{\sum_{i \in \It}}
	\newcommand{\maxt}{\max_{i \in \It}}
	
	\newcommand{\Ent}{\mathbb{E}_{n,\t}}
	
	\newcommand{\sumT}{  \sum_{\t \in \NT}  }
	
	\newcommand{\maxT}{\max_{\t \in \NT}}

	\newcommand{\sumTbar}{ \sum_{\t \in \NTbar}  }
	
	\newcommand{\maxTbar}{\max_{\t \in \NTbar}}
	\newcommand{\minTbar}{\min_{\t \in \NTbar}}

	\newcommand{\maxi}{\max_{i \leq n }}

		\newcommand{\sumj}{\sum_{j \in \N_p}}
		\newcommand{\maxj}{\max_{j \in \N_p}}

		\newcommand{\mut}{\mu_\t}
		\newcommand{\muthat}{\hat{\mu}_\t}
		\newcommand{\mutt}{\mu_{\t,\t'}}
		\newcommand{\mutthat}{\hat{\mu}_{\t,\t'}}

		\newcommand{\drf}{\bm{\mu}}
		\newcommand{\drfhat}{\hat{\drf}}
	
		\newcommand{\tot}{\bm{\tau}}
		\newcommand{\tothat}{\hat{\bm{\tau}}}

	\newcommand{\Xn}{\bm{X}}
	\newcommand{\Xnt}{\bm{X}_{\t}}

		\newcommand{\newdot}{\bm{\cdot}}

	\newcommand{\btrue}{\beta^*_{\newdot, \newdot}}
	\newcommand{\btruet}{\beta^*_\t}
	
	\newcommand{\btruetj}{\beta^*_{\t, j}}

	\newcommand{\btilde}{\tilde{\beta}_{\newdot, \newdot}}
	\newcommand{\btildet}{\tilde{\beta}_\t}
	\newcommand{\btildej}{\tilde{\beta}_{\newdot, j}}
	\newcommand{\btildetj}{\tilde{\beta}_{\t, j}}

	\newcommand{\bany}{\beta_{\newdot, \newdot}}
	\newcommand{\banyt}{\beta_\t}

	\newcommand{\bhat}{\hat{\beta}_{\newdot, \newdot}}
	\newcommand{\bhatt}{\hat{\beta}_\t}
	
	\newcommand{\bhattj}{\hat{\beta}_{\t, j}}

	\newcommand{\gtrue}{\gamma^*_{\newdot, \newdot}}
	\newcommand{\gtruet}{\gamma^*_\t}
	
	\newcommand{\gtruetj}{\gamma^*_{\t, j}}

	\newcommand{\gtilde}{\tilde{\gamma}_{\newdot, \newdot}}
	\newcommand{\gtildet}{\tilde{\gamma}_\t}
	\newcommand{\gtildej}{\tilde{\gamma}_{\newdot, j}}
	\newcommand{\gtildetj}{\tilde{\gamma}_{\t, j}}

	\newcommand{\gany}{\gamma_{\newdot, \newdot}}
	\newcommand{\ganyt}{\gamma_\t}
	\newcommand{\ganyj}{\gamma_{\newdot, j}}

	\newcommand{\ghat}{\hat{\gamma}_{\newdot, \newdot}}
	\newcommand{\ghatt}{\hat{\gamma}_\t}
	
	\newcommand{\ghattj}{\hat{\gamma}_{\t, j}}

	\newcommand{\dtilde}{\tilde{\delta}_{\newdot, \newdot}}
	\newcommand{\dtildet}{\tilde{\delta}_\t}
	\newcommand{\dtildej}{\tilde{\delta}_{\newdot, j}}
	\newcommand{\dtildetj}{\tilde{\delta}_{\t, j}}

	\newcommand{\dany}{\delta_{\newdot, \newdot}}
	\newcommand{\danyt}{\delta_\t}
	\newcommand{\danyj}{\delta_{\newdot, j}}
	\newcommand{\danytj}{\delta_{\t, j}}

	\newcommand{\dhat}{\hat{\delta}_{\newdot, \newdot}}
	\newcommand{\dhatt}{\hat{\delta}_\t}
	\newcommand{\dhatj}{\hat{\delta}_{\newdot, j}}
	\newcommand{\dhattj}{\hat{\delta}_{\t, j}}

	\newcommand{\M}{\mathcal{M}}
	\newcommand{\SSE}{\mathcal{E}}

	\renewcommand{\H}{\mathcal{H}}

	\newcommand{\ptany}{\hat{p}_\t(\{{x_i^*}'\ganyt\}_{\NT})}
	\newcommand{\pttilde}{\hat{p}_\t(\{{x_i^*}'\gtildet\}_{\NT})}
	\newcommand{\pthat}{\hat{p}_\t(\{{x_i^*}'\ghatt\}_{\NT})}
	\newcommand{\pttrue}{\hat{p}_\t(\{{x_i^*}'\gtruet\}_{\NT})}

	\newcommand{\dbarphi}{\overline{\overline{\phi}}}

\DeclareMathOperator*{\supp}{supp}

\begin{document}

\title{Robust Inference on Average Treatment Effects with Possibly More Covariates than Observations\thanks{An online suppplement contains additional proofs and simulations. I am deeply grateful to Matias Cattaneo for advice and support. I am indebted to Xuming He, Lutz Kilian, and Jeffrey Smith for thoughtful feedback and discussions. I thank Victor Chernozhukov for pointing to the relevant latest results, obtained in joint work Alexandre Belloni and Christian Hansen, and the latter two authors for conversations in the early stages of this project. I benefited from discussions with Rosa Matzkin, Blaise Melly, and Jack Porter. I also thank the co-editor, Han Hong, and two reviewers for their detailed comments and suggestions that improved the paper.}}
\author{Max H. Farrell\thanks{Correspondence to: University of Chicago Booth School of Business, 5807 South Woodlawn Avenue, Chicago, IL 60637, United States. Tel: +1  773-834-0161; Email: {\tt max.farrell@chicagobooth.edu}; Web: \url{http://faculty.chicagobooth.edu/max.farrell/}. } \\ University of Chicago Booth School of Business}
\date{June 9, 2015}
\maketitle
\vspace{-0.5in}
\begin{center}
Updated: February 1, 2018\footnote{The published version, \citeasnoun{Farrell2015_JoE}, contains an error in the proof [present also in \citeasnoun{Farrell2015_JoE--Supplement}] which is rectified here. Assumption 3(c) is required for the results of Section 5 to be valid; see the author's website for further detail. I am grateful to Whitney Newey for alerting me to this error.}
\end{center}
\bigskip

\thispagestyle{empty}
\setcounter{page}{0}

\begin{abstract}
This paper concerns robust inference on average treatment effects following model selection. Under selection on observables, we construct confidence intervals using a doubly-robust estimator that are robust to model selection errors and prove their uniform validity over a large class of models that allows for multivalued treatments with heterogeneous effects and selection amongst (possibly) more covariates than observations. The semiparametric efficiency bound is attained under appropriate conditions. Precise conditions are given for any model selector to yield these results, and we specifically propose the group lasso, which is apt for treatment effects, and derive new results for high-dimensional, sparse multinomial logistic regression. Both a simulation study and revisiting the National Supported Work demonstration show our estimator performs well in finite samples.
\end{abstract}

\bigskip



\newpage
\onehalfspacing

%%%%%%%%%%%%%%%%%%%%%%%%%%% 
%%%%%%%%%%%%%%%%%%%%%%%%%%%
\section{Introduction}
	\label{sec-intro}

Model selection has always had a place in empirical economics, whether or not it is formally acknowledged. A key problem in modern empirical work is that researchers face datasets with large numbers of variables, sometimes more than observations. A complementary problem is that economic theory and prior knowledge may mandate controlling for certain variables, but are generally silent regarding functional form. These two problems force researchers to search for a model that is simultaneously parsimonious and adequately flexible. Many formal methods are computationally infeasible with a large number of variables. A typical response to this challenge is to iteratively search over a small set of alternative specifications, guided only by the researcher's taste and intuition. But no matter the approach used, subsequent inference almost never takes accounts for this ``specification search'' and the resulting confidence intervals are not robust to model selection mistakes, and hence are unreliable in empirical work.

This problem is particularly important in estimating average treatment effects under selection on observables, because in this framework using the right covariates is crucial for identification and correct inference. In this context, we provide an easy-to-implement and objective method for covariate selection and post-selection inference on average treatment effects.\footnote{Treatment effects, missing data, measurement error, and data combination models are equivalent under selection on observables. Thus, all our results immediately apply to those contexts. For reviews of these literatures, see \citeasnoun{Tsiatis2006_Book}, \citeasnoun{Heckman-Vytlacil_2007a_Handbook}, \citeasnoun{Imbens-Wooldridge2009_JEL}, and \citeasnoun{Wooldridge2010_book}.} We establish four main results for multivalued treatments effects with arbitrary heterogeneity in observables and heteroskedasticity. First, we show that a doubly-robust estimator is robust to model selection errors. These estimators were initially developed for robustness to parametric misspecification, but are now known to be robust to selection.\footnote{Doubly-robust estimation and its role in program evaluation is discussed by \citeasnoun{Robins-Rotnitzky1995_JASA}, \citeasnoun{vanderLaan-Robins2007_book}, \citeasnoun[with discussion]{Kang-Schafer2007_SS}, \citeasnoun{Tan2010_Bmka}, and references therein.} By taking explicit account of the model selection stage and its inherent selection errors, we derive precise conditions required for any model selector to deliver confidence intervals for average treatment effects that are uniformly valid over a large class of data-generating processes. Second, we show that a simple refitting procedure allows researchers to augment variables chosen according economic theory with data-driven selection to deliver flexible inference that remains uniformly valid. Third, we prove that our estimator is asymptotically linear, and standard conditions imposed in the program evaluation literature, semiparametrically efficient bound. Fourth, we derive new results for multinomial (and binary) logistic regression, the most widely used model for treatment assignment.

Inference following model selection is notoriously difficult. In a sequence of papers, Leeb and P\"otscher \citeyear{Leeb-Potscher2005_ET,Leeb-Potscher2008_ET,Leeb-Potscher2008_JoE,Potscher-Leeb2009_JMA} have shown that inference relying too heavily on model selection can not be made uniformly valid. Loosely speaking, uniform validity of a confidence interval captures the idea that the interval should have the same quality (coverage) for many data-generating processes. This theoretical property is practically important because it implies greater reliability in applications. Our proposed methods for post model selection inference build upon the path-breaking recent work of \citeasnoun{Belloni-Chernozhukov-Hansen2014_REStud}. 

The crucial insight that leads to uniform inference is to change the goal of model selection away from perfect \emph{covariate} selection (the oracle property) and to high-quality approximation of the underlying \emph{functions}. This fundamental shift in focus allows us to circumvent, without contradicting, the impossibility results of Leeb and P\"otscher. Valid post-selection inference has attracted considerable attention during the preparation of this paper: in contexts and with methods quite different from ours, contributions have been made by \citeasnoun{Belloni-Chernozhukov-Wei2013_logit}, \citeasnoun{Berk-etal2013_AoS}, \citeasnoun{Zhang-Zhang2014_JRSSB}, \citeasnoun{Efron2014_JASA}, \citeasnoun{vandeGeer-etal2014_AoS}, and \citeasnoun{Belloni-etal2014_WP}, among others.

Our approach, based on the doubly-robust estimator, has several key features. The name ``doubly-robust'' reflects that it is robust to misspecification of either the treatment equation (propensity score) or the outcome equation, a property obtained by combining inverse probability weighting and regression imputation. First, we show that this robustness extends to model selection, enabling us to allow for selection errors in both equations without impacting inference. Second, we capture arbitrary treatment effect heterogeneity (dependence of the effect on an individual's observed characteristics), which is crucial in empirical work. With such heterogeneity, the average treatment effect and the treatment on the treated differ, and hence we present results for both. Third, the doubly-robust estimator also stems from the semiparametric efficient moment conditions, and hence we obtain the semiparametric efficiency bound, even under heteroskedasticity, under standard additional conditions. Thus, \possessivecite{Potscher2009_Sankhya} result that sparse estimators have large confidence sets is also circumvented. Taking all these features together enables us to obtain uniform inference over such a large class of treatment effects models.

In recent independent work, \citeasnoun{Belloni-Chernozhukov-Hansen2014_REStud}, propose a similar approach. Their main focus is inference on the linear part of a partially linear model, which motivates an estimator quite different from ours, but it will recover the average treatment effect in the special case of a binary treatment where the effect is constant across observables. However, their Section 5, developed independently from our work, considers heterogeneous effects and proposes an estimator based on the efficient influence function, similar to ours. There are two broad differences. First, we allow for multivalued treatments, which offers a larger set of estimands and can thus enhance the understanding of program impacts.\footnote{Discussion and applications may be found in, for example \citeasnoun{Imbens2000_Bmka}, \citeasnoun{Lechner2001_chapter}, \citeasnoun{Imai-vanDyk2004_JASA}, \citeasnoun{Abadie2005_REStud}, \citeasnoun{Cattaneo2010_JoE}, and \citeasnoun{Cattaneo-Farrell2011_chapter}.} In this context we propose a group lasso based approach that naturally exploits the already-present structure of treatment effects data to improve model selection by pooling information across treatment levels. This is particularly natural in the multivalued case, but even in the binary case there is still a grouped structure in the outcome regressions, though not in treatment assignment (i.e., in propensity score estimation). Second, although in both cases the doubly-robust estimator is used for average treatment effects\footnote{They use different asymptotic variance estimators, and for treatment effects on the treated they do not exploit the simplification discussed in Remark \ref{remark-tot}.} (following a quite different model selection step), we show that this estimator has two benefits: (i) it may require weaker conditions on the first stage (see Assumption \ref{first stage}); and (ii) it does not require using variables selected for the treatment equation in the outcome model, and vice versa (``post double selection''), and indeed, doing may require additional assumptions (see Assumption \ref{ATE union}).

Our analysis is conducted under selection on observables, which has a long tradition and remains quite popular in empirical economics.\footnote{For other approaches and reviews of the literature, see, e.g., \citeasnoun{Holland1986_JASA}, \citeasnoun{Hahn1998_Ecma}, \citeasnoun{Horowitz-Manski2000_JASA}, Chen, Hong, and Tarozzi \citeyear{Chen-Hong-Tarozzi2004_WP,Chen-Hong-Tarozzi2008_AoS}, \citeasnoun{Bang-Robins2005_Biometrics}, \citeasnoun{Abadie-Imbens2006_Ecma}, \citeasnoun{Wooldridge2007_JoE}, and references therein.} Covariates play three crucial roles in this framework. First, using more observed covariates as proxies, and more flexibly, may help account for unobserved confounding and hence increase the plausibility of unconfoundedness. Second, some observed variables may not be part of the causal mechanism under study, and should be excluded. Third, the efficient conditioning set are those variables that drive the outcome, not necessarily those important for treatment assignment. This reasoning mandates contradicting goals for practitioners: a large, rich set of controls on the one hand, and parsimony on the other. Our approach is a formal, theory-driven attempt to reconcile this contradiction.

A special feature of our analysis is that we match the empirical realities of large data sets by considering selection from amongst (possibly) more covariates than observations, so-called \emph{high-dimensional} data. The goal of variable selection is to find a small model that is nonetheless sufficiently flexible to capture unknown features of the data-generating process required for inference. If a small model can perfectly capture the unknown feature it is said to be \emph{exactly sparse}. More realistic is \emph{approximate sparsity}, when the bias from using a small model is well-controlled, but nonzero. Sparsity is a natural framework for thinking about model selection. Indeed, any time only a few of the available variables are used, a sparsity assumption has effectively been made. It is common empirical practice to report results from several small models, but for these results to be valid one must assume these specifications give high-quality, sparse representations of the unknown features. The alternative we provide involves selecting a sparse, yet flexible, model from among a large set of variables. Results may then be compared with more traditional methods.

With the aim of mimicking common empirical practice we estimate the propensity score with multinomial logistic regression, coupled with group lasso selection \cite{Yuan-Lin2006_JRSSB}. Our results are stated in the language of treatment effects, but apply to general data structures and are of independent interest in the high-dimensional literature.\footnote{Our techniques build on prior studies, in particular \citeasnoun{Bickel-Ritov-Tsybakov2009_AoS}, \citeasnoun{Lounici-etal2011_AoS}, \citeasnoun{Obozinski-Wainwright-Jordan2011_AoS}, \citeasnoun{Belloni-Chernozhukov2011_AoS}, \citeasnoun{BCCH2012_Ecma}, \citeasnoun{Belloni-Chernozhukov2013_Bern}, and \citeasnoun{Belloni-etal2014_WP}.} Much of the literature has focused on linear models (see \citeasnoun{Buhlmann-vandeGeer2011_book} for a survey), while prior studies of nonlinear models often assume exact sparsity or present limited results.\footnote{Examples include \citeasnoun{vandeGeer2008_AoS} and \citeasnoun{Negahban-etal2012_StatSci}, whose bounds do not imply our results. \citeasnoun{Bach2010_EJS} only gives an error bound on coefficients in exactly sparse logistic regression, which can not yield our results; and does not consider prediction error or post-selection estimation. In independent work, \cite{Kwemou2012_logit} and \citeasnoun{Belloni-Chernozhukov-Wei2013_logit} also apply \possessivecite{Bach2010_EJS} tools, but are focused on a different goals. \citeasnoun{Vincent-Hansen2014_CSDA} apply the group lasso to multinomial logistic regression, but do not derive any theoretical results.} Furthermore, these studies often use high-level conditions that can be hard to verify. In contrast, we obtain sharp results for logistic regression under the same simple and intuitive conditions used for linear modeling by exploiting mathematical techniques of self-concordant functions put forth by \citeasnoun{Bach2010_EJS}. We also provide extensions to prior work on linear models needed to apply them in treatment effect estimation.

Finally, we offer numerical evidence on the finite sample performance of our procedure. In a small simulation study we find that our procedure delivers very accurate coverage of confidence intervals even for models where covariate selection is difficult, either because of a low signal-to-noise ratio or lack of sparsity, thus highlighting the uniform validity of inference. We also apply our method to the widely-used National Supported Work Demonstration data \cite{LaLonde1986_AER} and find very accurate estimates and tight confidence intervals (see Table \ref{table-lalonde}).

The paper proceeds as follows. Section \ref{sec-overview} gives short, self-contained overview. Section \ref{sec-notation} collects notation. Section \ref{sec-model} describes the treatment effect models. Sparse models are discussed in Section \ref{sec-sparse}, which shows how several commonly used models fit in this framework. Section \ref{sec-ate} presents our estimation method and complete results on treatment effect inference. Theoretical results for the group lasso are in Section \ref{sec-grplasso}. Section \ref{sec-data} presents the numerical evidence and Section \ref{sec-conclusion} concludes. The main proofs are presented in the Appendix, while the remainder are available in a supplement.

%%%%%%%%%%%%%%%%%%%%%%%%%%%%
%%%%%%%%%%%%%%%%%%%%%%%%%%%%
\section{Overview of Results and Notation}
	\label{sec-overview}

Here we give an overview of the paper, including treatment effect inference (Section \ref{sec-overview ate}), our new results for the group lasso (Section \ref{sec-overview grplasso}), and notation used throughout (Section \ref{sec-notation}).

%%%%%%%%%%%%%%%%%%%%%%%%%%%%
\subsection{Treatment Effects and Results on Post-Selection Inference}
	\label{sec-overview ate}

We consider a multivalued treatment, with status indicated by $\D \in \{0, 1, \ldots, \T\}$. Interest lies in mean effects of the treatment on a scalar outcome $Y$. Let $\{Y(\t)\}_{\t = 0}^\T$ be the (latent) potential outcomes: $Y(\t)$ is the outcome a unit would have under $\D = \t$ and is only observed for units with $\D = \t$; that is, $Y = \sum_{\t = 0}^\T \indicator\{\D = \t\} Y(\t)$. Many interesting parameters combine means of potential outcomes, and having multivalued treatments allows for a wider range of estimands. Define the mean of one potential outcome as $\mut = \E[Y(\t)]$. To fix ideas, $\mu_1 - \mu_0$ is the average treatment effect in the binary case ($\D \in \{0,1\}$). Sections \ref{sec-model} and \ref{sec-ate} consider more general average effects, including effects on treated groups. For simplicity, in this section we focus on a single $\mut$.

We use the selection on observables framework to identify $\mut$. For a vector of covariates $X$, define the generalized propensity score and conditional outcome regressions as
\[p_\t(x) = \P[ \D = \t \vert X = x] 	\qquad \text{ and } \qquad 	\mut(x) = \E[ Y \vert \D = \t, X = x].\]
For identification it is sufficient to assume that $\E[Y(\t) \vert \D, X] = \E[Y(\t) \vert X]$ (mean independence) and $p_\t(X)$ is bounded away from zero (overlap) for all treatment levels. Broadly, these two assumptions imply that units from one treatment group are good proxies for other treatments and that there are always such proxies available (see Section \ref{sec-model}).

For an i.i.d. sample $\{(y_i, \d, x_i')\}_{i=1}^n$ and model-selection-based estimators $\hat{p}_\t(x_i)$ and $\muthat(x_i)$, we estimate $\mut$ with
	\[\muthat  = \frac{1}{n} \sumi \left\{ \frac{ \indicator\{\d = \t\} (y_i - \muthat(x_i) ) }{ \hat{p}_\t(x_i) } + \muthat(x_i) \right\}.\]
This doubly-robust estimator combines regression imputation and inverse probability weighting, and remains consistent if either the model $p_\t(x)$ or $\mut(x)$ is misspecified. Following widespread empirical practice, we estimate $\hat{p}_\t(x_i)$ with multinomial logistic regression and $\muthat(x_i)$ linearly (see Section \ref{sec-grplasso}). The choice of covariates in $\hat{p}_\t(x_i)$ and $\muthat(x_i)$ impacts consistency, efficiency, and finite sample performance. Covariate selection based on ad hoc, iterative searches is common in empirical work, but is not formal, objective, or replicable. Balancing tests are also common, but have the additional drawback of assuming the same covariates are important for outcomes and treatment assignment, and more generally do not weight the covariates by their importance for bias.

On the other hand, our proposed procedure gives practitioners an easy to implement, fully objective tool to perform data-driven covariate selection and treatment effect inference, with replicable results.\footnote{For the final step, the doubly-robust estimator is available in STATA and the package of \citeasnoun{Cattaneo-Drukker-Holland2013_stata}. The covariate selection stage is easily implemented in {\sf R}.} Importantly, we do not preclude the addition of variables known to be important from economic theory or prior knowledge. Our procedure is intended to supplement these variables with a flexible set of controls, guarding against misspecification or overfitting.

The following theorem is an example of the more general results presented in Section \ref{sec-ate theory}, wherein we also define $V_\t$ and $\hat{V}_\t$.
\begin{theorem}
	\label{thm-ate overview}
	Consider a sequence $\{P_n\}$ of data-generating processes that obey, for each $n$, Assumptions \ref{ignorability} and \ref{dgp asmpts} below. If the first stage obeys
	\begin{enumerate}[label=(\roman{*})]
		\item $\sumi (\hat{p}_\t(x_i)  -  p_\t(x_i))^2 /n = o_{P_n}(1)$ and $\sumi (\muthat(x_i) - \mut(x_i))^2 /n = o_{P_n}(1)$,

		\item $\bigl[ \sumi \indicator\{\d = \t\}(\hat{p}_\t(x_i) - p_\t(x_i))^2/n \bigr]^{1/2} \bigl[\sumi \indicator\{\d = \t\}(\muthat(x_i) - \mut(x_i))^2/n\bigr]^{1/2} = o_{P_n}(n^{-1/2})$, and 
		
		\item $\bigl[ \sumi  (\muthat(x_i) - \mut(x_i)) (1 - d_i^t/ p_\t(x_i))/n \bigr]    = o_{P_n}(n^{-1/2})$,
	\end{enumerate}
	then $\sqrt{n} ( \muthat - \mut ) \to_d N(0,V_\t)$ and $\hat{V}_\t/V_\t \to_{P_n} 1$. For each $n$, let $\bm{P}_n$ be the set of data-generating processes obeying Assumptions \ref{ignorability} and \ref{dgp asmpts} and (i) and (ii) above. Then for $c_\alpha = \Phi^{-1}(1 - \alpha/2)$
\[ \sup_{P \in \bm{P}_n}  \left|  \P_P \left[  \mut  \in \left\{ \muthat \pm c_\alpha \sqrt{ \hat{V}_\t / n}\right\} \right] - (1 - \alpha) \right| \to 0.\]
\end{theorem}

This result establishes the uniform validity of an asymptotic confidence interval for $\mut$, overcoming all the post model selection inference challenges: robustness to model selection errors, selecting a model that is small but flexible enough to capture the features of the underlying data generating process, and still retaining efficiency under standard conditions (see Section \ref{sec-efficiency}). Intuitively, this is similar to (but distinct from) overcoming pretesting bias in other contexts. Also, although our discussion is in terms of covariate selection in high-dimensional, sparse models, the inference result is generic for any first stage estimator.

The two conditions placed on the first stage are analogous to the commonly-used, high-level requirement in semiparametrics that first stage components converge faster than $n^{-1/4}$. However exploiting features of the doubly-robust estimator yields weaker conditions. The first is a mild consistency requirement. The second requires a rate on the product of errors and is thus easier to satisfy if one function is easier to estimate, e.g.\ more smooth or more sparse. In model selection, the rates for the first stage depend on the sample size, the number of covariates considered, and the sparsity level. Importantly, the rate will depend on the total number of covariates only logarithmically, allowing for a large number. We propose to use the group lasso and prove that these estimators satisfy (i) and (ii).

%%%%%%%%%%%%%%%%%%%%%%%%%%%%
\subsection{Model Selection Stage}
	\label{sec-overview grplasso}

We propose refitting following group lasso selection, and show that it meets all requirements on the model selector. The group lasso is well-suited to program evaluation applications because covariates are penalized according to their overall contribution in all treatment groups. This has two consequences. First, information from all treatments is pooled when doing selection, and hence a weaker signal may be extracted, which improves the selection properties. Second, the selected variables are common to all treatment levels. From a practical point of view this is desirable, as interest rarely lies in a single $\mut$, but rather a collection, and substantial commonality is expected in the variables important for different treatment levels.

We consider high-dimensional, sparse models for $p_\t(x)$ and $\mut(x)$. These are defined by a $p$-dimensional vector $X^*$ based on the original variables $X$. The $X^*$ may consist of any combination of the original variables, interactions, flexible parametric transformations, and/or nonparametric series terms (such as splines or polynomials). A model is approximately sparse if there are $s < n$ of these terms that yield a good approximation ($s\to \infty$ is allowed). To build intuition, suppose that $\mut(x)$ obeys a $p$-dimensional linear model. Then the sparsity assumption is that there is an $s$-dimensional submodel with sufficiently small specification bias. In the nonparametric case, sparsity is weaker than (but analogous to) the familiar assumption that a small set of basis functions can approximate the unknown objects well. In practice researchers employ a hybrid of these approaches, which is covered by our results. Section \ref{sec-sparse} gives more detail and examples.

We form $\hat{p}_\t(x)$ and $\muthat(x)$ in two steps (complete details in Section \ref{sec-grplasso}). First, the group lasso is applied separately to multinomial logistic and least squares regression to select covariates from $X^*$. We then estimate $p_\t(x)$ and $\mut(x)$ by refitting unpenalized models using the selected variables, possibly augmented with controls suggested by prior work or economic theory. It is not desirable for a model selector to discard theory and prior work, and our procedure explicitly avoids this. We also allow for using logistic-selected variables in the linear model refitting and vice versa, but this is not necessary for uniformity nor efficiency.

Our main results give precise bounds for the number of covariates selected and the estimation error, both for the penalized and unpenalized estimates. Section \ref{sec-grplasso} results gives nonasymptotic bounds, with exact constants. Such results are complex and so we give the following intuitive, asymptotic result (The notation $O_{P_n}$ is defined in Section \ref{sec-notation}).
\begin{corollary}
	\label{thm-grplasso overview}
	Suppose the biases from the best $s_d$- and $s_y$-term approximations to $p_\t(x)$ and $\mut(x)$ are order $\sqrt{s_d / n}$ and $\sqrt{s_y /n}$, respectively. Then under the assumptions in Section \ref{sec-grplasso}, and $\delta > 0$ described therein, with high probability we have:
		\begin{enumerate}
			\item $\sumi (\hat{p}_\t(x_i) - p_\t(x_i))^2/n = O_{P_n}\left(   n^{-1} s_d\log(p \vee n)^{3/2 + \delta}  \right)$ and
			\item $\sumi (\muthat(x_i) - \mut(x_i))^2/n = O_{P_n} \left(  n^{-1} s_y \log(p \vee n)^{3/2 + \delta}  \right)$.
		\end{enumerate}
\end{corollary}
These two results for our proposed group lasso estimators can be directly used to verify the high-level conditions in Theorem \ref{thm-ate overview} above. Specifically, if $s_d s_y \log(p)^{3 + 2\delta} = o(n)$, conditions (i) and (ii) of Theorem \ref{thm-ate overview} are met (requiring $s^2 = o(n)$, up to $\log$ factors, as found in other results in the literature). Further, it is clear how the doubly-robust estimator can help: if one function is more smooth or more sparse, $s_d$ or $s_y$ will be lower, easing the restriction. Section \ref{sec-grplasso results} gives further results: showing that the number of variables selected is the same order as the sparsity level, and provides bounds on the logistic and linear coefficients directly. Both these results are important for certain steps in treatment effect estimation that aren't reflected in the simple statement of Theorem \ref{thm-ate overview}. These results appear to be entirely new for the multinomial logistic regression, for any version of the lasso. From a practical point of view, these results provide formal justification for using multinomial logistic regression, coupled with group lasso selection and post-selection refitting.

%%%%%%%%%%%%%%%%%%%%%%%%%%%%
\subsection{Notation}
	\label{sec-notation}

We collect here notation to be used for the rest of the paper. The data generating process (DGP) is denoted by $P_n$ and is defined by the joint law of the random variables $(Y,\D, X')'$. For a given $n$, $\{(y_i, \d, x_i')'\}_{i=1}^n$ constitute draws from $P_n$. The DGP may vary with $n$, along with features such as parameters, distributions, and so forth, as discussed in Section \ref{sec-vary with n}. This is generally suppressed for clarity. We adopt the following conventions.
\begin{description}

	\item[Treatments.] Define the treatment sets $\NTbar = \{0, 1, 2, \ldots, \T\}$ and $\NT = \{1, 2, \ldots, \T\}$. No order is assumed in the treatments. For each unit $i$, $\d$ indicates treatment assignment, and define $\dt = \indicator\{\d =\t\}$. Let $n_\t = \sumi \dt$ be the number of individuals with treatment $\t$ and define $\underline{n} = \minTbar n_\t$ and $\overline{n} = \maxTbar n_\t$. Further define $\Tbar  = \T + 1$. 

	\item[Vectors.] Define $\Np = \{1, 2, \ldots, p\}$. For a doubly-indexed collection of scalars $\{\danytj : \t \in \NTbar, j \in \Np\}$, define $\danyj  \in \mathbb{R}^{\Tbar}$ as the vector that collects over all $\t$ for fixed $j$; $\delta_{\t,\newdot} \in \mathbb{R}^p$ collects over $j \in \Np$ for fixed $\t$; and $\dany \in \mathbb{R}^{p \times \Tbar}$ the concatenation of all $\delta_{\t,\newdot}$. For simplicity, we write $\delta_\t$ for $\delta_{\t,\newdot}$. When considering the multinomial logistic model, $\t$ will vary only over $\NT$ but the notation will be maintained. For a set $S \subset \Np$, let $\delta_{\t,S} \in \mathbb{R}^{\text{card}(S)}$ be the vector of $\{\delta_{\t,j} : j \in S\}$ for fixed $\t$ and similarly let $\delta_{\newdot,S} \in \mathbb{R}^{|S| \times \Tbar} = \{\delta_{\t,j} : \t \in \NTbar, j \in S\}$. 
		
	\item[Norms.] Single bars will be either absolute value or cardinality of a set, and will be clear from the context. For a vector $v$, let $\| v \|_1$ and $\| v \|_2$ denote the $\ell_1$ and $\ell_2$ norms, respectively. For the group lasso, define the mixed $\ell_2$/$\ell_1$ norm as $\vertiii{\dany }  = \sumj \|  \danyj \|_2$. It will always be the case that the (``outer'') $\ell_1$ norm is over the covariates and the (``inner'') $\ell_2$ norm is over the treatments (in our application). When discussing the multinomial logistic model, treatments will be restricted to $\NT$ with no change in notation.

	\item[Data-Generating Processes.] The set of all $P_n$ considered is $\bm{P}_n$. For sequences, $\{P_n\} = \{P_n : n \geq 1, P_n \in \bm{P}_n\}$. Expectations and probabilities are taken against $P_n$, though notationally suppressed. For asymptotic arguments dependence on $n$ is explicit, so that $O_{P_n}(\cdot)$ and $o_{P_n}(\cdot)$ have their usual meaning with the understanding that the measure $P_n$ is used for each $n$.

\end{description}

For a set of scalars $\{m_\t\}_{\t = 1}^{\T}$, let $\hat{p}_\t( \{m_\t\}_{\NT} ) = \exp(m_\t) [1 + \sum_{\t \in \NT} \exp( m_\t) ]^{-1}$ denote the multinomial logit function. Empirical expectation will be denoted $\En[w_i] = \sumi w_i/n$ and $\Ent[w_i] = \sumt w_i / n_\t = \sumi \dt w_i / n_\t$.

%%%%%%%%%%%%%%%%%%%%%%%%%%%%
%%%%%%%%%%%%%%%%%%%%%%%%%%%%
\section{Treatment Effects Model}
	\label{sec-model}

In this section we formally define the treatment effects model and the parameters of interest. Recall that $\D \in \{0, 1, \ldots, \T\}$ indicates treatment status, $\{Y(\t)\}_{\t \in \NTbar}$ are the (latent) potential outcomes, and $Y(\t)$ is only observed for units with $\D = \t$; that is, $Y = \sum_{\t \in \NTbar} Y(\t)$. The building blocks of many general estimands are the averages
\begin{equation}
	\label{eqn-ate}
	\mut = \E[Y(\t)],	\quad 	\t \in \NTbar,		\qquad \text{ and } \qquad		\mutt = \E[Y(\t) \vert D = \t'], 	\quad 	\t, \t' \in \NTbar \times \NTbar.
\end{equation}
In the binary case, the average treatment effect is $\mu_1 - \mu_0$ and the treatment on the treated is $\mu_{1,1} - \mu_{0,1}$. A multivalued treatment allows for a large range of interesting estimands. To fix ideas, we keep as running examples two leading cases. First, the so-called dose-response function: the $(\T + 1)$-vector $\drf = (\mu_0, \mu_1, \ldots, \mu_\T)'$. Second, define $\tot$ as the $\T$-vector with element $\t$ given by $\mu_{\t,\t} - \mu_{0,\t}$. This gives the effect of each treatment relative to the baseline $\t=0$, only for those who received that treatment. These are by no means the only interesting estimands constructed from $\mut$ and $\mutt$; many others are given by \citeasnoun{Lechner2001_chapter}, \citeasnoun{Heckman-Vytlacil_2007a_Handbook}, and others.

The following two conditions are sufficient to identify $\mut$ and $\mutt$. 
\begin{assumption}[Identification]
	\label{ignorability} \
	For all $\t \in \NTbar$ and almost surely $X$, $P_n$ obeys:
	\begin{enumerate}[label=(\alph{*}), ref=\ref{ignorability}(\alph{*})]
	
		\item (Mean independence) $\E[Y(\t) \vert D, X=x] = \E[Y(\t) \vert X=x]$, and			\label{mean inde} 
		\item (Overlap) $\P[\D = \t \vert X = x]) \geq p_{\min} > 0$ for all $\t \in \NTbar$. 			\label{overlap}

	\end{enumerate}
\end{assumption}
This assumption is a form of ``ignorability'' coined by \citeasnoun{Rosenbaum-Rubin1983_Bmka}. This model allows arbitrary treatment effect heterogeneity in observables, but not unobservables. This assumption is standard in the program evaluation literature, and its plausibility has been discussed at length, so we omit a general discussion (see, e.g., \citeasnoun{Imbens2004_REStat}, \citeasnoun[Chapter 21]{Wooldridge2010_book}, and references therein). However, in the context of model selection, three remarks are warranted.

First, in place of \ref{mean inde}, it is more common to instead assume full conditional independence: $Y \inde \D \vert X$. However, as observed by \citeasnoun{Heckman-Ichimura-Todd1997_REStud}, the weaker mean independence is sufficient. For our purposes, the ``gap'' between the two assumptions is important. Suppose full independence holds only conditional on a set of variables strictly larger than the variables entering the mean functions (e.g.\ the excess variables affect higher moments). In this case, because mean independence is still sufficient, we need not aim to select the larger set. Full independence is important for the efficiency discussed in Section \ref{sec-efficiency}. 

Second, the covariates may, in general, include instruments for treatment status, but they are not known as such. This is standard, but left implicit, in discussions of ignorability. If instruments are present, and selected for estimation, efficiency suffers but unbiasedness is not harmed. Efficiency bounds in this context typically (implicitly) assume there are no instruments in $X$. Assumption \ref{overlap} rules out perfect predictors. Section \ref{sec-efficiency} offers further discussion.

Finally, the main drawback of Assumption \ref{mean inde} is that it does not give identification of average effects on transformations of $Y(\t)$. However, we are expressly interested in model selection on the mean function of the level of $Y(\t)$, and hence Assumption \ref{mean inde} is more natural. To operationalize model selection, structure must be placed on $\E[Y(\t) \vert X = x]$, and hence functional form conditions tied to mean independence are not limiting per se. If the parameter of interest is changed, say to $\E[\log(Y(\t))]$, and a sparsity assumption is made for $\E[\log(Y(\t)) \vert X = x]$, then our method applies.

Assumption \ref{ignorability} yields identification of $\mut$ and $\mutt$ using either inverse weighting or regression, and double robustness follows from combining the two strategies. Recall the notation $p_\t(x) = \P[\D = \t \vert X = x]$ and $\mut(x) = \E[Y \vert \D = \t, X = x]$. Applying Assumption \ref{ignorability} we find that 
\begin{equation}
	\label{eqn-ate moments}
	\E\bigl[ \psi_\t\bigl(Y, \D, \mut(X), p_\t(X), \mut \bigr) \bigr] = \E\left[  \frac{\indicator\{\D = \t\} Y}{p_\t(X)} + \mut(X) - \frac{\indicator\{\D = \t\} \mut(X)}{p_\t(X)}  -  \mut \right]  =  0
\end{equation}
and 
\begin{multline}
	\label{eqn-tot moments}
	\E\bigl[ \psi_{\t,\t'} \bigl(Y, \D, \mut(X), p_\t(X), p_{\t'}(X), \mutt \bigr) \bigr] 		\\		= \E\left[ \frac{\indicator\{\D = \t'\}  \mut(X)}{p_{\t'}} +  \frac{ p_{\t'}(X)}{p_{\t'}} \frac{ \indicator\{\D = \t\} (Y - \mut(X))} { p_\t(X) }  -  \mutt\right]  =  0,
\end{multline}
where $p_\t = \P[\D = \t]$. The moment condition \eqref{eqn-ate moments} holds if either $p_\t(x)$ or $\mut(x)$ is misspecified. For $\mutt$, if $\mut(x)$ is misspecified, both $p_\t(X)$ and $p_{\t'}(X)$ must be correctly specified, while if $\mut(x)$ is correct, both propensity scores may be misspecified. It is important to note that the forms of $\psi_\t(\cdot)$ and $\psi_{\t, \t'}(\cdot)$ are fixed, so the function itself does not depend on the sample size even if its arguments do. Our estimator is a plug-in version of this moment condition.

\begin{remark}[Simplifications for $\mu_{\t,\t}$]
	\label{remark-tot}
Identification of $\mu_{\t,\t}$ does not require Assumption \ref{ignorability}. $Y(\t)$ is fully observed for the sub-population of interest and so a simple average will deliver $\mu_{\t,\t} = \E[ \indicator\{\D = \t\} Y] / p_\t$. Note that \eqref{eqn-tot moments} reduces to this when $\t = \t'$. For $\tot$ this means we must only estimate the function $\mut(x_i)$ for $\t = 0$. Intuitively, we must use comparison group observations to proxy for treated units, but not the other way around. Thus, for certain parameters of interest, Assumption \ref{ignorability} can be weakened to hold only for the comparison group. However, we cover generic estimands, without necessarily specifying a comparison group, and so we maintain Assumption \ref{ignorability} for simplicity, rather than keeping track of hosts of special cases.
\end{remark}

\begin{remark}[Efficient Influence Functions]
	\label{remark-eif}
	The efficient influence functions in this model are exactly $\psi_\t(\cdot)$ and $\psi_{\t, \t'}(\cdot)$, and so our estimators have the interpretation of being plug-in versions of these, and indeed, will be asymptotically linear with this influence function (see Section \ref{sec-efficiency}).
\end{remark}

%%%%%%%%%%%%%%%%%%%%%%%%%%%%
%%%%%%%%%%%%%%%%%%%%%%%%%%%%
\section{Approximately Sparse Models}
	\label{sec-sparse}

We now formalize approximate sparsity. Let $X_Y^*$ and $X_\D^*$ be $p$-dimensional transformations of the covariates $X$, with $p>n$ allowed. These transformations are specific to the outcome and treatment models, but may overlap. They do not vary with $\t$, nor depend on the DGP. Some examples are given below in Section \ref{sec-examples}. For the multinomial logistic model it is convenient to work with the log-odds ratio. We take $p_0(x) = 1 - \sum_{\t \in \NT} p_\t(x)$ and write
\begin{equation}
	\label{eqn-propensity}
	\log \left( \frac{p_\t(x) }{ p_0(x) }\right)  = {x_\D^*}' \gtruet + B_\t^\D,		\qquad \t \in \NT.
\end{equation}
Similarly, write the outcome regressions as
\begin{equation}
	\label{eqn-outcome}
	\mut(x)  = {x_Y^*}' \btruet + B_\t^Y,		\qquad t \in \NTbar,
\end{equation}
The terms $B_\t^\D = B_\t^\D(x)$ and $B_\t^Y = B_\t^Y(x)$ are bias terms arising from the parametric specification. As discussed below, these encompass the usual nonparametric bias as well. Approximate sparsity requires that only a small number of the $X^*$ are needed to make the bias small. Define $ S_*^\D = \bigcup_{\NT} \supp(\gtruet)$ and $ S_*^Y =\bigcup_{\NTbar} \supp(\btruet) $, so that these sets capture all variables important for treatment and outcomes, respectively. We assume that there are some $s_d < n$ and $s_y < n$, such that for $| S_*^\D | = s_d$ and $|S_*^Y | = s_y $, the biases $B_\t^\D$ and $B_\t^Y$ are sufficiently small. This is made precise by defining the bounds:
\begin{equation}
	\label{eqn-bias}
	\En[(\pttrue - p_\t(x_i))^2]^{1/2} \leq \bias^d 		\quad  \text{ and } \quad 		 \En[B_\t^Y(x_i)^2]^{1/2} \vee \Ent[B_\t^Y(x_i)^2]^{1/2} \leq \bias^y.
\end{equation}
Note that the former bias bound is placed directly on the propensity score because it is the ultimate object of interest, rather than on the linearization of the log-odds.

While a great deal of overlap is expected, in practice it is likely that a few covariates will be more or less important for different treatments, and so we do not require that the supports of $\gtruet, \t \in \NT$ or $\btruet, \t \in \NTbar$ are constant over $\t$, nor that $S_*^\D$ overlaps with $S_*^Y$. Instead, it may be better to think of $\N_p \setminus S_*^\D$ and $\N_p \setminus S_*^Y$ as the ``common nonsupports'' of the treatment and outcome equations. When it is clear from the context we will abbreviate both $X_\D^*$ and $X_Y^*$ by $X^*$ (and their realizations by $x_i^*$) and refer to them generically as ``covariates'', and further write $s$ for either $s_d$ or $s_y$. We assume $\En[({x_{i,j}^*})^2] = 1$ without loss of generality (see Remark \ref{remark-penalty}).

%%%%%%%%%%%%%%%%%%%%%%%%%%%%
\subsection{Parametric and Nonparametric Examples}
	\label{sec-examples}

To concretize the sparse model idea, we now discuss how several models commonly used in practice fit into this framework. These include parametric and nonparametric models for $p_\t(x)$ and $\mut(x)$, and hybrids of these. A common theme to all examples will be comparison to the \emph{oracle} model: the model that knows the true support in advance. Our uniform inference results include all these examples as special cases because, loosely speaking, we obtain uniformity over DGPs where $p_\t(x)$ and $\mut(x)$ have sparse representations. We aim for an accessible discussion of each model, and defer technicalities to the literature \cite{Raskutti-Wainwright-Yu2010_JMLR,Rudelson-Zhou2013_IEEE,Belloni-Chernozhukov-Hansen2014_REStud}.

\begin{example}[Oracle parametric model]
Assume models \eqref{eqn-propensity} and \eqref{eqn-outcome} hold with $B_\t^\D = B_\t^Y = 0$ and $X_\D^* = X_Y^* = X$. Let $p = s = \dim(X)$. All covariates are used in all modeling. If dimension is fixed this is the textbook parametric model, see for example \citeasnoun{Wooldridge2010_book}. Alternatively, the dimension can be diverging, but more slowly than $n$. We are not aware of any work which covers this case explicitly, though for the first stage, \citeasnoun{He-Shao2000_JMA} cover linear and logistic regression, and their results easily extend to multinomial logistic models.

The vast majority of treatment effect studies adopt this model (with dimension fixed), taking the set of covariates as given. In our framework, this is equivalent to the researcher having  prior knowledge of which covariates are important and which are not. Such knowledge no doubt plays an important role, but it cannot cover all situations or all variables. Furthermore, as more data become available, the researcher does not increase the complexity of their model.
\end{example}

\begin{example}[Exactly sparse parametric model]
	\label{eg-parametric}
Retain the exact parametric structure of the prior example, but let $\dim(X) = p$ be possibly larger than $n$, and assume that $S_*^Y$ and $S_*^\D$ are unknown sets of cardinality less than $n$. Model selection must be performed. Often, researchers (implicitly) rely on the \emph{oracle property}, that $S_*^Y$ and $S_*^\D$ can be found with probability approaching one, and conduct inference conditioning on this event. This approach cannot be made uniformly valid and has poor finite sample properties, as shown by Leeb and P\"otscher \citeyear{Leeb-Potscher2005_ET,Leeb-Potscher2008_ET,Leeb-Potscher2008_JoE,Potscher-Leeb2009_JMA}.
\end{example}

\begin{example}[Approximately sparse parametric model]
	\label{eg-approx sparse}
Again suppose a purely parametric model, so that $X_\D^* = X_Y^* = X$ and $\dim(X) = p$, possibly greater than $n$. Suppose that there exist coefficients $\gamma_{\newdot, \newdot}^0$ and $\beta_{\newdot, \newdot}^0$ such that $\log[p_\t(x) / p_0(x) ]  = {x_\D^*}' \gamma_\t^0$ and $\mut(x) = x'\beta_\t^0$ exactly, but instead of any coefficients being precisely zero, suppose they may be ordered such that $|\gamma_{\t,j}^0| \propto j^{-\alpha_\gamma}$ and $|\beta_{\t,j}^0| \propto j^{-\alpha_\beta}$, with $\alpha_\gamma$ and $\alpha_\gamma$ at least one. Then, there exist $s_d$ and $s_y$ that are $o(n)$ such that Equations \eqref{eqn-propensity} and \eqref{eqn-outcome}, and other conditions needed, are satisfied for $\gtruetj = \gamma_{\t,j}^0$  for $j \leq s_d$ and  $\btruetj = \beta_{\t,j}^0$ for $j \leq s_y$ and the rest truncated to zero. That is $S_*^\D$ and $S_*^Y$ collect the largest coefficients and $B_\t^\D = \sum_{\N_p \setminus S_*^\D} x_j \gamma_{\t,j}^0$, and similarly for $B_\t^Y$.
\end{example}

\begin{example}[Semiparametric model]
	\label{eg-semiparametric}
Assume $p_\t(x)$ and $\mut(x)$ are unknown functions that can be well-approximated by a linear combination of $s_d$ and $s_y$ basis functions, respectively (e.g. are sufficiently smooth). In \eqref{eqn-propensity} and \eqref{eqn-outcome}, $\gtrue$ and $\btrue$ are the coefficients of these approximations, while $B_\t^D$ and $B_\t^Y$ are the usual nonparametric biases. $X_\D^* = R_\D(X)$ and $X_Y^* = R_Y(X)$ are series terms used in the approximation. Standard semiparametric analyses, such as \citeasnoun{Hirano-Imbens-Ridder2003_Ecma}, \citeasnoun{Imbens-Newey-Ridder2007_MSE}, or \citeasnoun{Cattaneo2010_JoE}, can be viewed in this context as oracle models that know in advance which terms yield the best approximation, typically assumed to be the first terms. Instead, we only require that some $s_d$ (or $s_y$) of a set of $p$ series terms give good approximations. This allows for greater flexibility in applications, where there is no knowledge of which series terms to use, and the researcher may want to mix terms from different bases. 
\end{example}

\begin{example}[Mixed parametric and semiparametric model]
	\label{eg-mixed}
	Partition $X = (X_1, X_2)$. Suppose that the true log-odds function satisfies $\log [p_\t(x) / p_0(x) ] = x_1'\gamma_{\t}^1 + h_\t(x_2) + B_\t^1(x)$, where $B_\t^1(x)$ is a specification bias and $h_\t(\cdot)$ is a smooth unknown function. For a set of basis functions $R_\D(x_2)$, there will exist coefficients $\gamma_\t^2$ such that $h_\t(x_2) = R_\D(x_2)'\gamma_\t^2 + B_\t^2(x_2)$ and so
\[ \log \left( \frac{p_\t(x) }{ p_0(x) }\right)  =  {x_\D^*}' \gtruet + B_\t^\D,	\quad 	x_\D^* = (x_1', R_\D(x_2)')',		\quad 	\gtruet = ({\gamma_{\t}^1}', {\gamma_{\t}^2}')',  		\quad \text{ and } \quad 	B_\t^\D=B_\t^1+ B_\t^2.\]
We require that some collection of variables and series terms give a good, sparse approximation, without placing explicit conditions on how many of either. Implicitly, one will restrict the other. For example, if the dimension of the parametric part is large, then we require that $h_\t(\cdot)$ can be more easily approximated. We treat $\mut(x)$ the same. This example is closest to actual practice, where some variables (e.g. dummies) enter in a known way and should not be considered part of a nonparametric object, while other covariates must be considered flexibly.
\end{example}

It is important to note that misspecification of the type guarded against by double robustness can arise in any type of model. In parametric cases, this is most often functional form misspecification. While this type of misspecification does not occur in nonparametrics, others are possible, such as shape restrictions or separability assumptions being incorrect, or omitting relevant variables. None of these errors disappear asymptotically, and all of them are guarded against by use of the doubly-robust estimator.

%%%%%%%%%%%%%%%%%%%%%%%%%%%%
\subsection{Conceptual considerations in $n$-varying DGPs}
	\label{sec-vary with n}

Much of the DGP, including parameters and distributions, is allowed to depend on $n$.  Perhaps the most salient features that do not depend on $n$ are the set of treatments and the functions $\psi_\t$ and $\psi_{\t, \t'}$. It is likely that our results can be extended to accommodate a growing number of treatments, but that is beyond the scope of our study. In the models \eqref{eqn-propensity} and \eqref{eqn-outcome}, $X^*$, $\gtrue$, and $\btrue$ must depend on $n$ by construction. Our results on estimation of these models are nonasymptotic. For treatment effect inference, we use triangular array asymptotics to retain the dependence on $n$ of the DGP. The interpretation of the results does, and should, change depending on what is assumed about the DGP. To illustrate, let us return to Examples \ref{eg-parametric} and \ref{eg-semiparametric}.

First, consider the simple parametric models of Example \ref{eg-parametric}. In this case, $\mut = \E[ \E[Y(\t) \vert X]] = \E[ X']\btruet$, which depends on $n$ by construction, as the dimension is diverging. It may seem unnatural that the parameter to be estimated depends on $n$, as we typically think of ``true'' parameters being features of a (large) fixed study population. However, with a diverging number of covariates, there is no fixed DGP. Indeed, if we estimate $\mut = \mut^{(n_1)}$ based upon $n_1$ observations, and then proceed to gather $n_2$ \emph{more} observations, when we re-estimate our target is now $\mut^{(n_1 + n_2)} \neq \mut^{(n_1)}$. One possible resolution is as follows. First, the parameter of interest is $\mut^{(\infty)} = \E[Y(\t)]$, which is defined without reference to covariates. We can view each successive $n$-dependent $\mut$ as an approximation of $\mut^{(\infty)}$ based upon $p = p_n$ covariates. Note well that in our thought experiment, $p_{n_1} \neq p_{n_1 + n_2}$, and so additional variables should have been collected for all $n_1 + n_2$ samples.

Contrast this with the semiparametric model in Example \ref{eg-semiparametric}. It is common to assume the population DGP is fixed over $n$. The treatment effects may be constructed in terms of the underlying variables, e.g.\ $\mut^{(\infty)} = \E[Y(\t)] = \E[\E[Y(\t) \vert X]]$, with $X^*$ serving only the purpose of aiding in approximating the regression functions. Model selection is performed on series terms, not underlying variables, to estimate the coefficients $\gtrue$ and $\btrue$. If $\mut  = \E[{X_Y^*}'] \btruet + \E[B_\t^Y]$ does not depend on $n$, the bias term, by definition, exactly compensates for the $n$-dependence in $\E[{X_Y^*}'] \btruet$. We emphasize that our inference results allow for general $n$-dependence in the DGP, and interpretation by the econometrician must take careful account of any conceptual assumptions.

%%%%%%%%%%%%%%%%%%%%%%%%%%%%
%%%%%%%%%%%%%%%%%%%%%%%%%%%%
\section{Main Results on Treatment Effect Estimation and Inference}
	\label{sec-ate}

In this section we present results on uniformly valid treatment effect inference. We first present the estimators and conditions required for a generic first stage to yield uniform inference. Although our focus is on model selection and sparsity, our results are more general, showcasing the benefits of doubly robust estimation for any model in Section \ref{sec-sparse} where Assumption \ref{first stage} below (which does not refer to selection or sparsity) can be satisfied.

%%%%%%%%%%%%%%%%%%%%%%%%%%%%
\subsection{Estimation Procedure with a Generic Model Selector}
	\label{sec-estimator}

The moment functions $\psi_\t(\cdot)$ and $\psi_{\t, \t'}(\cdot)$ of Equations \eqref{eqn-ate moments} and \eqref{eqn-tot moments} have fixed and known form, and so for estimators $\hat{p}_\t(x)$ and $\muthat(x)$, we can define 
\begin{equation}
	\label{eqn-ate hat}
	\muthat  = \frac{1}{n} \sumi \left\{ \frac{ \dt (y_i - \muthat(x_i) ) }{ \hat{p}_\t(x_i) } + \muthat(x_i) \right\}
\end{equation}
and 
\begin{equation}
	\label{eqn-tot hat}
		\mutthat  = \frac{1}{n} \sumi \left\{ \frac{ d_i^{\t'}  \muthat(x_i)}{ \hat{p}_{\t'} }   +   \frac{ \hat{p}_{\t'}(x_i) }{\hat{p}_{\t'}} \frac{\dt (y_i - \muthat(x_i)) } {\hat{p}_{\t}(x_i)}  \right\},	
\end{equation}
where $\hat{p}_\t = n_\t / n$. By combining these estimators appropriately we can construct estimators $\drfhat$ and $\tothat$ for the dose-response function $\drf$ and the vector $\tot$, respectively, and any other estimand. Notice that when $\t=\t'$ $\hat{\mu}_{\t,\t}$ is an average over the appropriate subpopulation: $\hat{\mu}_{\t,\t} = \Ent[ y_i]$.

Although in this section we allow for generic estimates $\hat{p}_\t(x)$ and $\muthat(x)$, it is important to distinguish between estimates based upon selected sets that have no ``additional randomness'' and those that do. Model selection based estimation will naturally have two steps: first data-driven selection and then refitting to ameliorate the shrinkage bias and allow the researcher to augment the selected variables. Let $\tilde{S}^\D$ and $\tilde{S}^Y$ be the selected sets and $\hat{S}^\D$ and $\hat{S}^Y$ be the final sets of variables used in the refitting. We will say that these contain no ``additional randomness'' if the added variables (i.e. $\hat{S} \setminus \tilde{S}$, for $Y$ or $\D$) are nonrandomly selected, such as from economic theory or prior knowledge. On the other hand, the added variables may be selected from a random process beyond that included in $\tilde{S}$. The leading example would be using logistic-selected variables in the regressions or vice versa. Then the variables used in $\muthat(x_i)$ depend not only on the randomness of $\tilde{S}^Y$, but also on that of $\tilde{S}^\D$, and hence on $\{ \d \}_{i = 1}^n$. Additional conditions are required for the estimators with additional randomness.

The choice of method is in part dependent on the assumptions of the underlying  model. To illustrate, first, return to Example \ref{eg-parametric}, where we have a purely parametric model with $X = X_\D^* = X_Y^*$. The researcher may want to set $\hat{S}^\D \supset \tilde{S}^\D \cup \tilde{S}^Y$, in order to have a better chance that $S_*^Y \subset \hat{S}^\D$. The set $\hat{S}^\D$ now contains additional randomness due to $\tilde{S}^Y$. Conversely, consider Example \ref{eg-semiparametric}. It is natural to include ``low-order'' basis functions for each underlying covariate, say linear and quadratic polynomials. Thus, the researcher may want to include these in $\hat{S}$, whether or not selected by the group lasso. However, there is no reason that the series terms useful for approximating the functions $\mut(x)$ would be useful for $p_\t(x)$, or vice versa, and no additional randomness is injected.

We now state the sufficient conditions used for treatment effect estimation and inference. For exposition, we present these in three groups: those concerning the underlying DGP, requirements of $\hat{p}_\t(x)$ and $\muthat(x)$ in the ``no additional randomness'' case, and finally the additional conditions to allow for ``additionally random'' selected sets. Begin with conditions on the DGP. Let $U \equiv Y(\t) - \mut(X)$ and impose the following conditions.

\begin{assumption}[Data Generating Process]
	\label{dgp asmpts} \
	$P_n$ obeys the following, with bounds uniform in $n$.
	\begin{enumerate}[label=(\alph{*}), ref=\ref{dgp asmpts}(\alph{*})]

		\item $\{(y_i, \d, x_i')'\}_{i=1}^n$ is an i.i.d. sample from $(Y, \D, X')'$.		\label{iid}

		\item The covariates $X^*$  have bounded support, with $\maxj \vert X^*_j \vert \leq \mathcal{X} < \infty$. Transformations may depend on $n$ but not the underlying data generating process.		\label{bounded}

		\item  $\E[|U|^4 \mid X] \leq \mathcal{U}^4$.	\label{fourth moments} 

		\item $\min_{j \in \N_p,\ \t \in \NTbar} \E[ {X_j^*}^2 U^2] \wedge \E[ {X_j^*}^2 (\indicator\{\D = \t\} - p_\t(X))^2]$ is bounded away from zero.		\label{positive variance} 

		\item For some $r > 0$: $\E[|\mut(x_i) \mu_{\t'}(x_i)|^{1 + r}]$ and $\E[|u_i|^{4 + r}]$ are bounded.	\label{ATE moments}

	\end{enumerate}
\end{assumption}

These conditions are mild and intuitive, and not unique to high-dimensional models or model selection. Assumption \ref{iid} restricts attention to cross-sectional applications. The condition of bounded covariates is unlikely to be a limitation in practice. Any $X^*$ that are underlying variables will naturally be bounded in applications. This condition is automatically satisfied for most common choices of basis functions employed in nonparametric estimation. The rest are moment conditions on the potential outcome models, including allowing the errors to be heteroskedastic and non-Gaussian. The uniform bounds in $n$ are needed for array asymptotics.

We now give precise conditions on the model selector sufficient for uniformly valid inference.
\begin{assumption}[First Stage Restrictions] \label{first stage} \ 
	The estimators $\hat{p}_\t(x)$ and $\muthat(x)$ obey the following for a sequence $\{P_n\}$, uniformly in $\t \in \NTbar$.
	\begin{enumerate}[label=(\alph{*}), ref=\ref{first stage}(\alph{*})]

		\item $ \En[(\hat{p}_\t(x_i) - p_\t(x_i))^2]  = o_{P_n}(1)$ and $\En \left[ (\muthat(x_i) - \mut(x_i))^2\right]  = o_{P_n}(1)$,	\label{consistent}

		\item $ \En[ (\muthat(x_i) - \mut(x_i))^2]^{1/2}  \En[ (\hat{p}_\t(x_i) - p_\t(x_i))^2]^{1/2}    = o_{P_n}(n^{-1/2})$. 	\label{ATE RATES}
		
		\item $ \En[ (\muthat(x_i) - \mut(x_i)) (1 - d_i^t/ p_\t(x_i))]    = o_{P_n}(n^{-1/2})$. 	\label{new}

	\end{enumerate}
\end{assumption}
These two collectively play the same role as the commonly-used, high-level requirement in semiparametrics that each first-step component separately converge at $n^{-1/4}$ at least.\footnote{See \citeasnoun{Newey-McFadden1994_handbook} and \citeasnoun{Chen2007_handbook}, and references therein.} Indeed, \citeasnoun{Belloni-Chernozhukov-Hansen2014_REStud} employ just such a condition for each component. However, by making use of the doubly-robust property we have the weaker conditions shown.\footnote{Many studies in the semiparametric literature relax or do not rely on the $n^{1/4}$ condition, allowing the nonparametric portion to converge at a slower rate, at any rate, or in some cases be inconsistent; examples include \citeasnoun{Powell-Stock-Stoker1989_Ecma}, \citeasnoun{Newey1990_Ecma}, \citeasnoun{Robins-etal2008_IMS}, \citeasnoun{Cattaneo-Jansson-Newey2014_alt}, and Cattaneo, Crump, and Jansson \citeyear{Cattaneo-Crump-Jansson2013_JASA,Cattaneo-Crump-Jansson2014_ET}, among others.} The first is a mild consistency requirement. The second requires an explicit rate on the product of errors, and hence if one function is relatively easy to estimate Assumption \ref{ATE RATES} can be satisfied even if the other does not converge at $n^{-1/4}$. This formalizes the benefit of doubly-robust estimation in general. In high-dimensional, sparse modeling specifically the rates for the first stage depend on the sample size, the number of covariates considered, and the sparsity level. Thus, if one function requires fewer covariates to estimate, i.e. smaller $p$ or $s$, then greater complexity can be allowed for in the other (capturing, in particular, their relative smoothness).

The so-called ``additional-randomness'' estimators are more specific to the (approximately) sparse model context, and so we now codify the sparsity requirements of Section \ref{sec-sparse} and then give the additional conditions required for these estimators.
\begin{assumption}[Sparsity] \label{sparsity} \ 
	For each $n$, $P_n$ obeys \eqref{eqn-propensity}, \eqref{eqn-outcome}, and \eqref{eqn-bias}, with $| S_*^Y| = s_d$ and $|S_*^\D | = s_y$.
\end{assumption}
\begin{assumption}[Regularity conditions for union estimators] \label{ATE union} \ 
	For a sequence $\{P_n\}$, $\log(p) = o(n^{1/3})$ and the estimators $p_\t(x)$ and $\muthat(x)$ obey the following, uniformly $\t \in \NTbar$:
	\[\bigl(  \maxt |u_i| \bigr) \left|   \En [ (\hat{p}_\t(x_i) - p_\t(x_i))^2] \right| = o_{P_n}(n^{-1/2}) 		\quad \text{and} \quad 		
		\left\| \ghatt - \gtruet \right\|_1   \vee   \| \bhatt - \btruet \|_1 = o_{P_n}(\log(p \vee n)^{-1/2}).\]
\end{assumption}
These conditions are needed to apply bounds for self-normalized sums \cite{delaPena-Lai-Shao2009_book}. \citeasnoun{BCCH2012_Ecma} were the first to use these techniques in high-dimensional, sparse models. The first condition is high-level, but can be verified with conditions on the errors and a bound for estimation. For the former, \citeasnoun{BCCH2012_Ecma} assume that $ \max_{i \in \N_n} |u_i| = O_{P_n}(n^{1/q})$ for some $q>2$. A larger $q$ eases the restriction in Assumption \ref{ATE union} but at the expense of stronger conditions on the noise distribution. For example, if $u_i$ are assumed Gaussian, $q$ can be taken to be any (large) positive number.

\begin{remark}[Linear Probability Models]
	\label{remark-LPM}
	Our results cover use of a linear probability model for $p_\t(x)$, instead of the multinomial logistic form. All we require is a sufficiently high-quality approximation of the unknown function, and hence if Assumptions \ref{first stage}, and \ref{ATE union} if appropriate,\footnote{Assumption \ref{ATE union} can be slightly weakened in this case due to the linear link function.} are met then uniform inference is possible using a linear probability model. Our group lasso results (Theorems \ref{thm-ols} and \ref{thm-post ols}) can be used directly to verify these conditions. In the same vein, multinomial logistic regression can be used to estimate $\mut(x)$ if the outcome $Y$ is discretely valued.
\end{remark}

%%%%%%%%%%%%%%%%%%%%%%%%%%%%
\subsection{Theoretical Results}
	\label{sec-ate theory}

We now come to our main results on inference on average treatment effects. Most of our discussion will concern $\mut$ and $\drf$; similar points apply to results for $\mutt$ and $\tot$. Our first result formalizes consistency of our estimates under misspecification. 

\begin{theorem}[Double Robustness]
	\label{thm-ate consistency}
	Consider a sequence $\{P_n\}$ of data-generating processes. Suppose that for some $p_\t^0(x)$ and $\mut^0(x)$, $\En[(\hat{p}_\t(x_i)  -  p_\t^0(x_i))^2]  = o_{P_n}(1)$ and $\En [ (\muthat(x_i) - \mut^0(x_i))^2]   = o_{P_n}(1)$. Let Assumptions \ref{ignorability} and \ref{dgp asmpts} hold for each $n$, with the regularity conditions also holding for $p_\t^0(x)$ and $\mut^0(x)$. If $p_\t^0(x) = p_\t(x)$ or $\mut^0(x) = \mut(x)$, then $ \left|\muthat - \mut \right|  = o_{P_n}(1)$.
\end{theorem}

This theorem formalizes the double-robustness property of our estimators: the propensity score or regression may be misspecified if the limiting objects are well-behaved. Compare to Assumption \ref{consistent}. The nearly identical result for $\mutt$ is omitted to save space.

We now turn to our main inference results. First we demonstrate a Bahadur representation of a generic $\muthat$ or $\mutthat$. These are shown to be equivalent to a sample average of the moment functions $\psi_\t(\cdot)$ and $\psi_{\t,\t'}(\cdot)$, respectively, after proper centering and scaling, evaluated at the true $p_\t(x_i)$ and $\mut(x_i)$. Using these results, asymptotic normality can be obtained for general estimands. We state explicit results for the leading examples $\drf$ and $\tot$.

An asymptotic variance formula is needed to state the results. Define the conditional variance of the potential outcomes s $\sigma_\t^2(x) =  \E[U^2 \vert \D = \t, X = x]$ and the $\Tbar$-square matrix $V_{\drf}$ with elements
	\[V_{\drf}[\t, \t'] = \indicator\{\t = \t'\} \E\left[  \frac{\sigma_\t^2(X)}{p_\t(X)}\right] + \E\left[(\mut(X) - \mut)(\mu_{\t'}(X) - \mu_{\t'})\right] \equiv V_{\drf}^W(\t) + V_{\drf}^B(\t,\t').\]
Straightforward plug-in estimators for these two components are given by\footnote{Estimators can also be based on sample averages of outer products of influence functions, which would include the covariance term that vanishes in expectation.}
	\[\hat{V}_{\drf}^W(\t) = \En\left[ \frac{ \dt (y_i - \muthat(x_i) )^2}{\hat{p}_\t(x_i)^2}\right] 	\qquad \text{ and } \qquad 	\hat{V}_{\drf}^B(\t,\t') = \En\left[ (\muthat(x_i) - \muthat) (\hat{\mu}_{\t'}(x_i) - \hat{\mu}_{\t'}) \right].\]

Our first result gives the asymptotic behavior of $\muthat$ and $\drfhat$ for a sequence of DGPs. 

\begin{theorem}[Estimation of Average Treatment Effects]
	\label{thm-ate}
	Consider a sequence $\{P_n\}$ of data-generating processes that obey Assumptions \ref{ignorability}, \ref{dgp asmpts}, and \ref{first stage} for each $n$. If $\muthat(x_i)$ and $\hat{p}_\t(x_i)$ do not have additional randomness in the estimated supports, we have:
	\begin{enumerate}[label=\arabic{*}., ref=\ref{thm-ate}.\arabic{*}]

		\item \label{thm-ate linearity}  $\sqrt{n} ( \muthat - \mut ) = \sumi \psi_\t(y_i, \dt, \mut(x_i), p_\t(x_i), \mut)  / \sqrt{n} + o_{P_n}(1)$;

		\item \label{thm-ate normality}  $V_{\drf}^{-1/2} \sqrt{n} (\hat{\drf} - \drf) \to_d \mathcal{N}(0,I_{\Tbar})$; and

		\item \label{thm-ate variance} $\hat{V}_{\drf}^W(\t) - V_{\drf}^W(\t) = o_{P_n}(1)$ and $\hat{V}_{\drf}^B(\t,\t') - V_{\drf}^B(\t,\t') = o_{P_n}(1)$.

	\end{enumerate}
If, in addition, Assumptions \ref{sparsity} and \ref{ATE union} hold, then the same is true when the supports contain additional randomness.
\end{theorem}

Theorem \ref{thm-ate} itself may appear standard, but what is nonstandard is that the model selection step of the estimation has been explicitly accounted for. This immediately gives the following uniform inference results.
\begin{corollary}[Uniformly Valid Inference]
	\label{thm-ate uniform}
	Let $\bm{P}_n$ be the set of data-generating processes satisfying the conditions of Theorem \ref{thm-ate} for a given $n$ and $G: \mathbb{R}^{\Tbar} \to \mathbb{R}$ be a fixed, twice uniformly continuously differentiable function with gradient $\nabla_G$ such that $\liminf_{n \to \infty} \| \nabla_G(\drf)\|_2$ is bounded away from zero. Then for $c_\alpha = \Phi^{-1}(1 - \alpha/2)$, we have:
	\[ \sup_{P \in \bm{P}_n}  \left|  \P_P \left[  G(\drf) \in \left\{ G(\drfhat) \pm c_\alpha \sqrt{ \nabla_G(\drfhat)' \hat{V}_{\drf} \nabla_G(\drfhat) / n}\right\} \right] - (1 - \alpha) \right| \to 0.\]
\end{corollary}

Corollary \ref{thm-ate uniform} shows that these procedures are uniformly valid over the class of DGPs we consider, and hence will be reliable in applications. The crucial insight that leads to uniform inference is to change the goal of model selection away from perfect \emph{covariate} selection (the oracle property) and to high-quality approximation of the underlying \emph{functions} (here $p_\t(\cdot)$ and $\mut(\cdot)$). This fundamental shift in focus allows us to avoid the uniformity problems demonstrated by Leeb and P{\"o}tscher. Assumption \ref{first stage} formalizes exactly the quality of approximation needed. Such an approximation can be found for any element in $\bm{P}_n$, and hence inference is uniformly valid over that class. This method of proving uniformity follows \citeasnoun{Belloni-Chernozhukov-Hansen2014_REStud} and \citeasnoun{Romano2004_SJS}, and is distinct from the approach of \citeasnoun{Andrews-Guggenberger2009_JoE}.

Results for the treatment effects on the treated are similar. The variance formula for $\tot$ is slightly more cumbersome. Define the $\T$-square matrix $V_{\tot}$ with elements
	\begin{align*}
		V_{\tot}[\t, \t']  & =  \indicator\{\t = \t'\} \E\left[ \frac{p_\t(X)}{p_\t^2}\left[\sigma_t^2(X) + \left(\mut(X) - \mu_0(X) - \mu_{\t,\t} + \mu_{0,\t}\right)^2 \right] \right] +  \E\left[\frac{p_\t(X) p_{\t'}(X) }{p_\t p_{\t'} p_0(X)}\sigma_0^2(X) \right]  		\\
		& \equiv   V_{\tot}^W(\t) + V_{\tot}^B(\t,\t').
	\end{align*}
Straightforward plug-in estimators for these two components are given by
\[\hat{V}_{\tot}^W(\t) = \En\left[ \frac{ \dt }{\hat{p}_\t^2} \left[ \left(y_i - \hat{\mu}_0(x_i) - \hat{\mu}_{\t,\t} + \hat{\mu}_{0,\t}\right)^2 \right] \right] \text{ and } \hat{V}_{\tot}^B(\t,\t') = \En\left[ \frac{\hat{p}_\t(x_i) \hat{p}_{\t'}(x_i) }{\hat{p}_\t \hat{p}_{\t'} \hat{p}_0(x_i)^2} \d^0(y_i - \hat{\mu}_0(x_i) )^2 \right].\]
Note that we needn't estimate $\mut(x)$ and $\sigma_t^2(x)$, due to the simplification in Remark \ref{remark-tot}. With this notation, we have the following results. Proofs are so similar to those for Theorem \ref{thm-ate} and Corollary \ref{thm-ate uniform} that we omit them. 
\begin{theorem}[Estimation of Treatment Effects on Treated Groups]
	\label{thm-tot}
	Consider a sequence $\{P_n\}$ of data-generating processes that obey Assumptions \ref{ignorability}, \ref{dgp asmpts}, and \ref{first stage} for each $n$. Then under $P_n$, as $n \to \infty$, if $\muthat(x_i)$ and $\hat{p}_\t(x_i)$ do not have additional randomness in the estimated supports:
	\begin{enumerate}[label=\arabic{*}., ref=\ref{thm-ate}.\arabic{*}]

		\item \label{thm-tot linearity}  $\sqrt{n} ( \mutthat - \mutt ) = \sumi \psi_{\t,\t'}(y_i, \dt, \mut(x_i), p_\t(x_i), p_{\t'}(x_i),\mutt) / \sqrt{n} + o_{P_n}(1)$;

		\item \label{thm-tot normality}  $V_{\tot}^{-1/2} \sqrt{n} ( \tothat - \tot ) \to_d \mathcal{N}(0,I_\T)$; and

		\item \label{thm-tot variance} $\hat{V}_{\tot}^W(\t) - V_{\tot}^W(\t) = o_{P_n}(1)$ and $\hat{V}_{\tot}^B(\t,\t') - V_{\tot}^B(\t,\t') = o_{P_n}(1)$.
	\end{enumerate}
If, in addition, Assumptions \ref{sparsity} and \ref{ATE union} hold, then the same is true when the supports contain additional randomness.
\end{theorem}

\begin{corollary}[Uniformly Valid Inference]
	\label{thm-tot uniform}
	Let $\bm{P}_n$ be the set of data-generating processes satisfying the conditions of Theorem \ref{thm-tot} for a given $n$ and $G: \mathbb{R}^{\T} \to \mathbb{R}$ be a fixed, twice uniformly continuously differentiable function with gradient $\nabla_G$ such that $\liminf_{n \to \infty} \| \nabla_G(\tot)\|_2$ is bounded away from zero. Then for $c_\alpha = \Phi^{-1}(1 - \alpha/2)$, we have:
	\[ \sup_{P \in \bm{P}_n} \left| \P_P  \left[  G(\tot) \in \left\{ G(\tothat) \pm c_\alpha  \sqrt{ \nabla_G(\tothat)' \hat{V}_{\tot} \nabla_G(\tothat) / n}\right\} \right] - (1 - \alpha) \right| \to 0.\]
\end{corollary}

%%%%%%%%%%%%%%%%%%%%%%%%%%%%
\subsection{Efficiency Considerations}
	\label{sec-efficiency}

The prior theoretical results are aimed at delivering robust inference. In this section, we briefly discuss the efficiency of our estimator according to two criteria: semiparametric efficiency and oracle efficiency. To put each on sound conceptual footing we separate discussion and restrict to an appropriate set of models. 

For semiparametric efficiency, $p_\t(x)$ and $\mu_\t(x)$ are nonparametric objects, as in Example \ref{eg-semiparametric}, $X$ are fixed-dimension variables and the DGP does not vary with $n$. If we ``upgrade'' the mean independence of Assumption \ref{mean inde} to full, namely $ \{ Y(\t) \}_{\NTbar} \inde \D \vert X$, then Theorems \ref{thm-ate} and Theorem \ref{thm-tot} immediately yield asymptotic linearity and semiparametric efficiency, attaining \possessivecite{Hahn1998_Ecma} or \possessivecite{Cattaneo2010_JoE} bounds. This requires there be no (known) instruments for treatment status in $X$, as implicitly assumed in those works, else the bound may change \cite{Hahn2004_REStat}.

Turning to oracle efficiency, an alternative to our robust approach is to prove that the true support can be found with probability approaching one (the oracle property), then conduct inference conditioning on this event. This approach cannot be made uniformly valid, but may be of interest in the exactly sparse models of Example \ref{eg-parametric} (there is no ``true'' support in approximately sparse models), because discovering the true support is equivalent to finding the variables in the causal mechanism \cite{White-Lu2011_REStat}, if one exists. This may be interesting in its own right, or for future applications by way of hypothesis generation. The post oracle selection estimator is made efficient by using only the variables important for $\mut(x_i) = \E[Y \vert \D = \t, x_i]$. This amounts to entirely removing the instrumental variables indexed by $S_*^\D \setminus S_*^Y$, whose inclusion would, in general, reduce efficiency, though not increase bias. Further, $S_*^Y \setminus S_*^\D$ are excluded from propensity score estimation.

Perfect selection requires two strong conditions: (i) an orthogonality condition on the Gram matrices that restricts the correlation between the variables in and out of the true support \cite{Bach2008_JMLR}, and (ii) a \emph{beta-min} condition bounding the nonzero coefficients away from zero. Intuitively, highly correlated variables cannot be distinguished, nor can coefficients sufficiently close to zero be found with certainty. Both bounds may depend on $n$, and in particular the lower bound on the coefficients may vanish at an appropriate rate. Under such conditions, it is straightforward to show that $S_*^Y$ and $S_*^\D$ can be found with probability approaching one.

%%%%%%%%%%%%%%%%%%%%%%%%%%%%
%%%%%%%%%%%%%%%%%%%%%%%%%%%%
\section{Group Lasso Selection and Estimation}
	\label{sec-grplasso}

We now give details for group lasso model selection and estimation, and make the refitting precise. Section \ref{sec-lambda} discusses penalty choices and implementation. Restricted and sparse eigenvalues, key quantities in our bounds, are discussed in Section \ref{sec-eigenvalues}. Our main nonasymptotic results are stated in Section \ref{sec-grplasso results}. These results are of interest more generally in the literature on high-dimensional sparse models Finally, Section \ref{sec-rates} gives asymptotic rates and verifies the conditions of Section \ref{sec-ate}.

We first select covariates by applying the group lasso penalty to the multinomial logistic loss (for the propensity scores) and to least squares loss (to estimate the outcome regression). The loss functions are defined as
\[\M(\gany) = \sumT \En \left[  -  \dt \log\left( \ptany \right) \right]		\qquad \text{ and } \qquad 		\SSE(\bany) = \sumTbar \Ent[ (y_i - {x_i^*}' \banyt)^2].\] Then, the group lasso estimates for the propensity score coefficients, denoted $\gtilde$, and the regression coefficients, $\btilde$, respectively solve
\begin{equation}
	\label{grplasso}
	\gtilde  = \argmin_{\gany \in \mathbb{R}^{p\T}} \left\{ \M(\gany)  +  \lambda_\D \vertiii{\gany}  \right\}		\quad \text{ and } \quad 		\btilde  = \argmin_{\bany \in \mathbb{R}^{p\Tbar}} \left\{ \SSE(\bany)  + \lambda_Y \vertiii{\bany}  \right\},
\end{equation}
where $\lambda_\D$ and $\lambda_Y$ are the penalty parameters discussed below and $\vertiii{\gany}$ is the mixed $\ell_2$/$\ell_1$ norm.

To ameliorate the downward bias induced by the penalty and to allow for researcher-added variables, we refit unpenalized models.\footnote{The bias is away from the pseudo-true coefficients of the sparse parametric representation, $\gtrue$ and $\btrue$. There is no relation to specification biases $B_\t^\D$ and $B_\t^Y$.} Let $\tilde{S}^\D = \{j : \| \gtildej \|_2 > 0 \}$ and $\tilde{S}^Y = \{j : \| \btildej \|_2 > 0 \}$ be the selected covariates and $\hat{S}^\D$ and $\hat{S}^Y$ those used in refitting.\footnote{When $\supp(\gtruet)$ and $\supp(\btruet)$ do not vary much over $\t$, the group lasso is known to have better properties than the ordinary lasso in terms of selection and convergence. \citeasnoun{Obozinski-Wainwright-Jordan2011_AoS} give a sharp bound on the overlap necessary to yield improvements, while \citeasnoun{Huang-Zhang2010_AoS}, \citeasnoun{Kolar-Lafferty-Wasserman2011_JMLR}, and \citeasnoun{Lounici-etal2011_AoS} also demonstrate advantages of the group lasso approach. These works show, among other things, that the group lasso advantage increases with large $\T$, and with the group structure, may perform better with smaller samples. We defer to the works cited for a formal discussion.} We require $\hat{S} \supset \tilde{S}$ and $|\hat{S}| \leq s$ for $\D$ and $Y$ (we will prove that $| \tilde{S}| \leq s$ in both cases). The refitting estimators solve
\begin{equation}
	\label{post grplasso}
	\ghat  = \argmin_{\gany, \ \supp(\ganyt) = \hat{S}^\D} \left\{ \M(\gany)  \right\}
		\qquad \text{ and } \qquad 		
	\bhat  = \argmin_{\bany, \ \supp(\banyt) = \hat{S}^Y} \left\{ \SSE(\bany)  \right\}.
\end{equation}

\begin{remark}[Weighted Penalties]
	\label{remark-penalty}
The group lasso penalty can be weighted in two ways. First, one may weight the $\ell_2$ portion, as in $\lambda_\D \sumj \| \bm{X}_j \ganyj \|_2$, where $\bm{X}_j$ is the design matrix for covariate $j$, across all the treatments. Other weight matrices are possible, but with this choice, the estimate is invariant to within group (treatment) reparameterizations, and is thus scale invariant for each covariate. We therefore assume $\En[({x_{i,j}^*})^2] = 1$ without loss of generality.

Second, the $\ell_1$ norm can be weighted to give a penalty of the form $\lambda_\D \sumj w_j \| \ganyj \|_2$. Two common choices for $w_j$ are the number of variables in group $j$ or an adaptive penalty from a pilot estimate. Our groups are equally sized, and although adaptive procedures may improve oracle properties \cite{Zou2006_JASA,Wei-Huang2010_Bern}, our goal is not perfect selection. 
\end{remark}

%%%%%%%%%%%%%%%%%%%%%%%%%%%%
\subsection{Choice of Penalty}
	\label{sec-lambda}

We must specify choices of $\lambda_\D$ and $\lambda_Y$ for programs \eqref{grplasso}. From a theoretical point of view, these must be chosen so that the penalty dominates the noise, which is captured by the magnitude of the score in the dual of the $\vertiii{\cdot}$ norm, with high probability. To acheive this, we set
\begin{equation}
	\label{lambda}
	\lambda_\D =  \frac{ 2 \mathcal{X} \sqrt{\T}  }{\sqrt{ n } } \left( 1 + \frac{\log(p \vee n)^{3/2 + \delta_\D}}{\sqrt{\T}} \right)^{1/2}     	 \text{ and }    		  \lambda_Y = \frac{ 4 \mathcal{X} \mathcal{U} \sqrt{\Tbar} }{\sqrt{\underline{n} } } \left( 1 + \frac{\log(p \vee \underline{n})^{3/2 + \delta_Y}}{\sqrt{\Tbar}} \right)^{1/2},
\end{equation}
for some $\delta_\D > 0$ and $\delta_Y > 0$. With these choices, $\lambda_\D > 2 \maxj \| \En [(p_\t(x_i) - \dt)  x_{i,j}^* ] \|_2$ and $\lambda_Y >  4 \maxj \| \Ent [u_i x_{i,j}^* ] \|_2$ with probability $1 - \mathcal{P}$ for a small (and shrinking) $\mathcal{P}$. In generic terms, $\lambda$ is of the form $\Lambda (1 + r_n)$, where $\Lambda$ is an upper bound on the true score and $r_n$ is a rate that depends on $n$ and $p$.\footnote{The slight differences in the two are as follows. The full sample has information on the logistic coefficients, so $n$ appears instead of $\underline{n}$. No error bound appears in $\lambda_\D$ because the errors are bounded by one. The multiple $4$ for $\lambda_Y$, instead of 2, can be traced to the quadratic loss. These forms are determined at heart by the maximal inequality of \citeasnoun{Lounici-etal2011_AoS}.} The specific rate chosen serves to balances the rate of convergence against the concentration effect: a smaller $r_n$ would increase the rate of convergence, but at the expensive of lowering the concentration probability $1 - \mathcal{P}$. In the Appendix we show that (for appropriate $\delta$ and $n$ or $\underline{n}$) the concentration probability is given by
\begin{equation}
	\label{probability}
	\mathcal{P} = \frac{ 4 \sqrt{ \log(2p) ( 1 + 64 \log(12 p)^2) } } { \log(p \vee n)^{3/2 + \delta}  },
\end{equation}

There are two practical methods to make these choices for feasible for implementation.  When $\hat{p}_\t(x)$ and $\muthat(x)$ are used to estimate average treatment effects, the decreased sensitivity of the final estimate to the first stage, thanks to the doubly-robust estimator, in turn results in less sensitivity to the choice of penalty (through the sparsity).\footnote{To our knowledge, no formal results exist on ``optimal'' penalty parameter choices for inference in high-dimensional problems nor are any procedures free of user-specified choices.} 
The first option is an iterative procedure to estimate the unknown $\mathcal{X}$ and $\mathcal{U}$ in $\lambda_Y$ and $\lambda_\D$, as employed by \citeasnoun{BCCH2012_Ecma} (validity of this procedure may be established along the same lines as in that study). We use $\maxi \maxj |x_{i,j}^*|$ for $\mathcal{X}$ and estimate $\mathcal{U}$ by iteration: given an initial estimate $\muthat^{(0)}(x)$, set $\hat{\mathcal{U}}^{(k)} = \En[(y_i - \muthat^{(k-1)}(x_i))^4]^{1/4}$, where $\muthat^{(k)}(x_i)$, $k > 0$, is based on \Eqref{post grplasso}. In implementation we found 10 iterations more than sufficient, and based the initial estimate on ridge regression (with penalty chosen by cross validation). A second option is to select $\lambda_Y$ and $\lambda_\D$ directly by cross-validation. This has the appealing feature that the precise forms of \Eqref{lambda} need not be characterized and estimated. If interest lies in the underlying functions $p_\t(x)$ and $\mut(x)$, cross validation is appropriate as it minimizes a relevant loss function. Formal results establishing the validity of cross-validation are not available, but it performs well in practice.

%%%%%%%%%%%%%%%%%%%%%%%%%%%%
\subsection{Restricted Eigenvalues}
	\label{sec-eigenvalues}

The local behavior of optimizations \eqref{grplasso} and \eqref{post grplasso} is captured by their respective Hessians, which involve the second moment matrix of the covariates. The eigenvalues of such matrices will be explicit in our bounds. We are interested in finite sample bounds, and so we will only discuss the empirical Gram matrices (see Remark \ref{remark-gram}). Define
\begin{equation}
	\label{gram}
	Q = \En[{x_i^*}{x_i^*}'] 	\qquad \text{ and } \qquad 	Q_\t = \Ent[{x_i^*}{x_i^*}'].
\end{equation}
In high-dimensional data, both are singular, and so we use restricted eigenvalues and sparse eigenvalues \cite{Bickel-Ritov-Tsybakov2009_AoS}.

For the multinomial logistic regression, the minimal restricted eigenvalue is defined by
\begin{equation}
	\label{RE full} 
	\kappa_\D^2 \leq \min_\delta \left\{ \frac{\sumT \delta_{\t}' Q \delta_\t }{\| \delta_{\newdot, S_\D^*} \|_2^2 } : \delta \in \mathbb{R}^{p \T} \setminus \{0\}, \vertiii{\delta_{\newdot, \{S_*^\D\}^c} } \leq 4 \vertiii{\delta_{\newdot, S_*^\D } }  		 \right\}.
\end{equation}
For least squares estimation we instead use
\begin{equation}
	\label{RE} 
	\kappa_Y^2 \leq \min_\delta \left\{ \frac{\sumTbar \delta_{\t}' Q_\t \delta_\t }{\| \delta_{\newdot, S_Y^*} \|_2^2 } : \delta \in \mathbb{R}^{p \Tbar} \setminus \{0\}, \vertiii{\delta_{\newdot, \{S_*^Y\}^c} } \leq 3 \vertiii{\delta_{\newdot, S_*^Y } }  	 \right\}.
\end{equation}
Note that $Q$ appears for $\kappa_\D$, whereas the $Q_\t$ are used in $\kappa_Y$. The restricted set, or cone constraint, requires the magnitude of $\dany$ off the true support be small relative to the true support, measured in the group lasso norm.\footnote{The multiplier of 4 in the constraint for $\kappa_\D$ is traceable to the nonlinear model.} We will show that $(\gtilde - \gtrue)$ and $(\btilde - \btrue)$ obey the respective constraints.

In contrast, the refitting errors $(\ghat - \gtrue)$ and $(\bhat - \btrue)$ from \eqref{post grplasso} may not obey the cone constraint, but are sparse by construction. This motivates the use of sparse eigenvalues. For a set $S \subset \Np$ and a $p \times p$ matrix $\tilde{Q}$, define
\begin{equation}
	\label{SE}
	\underline{\phi}\{\tilde{Q}, S\}^2 = \min_{\delta \in \mathbb{R}^p,\,  \supp(\delta) = S} \frac{\delta' \tilde{Q} \delta}{\| \delta \|_2^2} 
	\qquad \text{ and }  \qquad 
	\overline{\phi}\{\tilde{Q}, S\}^2 = \max_{\delta \in \mathbb{R}^p,\,  \supp(\delta) = S} \frac{\delta' \tilde{Q} \delta}{\| \delta \|_2^2}.
\end{equation}
Finally, it will be useful to define a bound on $\overline{\phi}\{\tilde{Q}, S\}$ over all subsets of a certain size. To this end, for any integer $m$, define $\dbarphi(\tilde{Q},m) = \max_{S \subset \N_p,\, |S| \leq m} \overline{\phi}\{\tilde{Q}, S\}$. 

We take these quantities to be primitive, and defer to the literature. For example, \citeasnoun{vandeGeer-Buhlmann2009_EJS}, \citeasnoun{Huang-Zhang2010_AoS}, \citeasnoun{Raskutti-Wainwright-Yu2010_JMLR}, \citeasnoun{Rudelson-Zhou2013_IEEE}, and \citeasnoun{Belloni-Chernozhukov-Hansen2014_REStud}. In particular, \citeasnoun{Huang-Zhang2010_AoS} show that the group lasso may need fewer observations to satisfy conditions on $\underline{\phi}\{\tilde{Q}, S\}$.

\begin{remark}
	\label{remark-gram}
Often, invertibility of $Q$ and $Q_\t$ relies on their convergence to nonsingular population counterparts.\footnote{This is standard in fixed-dimension models, and has been used for diverging-dimensions parametric models \cite{He-Shao2000_JMA} and nonparametrics \cite{Newey1997_JoE,Huang2003_AoS,Cattaneo-Farrell2013_JoE,Belloni-etal2015_JoE,Chen-Christensen2015_JoE}. The eigenvalue assumptions employed in those works are conceptually the same as the the restricted eigenvalues used here, only restricted to the $p<n$ case.} Some of the papers cited use this approach and our results can be restated in this way by conditioning on the event that $Q$ and $Q_\t$ are close to their counterparts in the appropriate sense, and adjusting the probability with which the conclusions hold. We instead take bounds to be infinite if the minimum eigenvalues are zero. 
\end{remark}

%%%%%%%%%%%%%%%%%%%%%%%%%%%%
\subsection{Finite Sample Theoretical Results}
	\label{sec-grplasso results}

We now have the necessary notation and assumptions to state our theoretical results on group lasso estimation, beginning with multinomial logistic regression, followed by a terse treatment of linear models. Corollary \ref{thm-grplasso overview} is a special case of the results in this section, see Section \ref{sec-rates}.

Our first result is a nonasymptotic bound on the group lasso estimates from \eqref{grplasso}.
\begin{theorem}[Group Lasso Estimation of Multinomial Logistic Models]
	\label{thm-mlogit}
	Suppose Assumptions \ref{overlap}, \ref{iid}, \ref{bounded}, \ref{fourth moments}, and \ref{sparsity} hold. Define $ A_p = p_{\min}/(0 \vee (p_{\min} - \bias^d))$ and
	\[R_\M =  \left(A_p \big/ p_{\min}\right)^{\Tbar}     \T A_K \left(    6  \lambda_\D \sqrt{|S_*|} \kappa_\D^{-1}    +   8  \bias^d \sqrt{\T}   \right),\]
for $A_K > 2 \kappa_\D^2 \left\{   \kappa_\D^2 - (2/3)  \mathcal{X} \sqrt{\T} \left( 30 \lambda_\D |S_*| + 100 \sqrt{|S_*|} \kappa_\D \bias^d \sqrt{\T} + 80 \kappa_\D^2 (\bias^d)^2 \T \lambda_\D^{-1} \right) \right\}^{-1}$. Then with probability $1 - \mathcal{P}$, we have
\begin{enumerate}

	\item $\displaystyle \max_{\t \in \NT} \En[(\pttilde - p_\t(x_i))^2] ^{1/2} \leq R_\M + \bias^d$,

	\item $\displaystyle \max_{\t \in \NT} \left\| \gtildet - \gtruet \right\|_1  \leq R_\M \sqrt{|\tilde{S}^\D \cup S_\D^*| \big/ \underline{\phi}\{Q, \tilde{S}^\D \cup S_\D^*\}}$,

	\item and $\displaystyle | \tilde{S}^\D | \leq 8 s L_n \left( \min \{ \dbarphi(Q, m)  : m \in \N_Q^\D \} \right)$, 
\end{enumerate}
where $\N_Q^\D = \left\{ m \in \{1, 2, \ldots n\} : m > 8 s L_n  \dbarphi(Q, m) \right\}$ and $L_n = \T \left( (R_\M + \bias^d) \big/ (\lambda_\D \sqrt{s}) \right)^2 $. 
\end{theorem}
This theorem is new to the literature, to the best of our knowledge. Much of the detail involves capturing the finite sample behavior of the Hessian and Gram matrices. We discuss the features of this result in the following remarks.
\begin{itemize}
		
	\item The Hessian of $\M(\gany)$ is $\En[\H_i \otimes {x_i^*}{x_i^*}']$ for a $\T$-square matrix $\H_i$ that depends on the coefficients and ${x_i^*}$ through the estimated probabilities $\ptany$. The error $R_\M$ depends on how well-controlled is this matrix. The factors $p_{\min}$, $A_p$, and $A_K$ capture the behavior of $\H_i$ and $\kappa_\D^{-1}$ accounts for the rest. Under overlap, the true probabilities are bounded below by $p_{\min}$, and hence $p_{\min}^{-\Tbar}$ captures the nonsingularity of the population version of $\H_i$. To get to this point requires two steps. First, the sparse parametric representations $\pttrue$ must also be bounded away from zero, leading to the factor of $A_p$. This is essentially a bias condition, which in the asymptotic case holds trivially: $A_p$ may be chosen arbitrarily close to one as $\bias^d \to 0$. Second, $A_K$ controls the neighborhood in which $\pttilde$ is also bounded away from zero. Intuitively (and asymptotically), the estimate will be in a small (shrinking) neighborhood of the $\pttrue$. In asymptotics $A_K$ may be chosen arbitrarily close to 2, which stems from the factor of 1/2 in a quadratic expansion of $\M(\cdot)$. A lower bound on $A_K$ is required in finite samples to ensure that $\pttilde$ is positive, and hence the two-term expansion is valid. This is analogous to \possessivecite{Belloni-Chernozhukov2011_AoS} ``restricted nonlinear impact coefficient'' approach, also used by \citeasnoun{Belloni-etal2014_WP} with a central difference that $A_K$ is captured in our bound directly.

	\item The maximal sparse eigenvalues are crucial to the bound on $| \tilde{S}^\D |$. In many prior results, the latter is bounded using the largest eigenvalue of $Q$ itself, i.e.\! $\dbarphi(Q,n)$. Adapting the technique of \citeasnoun{Belloni-Chernozhukov2013_Bern} to the present case, we are able to find a tighter bound, which yields sparsity proportional to $s$ under weaker conditions. This is crucial for refitting.

	\item For the linear model the constants in the group lasso bounds can offset the (logarithmic) suboptimality in rate \cite{Huang-Zhang2010_AoS,Lounici-etal2011_AoS}, and this may be true here as well. This is application dependent however.

\end{itemize}

The error bounds for post-selection estimation are more complex and depend in part on the good properties of the initial group lasso fit. The following theorem gives our results.
\begin{theorem}[Post-Selection Multinomial Logistic Regression]
	\label{thm-post mlogit}
	Suppose the conditions of Theorem \ref{thm-mlogit} hold. To save notation, let $S_\D = \hat{S}_\D \cup S_\D^*$ and $\underline{\phi} = \underline{\phi}\{Q, \hat{S}_\D \cup S_\D^*  \}$. Then for
	\[A_K > 2 \left\{ \frac{\underline{\phi}^2}{ \underline{\phi}^2  -  \mathcal{X} \sqrt{\T} ( \lambda_\D  | S_\D | + \bias^d \underline{\phi} \sqrt{\T} \sqrt{| S_\D |} )  } \right\} 				\vee \left\{  \frac{\underline{\phi}}{ \underline{\phi}  -   2 R_\M  \mathcal{X}   \sqrt{\T} \sqrt{| S_\D |} } \right\}\]
define $R_\M' = \left(A_p / p_{\min}\right)^{\Tbar} \T A_K \left( \lambda_\D \sqrt{ | S_\D |} \underline{\phi}^{-1}/2     + \bias^d \sqrt{\T} \right)$ and 
\[R_\M'' = \left\{  R_\M  \right\}  \vee  \left\{  R_\M' + \left[  R_\M' R_\M + \left(A_p \big/ p_{\min}\right)^{\Tbar} \T A_K R_\M^2 \right]^{1/2} \right\}.\]
Then with probability $1 - \mathcal{P}$, $\displaystyle \max_{\t \in \NT} \En[(\pthat - p_\t(x_i))^2] ^{1/2} \leq   R_\M''   + \bias^d$, and \linebreak $\displaystyle \max_{\t \in \NT} \left\| \ghatt - \gtruet \right\|_1  \leq \left(|S^\D| / \underline{\phi} \right)^{1/2}  R_\M''$.
\end{theorem}
It is not readily discernible if these bounds improve upon the initial fit. This will depend on the DGP, the selection success of the initial fit, and any added variables. In this result, further lower bounds on $A_K$ are required to handle the sparse eigenvalues, compared to the restricted version in Theorem \ref{thm-mlogit}. The role played by $A_K$ is the same in both cases, as with the other factors.

It is worth noting that, despite the complexity of multinomial logistic regression, the conditions for Theorems \ref{thm-mlogit} and \ref{thm-post mlogit} are simple and intuitive, and match those used for linear models.

We now give our results for group lasso estimation of the conditional outcome regressions. In computing $\mut(x_i)$ for $\dt \neq 1$ we are performing out of sample prediction, which slightly complicates the bounds. Our first result is on the initial group lasso fit.

\begin{theorem}[Group Lasso Estimation of Linear Models]
	\label{thm-ols}
	Suppose Assumptions \ref{overlap}, \ref{iid}, \ref{bounded}, \ref{fourth moments}, and \ref{sparsity} hold.  To save notation, let $S_Y = \tilde{S}^Y \cup S_Y^*$. Define
	\[R_\SSE =  \left( \frac{3 \lambda_Y \sqrt{s}}{\kappa_Y} +  2 \bias^y \right).\]
	Then with probability $1 - \mathcal{P}$, we have
	\begin{enumerate}

		\item $\displaystyle \maxTbar \En[({x_i^*}'\btildet - \mut(x_i))^2] ^{1/2} \leq \left( \overline{\phi}\{Q, S_Y \} \big/ \underline{\phi}\{Q_\t, S_Y \} \right)^{1/2} R_\SSE + \bias^y$,

		\item $\displaystyle \maxTbar  \left\| \btildet - \btruet \right\|_1 \leq \left( |S_Y| \big/ \underline{\phi}\{Q, S_Y \} \right)^{1/2} \left( \overline{\phi}\{Q, S_Y \} \big/ \underline{\phi}\{Q_\t, S_Y \} \right)^{1/2} R_\SSE$, 

		\item and $ | \tilde{S}^Y | \leq 32 s L_n \left\{ \min_{m \in \N_Q^Y} \sum_{\t \in \NTbar} \dbarphi(Q_\t, m) \right\}$,

	\end{enumerate}
	where $ \N_Q^Y = \left\{ m \in \{1, 2, \ldots, \overline{n}\} : m > 32 s L_n \sum_{\t \in \NTbar} \dbarphi(Q_\t, m) \right\}$ and $L_n = \left( (R_\SSE + \bias^y) \big/ (\lambda_Y \sqrt{s}) \right)^2$.
\end{theorem}

This theorem generalizes \citeasnoun{Lounici-etal2011_AoS} to the nonparametric, approximately sparse case, improves the sparsity bound, and gives out of sample prediction (imputation) results. The analogous generalization for within sample prediction loss (e.g.\ multi-task learning), $\Ent[({x_i^*}'\btildet - \mut(x_i))^2] ^{1/2}$, may be found in the Supplement.

For refitting, we are predicting for the entire sample and so we utilize the general results given by \citeasnoun{BCCH2012_Ecma} for post-selection estimation of least squares. The following result is a direct implication of their Lemma 7 and our Theorem \ref{thm-ols}.

\begin{theorem}[Post-Selection Linear Regression]
	\label{thm-post ols}
	Suppose $\log(p) = o(n^{1/3})$ in addition to the conditions of Theorem \ref{thm-ols}. Then for constants $A_1$, $A_2$, $A_3$, and $A_4$ not depending on $n$ nor the DGP:
	\[\En[ (x_i'\bhatt - \mu_t(x_i))^2]^{1/2}  \leq A_1  \sqrt{ \frac{ s (\T \wedge \log(s \T)) }{ n \underline{\phi}\{Q, S_Y^*\}} }  +  A_2 \sqrt{ \frac{ |\hat{S}_Y \setminus S_Y^*| \log(p\T) } {n \underline{\phi} \{Q, S_Y^{FP}\} } }  +   A_3 \sqrt{ \En[({x_i^*}'\btildet - \mut(x_i))^2] } \]
and $\displaystyle \maxTbar \| \bhatt - \btruet\|_1  \leq    A_4 \left( |\hat{S}_Y \cup S_Y^* | \En[ (x_i'\bhatt - \mu_t(x_i))^2]  \big/ \underline{\phi}\{Q, \hat{S}_Y \cup S_Y^*\}  \right)^{1/2}$. 
	
\end{theorem}

As above, the performance of the refitting procedure depends in part on the success of the initial group lasso fit. Indeed, the middle term is dropped if the true support union is found. The constants $A_k$, k=1, 2, 3, 4 are not given explicitly but are known to be absolute bounds \cite{delaPena-Lai-Shao2009_book} under Assumption \ref{dgp asmpts}. This result is less precise than Theorems \ref{thm-mlogit} and \ref{thm-post mlogit}, but sufficient to verify Assumptions \ref{first stage} and \ref{ATE union}.

%%%%%%%%%%%%%%%%%%%%%%%%%%%%
\subsection{Asymptotic Analysis and Verification of High-Level Conditions}
	\label{sec-rates}

This section derives rates of convergence for the group lasso estimates and uses these results to verify Assumptions \ref{first stage} and \ref{ATE union} in Section \ref{sec-ate}. For simplicity, we only state results for the post-selection estimators that we recommend in practice. In reducing the finite sample results of Theorems \ref{thm-post mlogit} and \ref{thm-post ols} to rates we retain the dependence on $n$, $p$, $s$, and the bias. Note that the number of treatments is fixed, and the overlap assumption ensures that all $n_\t \propto n$. Further, the various (restricted and sparse) eigenvalues are commonly taken to be bounded (or bounded away from zero) in asymptotic analyses. This accounts for the remaining factors in the bounds. For multinomial logistic regression, we obtain the following result.
\begin{corollary}[Asymptotics for Multinomial Logistic Regression]
	\label{thm-mlogit rates}
	Suppose the conditions of Theorem \ref{thm-post mlogit} hold and further that (i) $\lambda_\D s_d = o(1)$, (ii) $\kappa_\D$ is bounded away from zero, and (iii) $\min_{S: |S| =O(s)} \underline{\phi}\{Q,S\}$ is bounded away from zero and $\dbarphi(Q, \cdot)$ is bounded, uniformly in $\N_Q^\D$. Then
	\begin{enumerate}
		\item $| \tilde{S}^\D | = O_{P_n}(s_d)$,
		\item $\En[(\pthat - p_\t(x_i))^2] = O_{P_n} \left(n^{-1} s_d \log(p \vee n)^{3/2 + \delta_\D}   + (\bias^d)^2 \right)$,vand
		\item $ \| \ghatt - \gtruet \|_1 = O_{P_n}\left(\sqrt{ n^{-1} s_d^2\log(p \vee n)^{3/2 + \delta_\D} } + \bias^d \sqrt{s_d} \right)$.
	\end{enumerate}
\end{corollary}
Similarly, we have the following for the linear models.
\begin{corollary}[Asymptotics for Linear Regression]
	\label{thm-ols rates}
	Suppose the conditions of Theorem \ref{thm-post ols} hold and further that (i) $\lambda_Y \sqrt{s_y} = o(1)$, (ii) $\kappa_Y$ is bounded away from zero, and (iii) uniformly in $\NTbar$, $\min_{S: |S| =O(s)} \underline{\phi}\{Q_\t,S\} \wedge \underline{\phi}\{Q,S\}$ is bounded away from zero and $\dbarphi(Q, \cdot) \vee \dbarphi(Q_\t, \cdot)$ is bounded uniformly in $\N_Q^Y$. Then
	\begin{enumerate}
		\item $| \tilde{S}^Y | = O_{P_n}(s_y)$,
		\item $\En[(\muthat(x_i) - \mut(x_i))^2] = O_{P_n} \left(  n^{-1} s_y \log(p \vee n)^{3/2 + \delta_Y}  + ( \bias^y)^2  \right)$, and
		\item $ \| \btildet - \btruet \|_1 = O_{P_n}\left(\sqrt{ n^{-1} s_y^2 \log(p \vee n)^{3/2 + \delta_Y} } + \bias^y \sqrt{s} \right)$.
	\end{enumerate}
\end{corollary}

It is now straightforward to verify the requirements of Section \ref{sec-ate}. Assumption \ref{ATE RATES} requires 
\[(n^{-1} s_d \log(p \vee n)^{3/2 + \delta_\D}   + (\bias^d)^2 ) (  n^{-1} s_y \log(p \vee n)^{3/2 + \delta_Y}  + ( \bias^y)^2  ) = o\left(n^{-1}\right).\]
Under the common assumption that $\bias = O(\sqrt{s/n})$, we require $s_d s_y \log(p \vee n)^{3 + \delta_\D + \delta_Y} = o(n)$. Both this, and the display above, clearly show how the sparsity and smoothness of the two functions interact due to the double robustness. Assumption \ref{ATE union} can be verified similarly.

These rates of convergence (i.e. part 2 of each corollary) are optimal up to factor $\log(p \vee n)^{1/2 + \delta}$. At heart, this loss appears to stem from the maximal inequality used to establish the concentration probability of \eqref{probability}. In practice, this is unlikely to be a limitation. As mentioned above, the use of group lasso can yield improvements in the constants if the data obey a grouped sparsity pattern, as is expected for treatment effects data, and may even yield improvements in the detection of the sparse signal, further offsetting the suboptimal $\log$ factor (see for example \citeasnoun{Lounici-etal2011_AoS} or \citeasnoun{Obozinski-Wainwright-Jordan2011_AoS}). Alternative methods could, in principle, yield a rate improvement. Chief among these would be lasso-penalized linear probability models (see also Remark \ref{remark-LPM}) or separate logistic regressions. The group lasso approach adopted here reflects common practice, and so it may be preferred. In any case, the $\log$ factors do not impact the treatment effect inference.

%%%%%%%%%%%%%%%%%%%%%%%%%%%%
%%%%%%%%%%%%%%%%%%%%%%%%%%%%
\section{Numerical and Empirical Evidence}
	\label{sec-data}

%%%%%%%%%%%%%%%%%%%%%%%%%%%%
\subsection{Simulation Study}
	\label{sec-simuls}

We conducted a Monte Carlo exercise to study how our estimator behaves as the propensity score and regression functions change, and the model selection problem becomes more or less difficult.\footnote{The supplemental appendix contains the additional results.} For simplicity we focus on the average effect of a binary treatment. We generated 1000 observations $(y_i, \d, x_i')'$ from the models in Example \ref{eg-approx sparse}, using both $p=1000$ and $p=1500$. The covariates include an intercept, with the remainder drawn from $N(0,\Sigma)$, with covariance $\Sigma[j_1,j_2] = 2^{-|j_1 - j_2|}, 2 \leq j_1, j_2 \leq p$. Errors are standard Normal. The crucial aspects of the DGP are the coefficient vectors $\beta_0^0$, $\beta_1^0$, and $\gamma^0$, which are defined to vary with the positive scalars $\rho_\beta$, $\rho_\gamma$, $\alpha_\beta$, and $\alpha_\gamma$, as follows:
	\begin{gather*}
		\beta_0^0 = \rho_\beta (-1,1,-1, 2^{-\alpha_\beta}, -3^{-\alpha_\beta}, \ldots, j^{-\alpha_\beta}, \ldots, p^{-\alpha_\beta})',\\
		\gamma^0 = \rho_\gamma (1,-1,1, -2^{-\alpha_\gamma}, 3^{-\alpha_\gamma}, \ldots, j^{-\alpha_\gamma}, \ldots, -p^{-\alpha_\gamma})',
	\end{gather*}
with $\beta_1^0 = - \beta_0^0$. The $\rho$ multipliers affect the signal-to-noise ratio, but not the sparsity. For smaller values distinguishing the large and small coefficients is more difficult for a given sample. The exponents $\alpha$ control the sparsity, where a sparse representation is not possible for small values.

Figure \ref{fig-plot-main} shows the empirical coverage rates of 95\% confidence intervals for $\mu_1 - \mu_0$ for different DGPs, for $p=1000$ and $1500$. Panels (a) and (c) show coverage as the multipliers $\rho_\beta$ and $\rho_\gamma$ range over 0.01 (weak signal) to 1 (strong), with $\alpha_\beta = \alpha_\gamma=2$. Panels (b) and (d) vary the sparsity exponents $\alpha_\beta$ and $\alpha_\gamma$ over 1/8 (not sparse) to 4 (very sparse), with $\rho_\beta = \rho_\gamma = 1$. Of 1000 observations total, the (mean) size of the comparison group declines from roughly 500 to 300 as $\rho_\gamma$ increases and 450 to 300 as $\alpha_\gamma$ increases, over their given ranges. Coverage is accurate over all signal strengths, and breaks down only when neither $\mut(x_i)$ nor $p_\t(x_i)$ is sparse, which is exactly when Assumption \ref{ATE RATES} (or condition (ii) of Theorem \ref{thm-ate overview}) cannot be satisfied. Note that coverage accuracy is retained when only one function is sparse, showcasing the double-robustness property.

The penalty parameters $\lambda_\D$ and $\lambda_Y$ are chosen using the iterative procedure described in Section \ref{sec-lambda}, with $\delta_\D = 4.5$ and $\delta_Y = 5$ throughout. Different DGPs exhibit different sensitivity to these values. Results using penalties chosen via 10-fold cross-validation appear in Figure \ref{fig-plot-cross}, which also exhibits excellent coverage across all sparse designs.\footnote{The R routines appear unstable for nonsparse designs, thus the analogues to Panels (b) and (d) of Figure \ref{fig-plot-main} are omitted. See the supplement for limited versions. This will be explored for future software development.}

%%%%%%%%%%%%%%%%%%%%%%%%%%%%
\subsection{Empirical Application}
	\label{sec-lalonde}

To illustrate the role that model selection can play in a real-world application, we revisit the National Supported Work (NSW) demonstration. The NSW has been analyzed numerous times since \citeasnoun{LaLonde1986_AER}. Our aim is a simple study of model selection, not a comprehensive or conclusive evaluation of the NSW. We focus on the subsample used by \citeasnoun{Dehejia-Wahba1999_JASA} and the Panel Study of Income Dynamics (PSID) comparison sample, taking as given their data definitions, sample selection, and trimming rules. Detailed discussion of these choices, and the NSW program may be found in Dehejia and Wahba \citeyear{Dehejia-Wahba1999_JASA,Dehejia-Wahba2002_REStat} (hereafter DW99 and DW02) and \citeasnoun{Smith-Todd2005_JoE}, and references therein. Briefly, the outcome of interest is earnings following a job training program. The dataset includes a treatment indicator, post-treatment earnings (1978), two years of pre-treatment earnings (1974\footnote{This naming follows DW99, but the variable may be measured outside 1974, see discussion in the works cited.} and 1975), as well as age, education, a marital status, and indicators for Black and Hispanic. Thus, $X$ consists of seven variables. We will keep the estimator fixed: all estimates will be based on the doubly-robust estimator with standard errors from Section \ref{sec-ate theory}. We will compare the following specifications for $X^*$:
\begin{enumerate}

	\item {\bf No Selection}: $X$,  (earn1974)$^2$, (earn1975)$^2$, (age)$^2$, and (educ)$^2$;

	\item {\bf Informally Selected:} The above, plus $\indicator$\{educ$<$HS\}, $\indicator$\{earn1974=0\}, $\indicator$\{earn1975=0\}, and ($\indicator$\{earn1974=0\}$\times$Hispanic). This specification was selected by DW02 using an informal balance test.

	\item {\bf Group Lasso Selection:} $X$, $\indicator$\{educ$<$HS\}, $\indicator$\{earn1974=0\}, $\indicator$\{earn1975=0\}, all possible first-order interactions, and all polynomials up to order five of the continuous covariates (age, educ, earn1974, earn1975).

\end{enumerate}
For specifications 1 and 2, the same covariates are in the outcome and treatment models. All specifications include an intercept and we include education and pre-treatment income in the refitting step following model selection. We follow DW99 and DW02 and trim comparisons with estimated propensity score larger (smaller) than the maximum (minimum) in the treated sample.\footnote{A formal treatment of trimming is beyond the scope of the present study. The goal of our analysis is illustrative, and hence we take DW99's trimming as given. This issue is discussed by DW99, DW02, and \citeasnoun{Smith-Todd2005_JoE}.}

Table \ref{table-lalonde} presents results from these three specifications, and includes the experimental arm of the NSW. The group lasso based estimate performs very well: the point estimate is accurate and the interval is tight. Selecting from 171 possible covariates allows for a great deal of flexibility, but the sparsity of the estimate keeps the variance well-controlled. The no-selection point estimate is accurate, but fails to yield significance, while the specification of DW02 yields a significant, but overly high estimate and wide confidence interval. The benefits of explicit model selection are clear.

%%%%%%%%%%%%%%%%%%%%%%%%%%%%
%%%%%%%%%%%%%%%%%%%%%%%%%%%%
\section{Discussion}
	\label{sec-conclusion}

This paper proposed a method that achieves uniformly valid inference on mean effects of a multivalued treatment even after model selection among possibly more covariates than observations. We demonstrated robustness to model selection errors, misspecification, and heterogeneous effects in observables. To accomplish this, a doubly-robust estimator was employed and shown to have excellent properties following model selection. We proved new results on group lasso estimation, which we argue is natural for treatment effects data. Multinomial logistic regression was studied in some detail. Numerical evidence shows that our method is quite promising for applications. 

A key outstanding question in this work and in the high-dimensional, sparse modeling literature more generally, is penalty parameter choice. Very little work has been done in this area, which is a crucial gap in implementability of these techniques. We plan to develop a formal choice for the penalty parameter that is appropriately optimal. Tuning parameter selection in semi- and nonparametric analysis, and its impact on estimation and inference, is becoming better understood, and parallel developments must take place in model selection contexts.

%%%%%%%%%%%%%%%%%%%%%%%%%%%%
%%%%%%%%%%%%%%%%%%%%%%%%%%%%
\begin{appendices}

%\small
\singlespacing

\numberwithin{lemma}{section}
\numberwithin{equation}{section}
\numberwithin{theorem}{section}

%%%%%%%%%%%%%%%%%%%%%%%%%%%%
%%%%%%%%%%%%%%%%%%%%%%%%%%%%
\section{Proofs for Treatment Effect Inference}
	\label{appx-ate proofs}

The proofs in this section are asymptotic. Order symbols hold for the sequence being considered, as a shorthand for the more formal versions given in e.g. Assumption \ref{first stage}. $C$ will denote a generic positive constant, which may be a matrix. Define the set of indexes $\It = \{i : \d = \t\}$. The online supplement contains much greater detail. We make frequent use of the linearization
\begin{equation}
	\label{appx-linearization}
	\frac{1}{a} = \frac{1}{b} + \frac{b - a}{ab} =  \frac{1}{b} + \frac{b - a}{b^2} + \frac{(b - a)^2}{ab^2}.
\end{equation}

\begin{proof}[Proof of Theorem \ref{thm-ate consistency}.]
SEE SUPPLEMENTAL APPENDIX.
\end{proof}

\begin{proof}[Proof of Theorem \ref{thm-ate linearity} without Additional Randomness.]
With $\psi_\t (\cdot)$ defined in \Eqref{eqn-ate moments}, we have $ \sqrt{n} ( \muthat - \mut ) = \sqrt{n} \En[ \psi_\t (y_i, \dt, \mut(x_i), p_\t(x_i), \mut) ] + R_1 + R_2$, where
\[R_1 = \frac{1}{\sqrt{n}} \sumi \dt (y_i - \mut(x_i)) \left( \frac{1}{ \hat{p}_\t(x_i) }  -  \frac{1}{ p_\t(x_i) } \right)\]
and
\[ R_2 = \frac{1}{\sqrt{n}} \sumi (\muthat(x_i) - \mut(x_i)) \left( 1 - \frac{ \dt }{ \hat{p}_\t(x_i) } \right).\]
The proof proceeds by showing that both $R_1$ and $R_2$ are $o_{P_n}(1)$. Applying the first equality in \Eqref{appx-linearization}, we rewrite $R_1$ as 
	\[R_1 = \frac{1}{\sqrt{n}} \sumi \dt u_i \left( \frac{p_\t(x_i) - \hat{p}_\t(x_i)}{ \hat{p}_\t(x_i) p_\t(x_i)}  \right).\]
Applying Assumptions \ref{overlap} and \ref{fourth moments} and the first-stage consistency condition of Assumption \ref{consistent}:
\[ \E\left[ R_1^2 \vert \{x_i, \d\}_{i = 1}^n \right]  = \En \left[ \frac{ \dt \sigma_t^2(x_i) }{\hat{p}_\t(x_i)^2 p_\t(x_i)^2} \left( p_\t(x_i) - \hat{p}_\t(x_i) \right)^2 \right] \leq C \En[(p_\t(x_i) - \hat{p}_\t(x_i))^2] = o_{P_n}(1). \]

Next, again using \Eqref{appx-linearization}, we have $R_{2} = R_{21} + R_{22}$, where
\[R_{21} =  \frac{1}{\sqrt{n}} \sumi (\muthat(x_i) - \mut(x_i)) \left(\frac{ p_\t(x_i)  -  \dt }{ p_\t(x_i) } \right) \]
and
\[R_{22} = \frac{1}{\sqrt{n}} \sumi (\muthat(x_i) - \mut(x_i))(\hat{p}_\t(x_i)  -  p_\t(x_i))\left(  \frac{   \dt }{ \hat{p}_\t(x_i) p_\t(x_i) }  \right)  .   \]
For the first term, $ R_{21} = \sqrt{n}  \En[ (\muthat(x_i) - \mut(x_i)) (1 - d_i^t/ p_\t(x_i))] = o_{P_n}(1)$ by Assumption \ref{new}. Next, 
\[  |R_{22}| \leq \sqrt{n} \left( \maxi \frac{  1 }{ \hat{p}_\t(x_i) p_\t(x_i) } \right)\sqrt{ \En[ (\muthat(x_i) - \mut(x_i))^2]\En[ (\hat{p}_\t(x_i) - p_\t(x_i))^2]}	 = o_{P_n}(1).\]
	by H\"older's inequality, Assumption \ref{overlap} and the rate condition of Assumption \ref{ATE RATES}.
\end{proof}

\begin{proof}[Proof of Theorem \ref{thm-ate linearity} with Additional Randomness.]
We must reconsider the remainders $R_1$ and $R_2$. For the former, applying \Eqref{appx-linearization}, we find $R_1 = R_{11} + R_{12}$, where
\[R_{11} = \frac{1}{\sqrt{n}} \sumi \frac{\dt u_i}{p_\t(x_i)^2} \left( p_\t(x_i) - \hat{p}_\t(x_i)\right) 	\qquad \text{ and } \qquad     	 R_{12} =  \frac{1}{\sqrt{n}} \sumi \frac{\dt u_i}{p_\t(x_i)^2 \hat{p}_\t(x_i) } \left( \hat{p}_\t(x_i)   -  p_\t(x_i)\right)^2. \]
For $R_{11}$, we first add and subtract the parametric representation to get $R_{11} = R_{111} + R_{112}$, where, 
\[R_{111} = \frac{1}{\sqrt{n}} \sumi \frac{\dt u_i}{p_\t(x_i)^2} \left( \pttrue - \pthat\right) 	\quad \text{ and }\]
\[	R_{112} = \frac{1}{\sqrt{n}} \sumi \frac{\dt u_i}{p_\t(x_i)^2} \left( p_\t(x_i) - \pttrue\right).\]

By a two-term mean-value expansion $R_{111} = R_{111a} + R_{111b}$, with
\[R_{111a} = \frac{1}{\sqrt{n}} \sumi \frac{\dt u_i}{p_\t(x_i)^2} \sumT \left\{\pttrue ( 1- \pttrue) \left({x_i^*}'(\ghatt - \gtruet) \right) \right\} \]
\[\text{and} \qquad \qquad     R_{111b} = \frac{1}{2 \sqrt{n}} \sumi \frac{\dt u_i}{p_\t(x_i)^2} v_i' \bar{\H} v_i,\]
where $v_i  = \{ {x_i^*}'(\ghatt - \gtruet)\}_{\NT}$ and $\overline{\H}  =  \H(\{{x_i^*}'\gtruet  +  m_\t {x_i^*}'\ghatt\}_{\NT})$ for appropriate scalars $m_\t$ and the $\T$-square Hessian matrix $\H(\{{x_i^*}'\ganyt\}_{\NT})$ (defined in Appendix \ref{appx-mlogit proofs}).

For $R_{111a}$, consider each term in the sum over $\NT$ one at a time; let $R_{111a} = \sumT R_{111a}(\t)$. Let $\t'$ denote the original treatment under consideration. Define \[\Sigma_{\t,j} = \E\left[ (x_{i,j}^*)^2 \sigma_{\t'}^2(x_i)  \pttrue^2 ( 1 - \pttrue)^2 / p_{\t'}(x_i)^3 \right].\] Then proceed as follows
\begin{align*}
	 R_{111a}(\t)  & =  \sum_{j \in \hat{S}_\D} \left\{ \frac{1}{\sqrt{n}} \sumi \left(x_{i,j}^* \frac{\d^{\t'} u_i \pttrue ( 1- \pttrue)}{p_{\t'}(x_i)^2 \Sigma_{\t,j}^{1/2}}  \right) \right\} \Sigma_{\t,j}^{1/2}  (\ghattj - \gtruetj)		\\
	& \leq \left( \maxj \Sigma_{\t,j}^{1/2} \right) \left( \maxj \frac{1}{\sqrt{n}} \sumi x_{i,j}^* \frac{\d^{\t'} u_i \pttrue ( 1- \pttrue)}{p_{\t'}(x_i)^2 \Sigma_{\t,j}^{1/2}}    \right) \left\| \ghatt - \gtruet \right\|_1		\\
	& = O(1) O_{P_n}( \log(p) )  \left\| \ghatt - \gtruet \right\|_1= o_{P_n}(1).
\end{align*}
Convergence follows under Assumption \ref{ATE union}. For the penultimate equality, it follows from Assumptions \ref{overlap}, \ref{bounded}, and \ref{fourth moments} that $\maxj \Sigma_{t,j} = O(1)$. Finally, the center factor is shown to be $O_{P_n}( \log(p) )$ by applying the moderate deviation theory for self-normalized sums of \citeasnoun[Theorem 7.4]{delaPena-Lai-Shao2009_book} and in particular \citeasnoun[Lemma 5]{BCCH2012_Ecma}. To apply this lemma, first note that the summand of the center factor has bounded third moment and second moment bounded away from zero, from Assumptions \ref{overlap}, \ref{bounded}, \ref{fourth moments}, and the requirements of Assumptions \ref{first stage} and \ref{ATE union}. $ \Sigma_{t,j}$ normalizes the second moment, and the lemma applies under Assumptions \ref{sparsity} and the first restriction of Assumption \ref{ATE union}.

For $R_{111b}$, the results of \citeasnoun{Tanabe-Sagae1992_JRSSB} coupled with Assumption \ref{first stage} give $v_i' \bar{\H} v_i \leq C \| v_i \|_2^2$. Thus, using Assumption \ref{overlap} to bound $\maxi p_\t(x_i)^{-2} < C$, we find $R_{111b}$ may be bounded as follows:
\begin{align*}
	|R_{111b}| & \leq C \sumT  \sqrt{n} (\maxt |u_i| )  \En \left[ |{x_i^*}'(\ghatt - \gtruet)|^2\right] 		\\
	& \leq C \T \maxT \left| \sqrt{n} (\maxt |u_i| )  \En \left[ | \pthat - \pttrue |^2 \right] \right|= o_{P_n}(1),
\end{align*}
by the union bound and Assumption \ref{ATE union}, using the Assumptions \ref{overlap} and \ref{consistent} to apply \Eqref{appx-mlogit rate mvt} with the inequality reversed.

A variance bound may be applied to $R_{112}$ as in the previous proof, and we have $|R_{112}| = O_{P_n}(\bias) = o_{P_n}(1)$ by Markov's inequality.

Next, $R_{12}$ is simply bounded by
\begin{align*}
	|R_{12}| & \leq \sqrt{n} (\maxt |u_i| ) \left( \maxt \frac{1}{p_\t(x_i)^2 \hat{p}_\t(x_i) } \right) \En \left[ \left( \hat{p}_\t(x_i)   -  p_\t(x_i)\right)^2 \right]  		\\
	& \leq O_{P_n}(1) \sqrt{n} (\maxt |u_i| ) \En \left[ \left( \hat{p}_\t(x_i)   -  p_\t(x_i)\right)^2 \right] = o_{P_n}(1),
\end{align*}
where the rate follows from Assumptions \ref{overlap}, \ref{dgp asmpts}, and \ref{first stage}, and this tends to zero by Assumption \ref{ATE union}.

As in the prior proof, write $R_2 = R_{21} + R_{22}$. The same bound is used for $R_{22}$. However, for $R_{21}$, add and subtract the pseudotrue values to get $R_{21} = R_{211} + R_{212}$, where
\[R_{211} =  \frac{1}{\sqrt{n}} \sumi ({x_i^*}'\bhatt - {x_i^*}\btruet) \left(\frac{ p_\t(x_i)  -  \dt }{ p_\t(x_i) } \right)  	 \qquad \text{ and } \qquad   	R_{212} =  \frac{1}{\sqrt{n}} \sumi ({x_i^*}\btruet - \mut(x_i)) \left(\frac{ p_\t(x_i)  -  \dt }{ p_\t(x_i) } \right) \]
For the first term, define $\tilde{\Sigma}_{\t,j} = \E\left[ (x_{i,j}^*)^2 (\dt - p_\t(x_i))^2 / p_\t(x_i)^2 \right]$ and then proceed as follows:
\begin{align*}
	 R_{211}  & = \frac{1}{\sqrt{n}} \sumi \left(\frac{ p_\t(x_i)  -  \dt }{ p_\t(x_i) } \right) \sum_{j \in \hat{S}_Y} x_{i,j}^*  (\bhattj - \btruetj)		\\
	& =  \sum_{j \in \hat{S}_Y} \left\{ \frac{1}{\sqrt{n}} \sumi \frac{x_{i,j}^*  (p_\t(x_i) - \dt) / p_\t(x_i)}{  \tilde{\Sigma}_{\t,j}^{1/2} } \right\} \tilde{\Sigma}_{\t,j}^{1/2}  (\bhattj - \btruetj)		\\
	& \leq \left( \maxj \tilde{\Sigma}_{\t,j}^{1/2} \right) \left( \maxj \frac{1}{\sqrt{n}} \sumi \frac{x_{i,j}^*  (p_\t(x_i) - \dt) / p_\t(x_i)}{  \tilde{\Sigma}_{\t,j}^{1/2} }   \right) \left\| \bhatt - \btruet \right\|_1		\\
	& = O(1) O_{P_n}( \log(p) )  \left\| \bhatt - \btruet \right\|_1 = o_{P_n}(1),
\end{align*}
where the final line follows exactly as above. A variance bound may be applied to $R_{212}$ as in the previous proof, and we have $|R_{212}| = O_{P_n}(\bias) = o_{P_n}(1)$ by Markov's inequality.
\end{proof}

\begin{proof}[Proof of Theorem \ref{thm-ate normality}.]
	This follows from the prior result and Assumption \ref{ATE moments}.
\end{proof}

\begin{proof}[Proof of Theorem \ref{thm-ate variance}.]
We begin with $\hat{V}_{W}(\t)$. Expanding the square and using \Eqref{appx-linearization}, rewrite $\hat{V}_{\drf}^W(\t) = \En[ \dt u_i^2 p_\t(x_i) ^{-2}] + R_{W,1} + R_{W,2} + R_{W,3}$ where
\begin{gather*}
	R_{W,1} = \En\left[ \frac{\dt u_i^2}{\hat{p}_\t(x_i)^2 p_\t(x_i)^2} \left( \hat{p}_\t(x_i) - p_\t(x_i)\right)\left( \hat{p}_\t(x_i) + p_\t(x_i)\right) \right],		\\
	R_{W,2} = \En\left[ \frac{\dt (\mut(x_i)  -  \muthat(x_i))^2 }{\hat{p}_\t(x_i)^2} \right],	\qquad \text{and} \qquad R_{W,3} = 2 \En\left[ \frac{\dt u_i (\mut(x_i)  -  \muthat(x_i)) }{\hat{p}_\t(x_i)^2} \right].
\end{gather*}
Using H\"older's inequality, Assumptions \ref{overlap}, \ref{ATE moments}, and \ref{consistent}, we have the following
	\[R_{W,1}  \leq \left( \maxt \frac{  \hat{p}_\t(x_i)  +  p_\t(x_i) }{ \hat{p}_\t(x_i)^2 p_\t(x_i)^2 } \right) \En[\dt |u_i|^4]^{1/2} \En[ \dt (\hat{p}_\t(x_i) - p_\t(x_i))^2]^{1/2} = o_{P_n}(1), 		\]
	\[R_{W,2}  \leq \left( \maxt \frac{  1 }{ \hat{p}_\t(x_i)^2 } \right) \En[ \dt (\muthat(x_i) - \mut(x_i))^2] = o_{P_n}(1),\]
	\[\text{and,} \qquad \qquad     R_{W,3}  \leq 2 \left( \maxt \frac{  1 }{ \hat{p}_\t(x_i)^2 } \right) \En[\dt |u_i|^2]^{1/2}  \En[ \dt (\muthat(x_i) - \mut(x_i))^2]^{1/2} = o_{P_n}(1),\]
where $ \En[|u_i|^4] = O_{P_n}(1)$ from the inequality of \citeasnoun{vonBahr-Esseen1965_AoMS}. From the same inequality it follows that $\En[ \dt u_i^2 p_\t(x_i) ^{-2}] - V_{\drf}^W(\t)| = o_{P_n}(1)$, under Assumptions \ref{overlap} and \ref{fourth moments}.

Next consider the ``between'' variance estimator, $\hat{V}_{\drf}^B$. For any $\t \NTbar$ and $\t' \in \NTbar$, define 
	\[R_{B,1}(\t, \t') =  \En\left[ (\muthat(x_i) - \mut(x_i)) (\hat{\mu}_{\t'}(x_i) - \mu_{\t'}(x_i)) \right],\]
	\[R_{B,2}(\t, \t') =  \muthat \En\left[ \hat{\mu}_{\t'}(x_i) - \mu_{\t'}(x_i) \right],   \quad \text{and} \quad  R_{B,3}(\t, \t') =  \En\left[ \mu_{\t}(x_i) (\hat{\mu}_{\t'}(x_i) - \mu_{\t'}(x_i)) \right].\]
From H\"older's inequality, Assumption \ref{consistent}, Theorem \ref{thm-ate normality}, the von Bahr and Esseen inequality, and Assumptions \ref{fourth moments} and \ref{ATE moments} it follows that $R_{B,k}(\t,\t') = o_{P_n}(1)$ for $k \in \N_3$ and all pairs $(\t, \t') \in \N_\t^2$. With this in mind, we decompose 
\begin{align*}
	\hat{V}_{\drf}^B(\t,\t') & = \En\left[ \mut(x_i) \mu_{\t'}(x_i) \right] - \muthat \En\left[ \mu_{\t'}(x_i)\right]  - \hat{\mu}_{\t'} \En\left[ \mut(x_i)\right] +  \muthat \hat{\mu}_{\t'}		\\
	& \qquad + R_{B,1}(\t,\t') + R_{B,2}(\t, \t') + R_{B,2}(\t', \t) + R_{B,3}(\t, \t') + R_{B,3}(\t', \t).
\end{align*}
Consistency of $\hat{V}_{\drf}^B(\t,\t') $ now follows from the von Bahr and Esseen inequality and Theorem \ref{thm-ate normality}.
\end{proof}

\begin{proof}[Proof of Corollary \ref{thm-ate uniform}.]
	Suppose the result did not hold. Then, there would exist a subsequence $P_m \in \bm{P}_m$, for each $m$, such that 
	\[\lim_{m \to \infty} \left| \P_{P_m} \left[  G(\drf) \in \left\{ G(\drfhat) \pm c_\alpha \sqrt{\nabla_G(\drfhat)  \hat{V} \nabla_G'(\drfhat)  / n}\right\} \right] - (1 - \alpha) \right| >0 .\]
	But this contradicts Theorem \ref{thm-ate}, under which $ (\nabla_G(\drfhat)  \hat{V} \nabla_G'(\drfhat)/n)^{-1/2} ( G(\drfhat) - G(\drf))$ is asymptotically standard normal under the sequence $P_m$.
\end{proof}

%%%%%%%%%%%%%%%%%%%%%%%%%%%%
%%%%%%%%%%%%%%%%%%%%%%%%%%%%
\section{Proofs for Group Lasso Selection and Estimation of Multinomial Logistic Models}
	\label{appx-mlogit proofs}

This section is nonasymptotic. We use generic notation $X^*$, $\delta$, etc. The online supplement has greater detail.

%%%%%%%%%%%%%%%%%%%%%%%%%%%%
\subsection{Lemmas}

The following three lemmas are needed for the proofs of Theorems \ref{thm-mlogit} and \ref{thm-post mlogit}. Due to space considerations, only a short sketch of the proofs will be given, highlight the main ideas in each. Full details are available in the online supplement.

\begin{lemma}[Score Bound]
	\label{appx-mlogit score bound}
	For $\lambda_\D$ and $\mathcal{P}$ defined in \Eqref{lambda} and \Eqref{probability} we have
	\[\P \left[ \maxj \| \En [(p_\t(x_i) - \dt)  x_{i,j}^* ] \|_2 \geq \frac{\lambda_\D}{2} \right] \leq \mathcal{P}. \]
\end{lemma}
\begin{proof}
The residuals $v_{\t,i} = p_\t(x_i) - \dt$ are conditionally mean-zero by definition and satisfy $\E[v_{\t,i}^2 \vert x_i ] \leq 1$. Using this, Assumption \ref{iid}, and the definition of $\mathcal{X}$, we find that $\E\left[ \| \En [v_{\t,i} x_{i,j}^* ] \|_2^2 \right]  \leq \mathcal{X}^2 \T / n$, uniformly in $j \in \Np$. Define the mean-zero random variables $\xi_{t,j} = (\En[ v_{\t,i} x_{i,j}^* ])^2 - \frac{1}{n} \E[V_\t^2 {X_j^*}^2]$ and set $r_n = \T^{-1/2} \log(p \vee n)^{3/2 + \delta}$. Then
\begin{align*}
	\P  \left[ \maxj \| \En [(p_\t(x_i) - \dt)  x_{i,j}^* ] \|_2 \geq \frac{\lambda_\D}{2} \right]
	& \leq \P \left[ \maxj \sumT \xi_{t,j}  \geq  \frac{\mathcal{X}^2  \T r_n }{n} \right]	  
	 \leq \E \left[ \maxj \left| \sumT \xi_{t,j} \right|\right]   \frac{n}{ \mathcal{X}^2  \T r_n }
\end{align*}
where final line follows from Markov's inequality. Next, applying Lemma 9.1 of \citeasnoun{Lounici-etal2011_AoS}, Jensen's inequality, and Assumption \ref{fourth moments}, we find that
	\[\E \left[ \maxj \left| \sumT \xi_{t,j} \right|\right]  \leq  4 \log(2p)^{1/2} \left(  \sumT  \frac{\mathcal{X}^4 }{n^2} + \sumT \E \left[ \maxj \left| \En[ v_{\t,i} {x_{i,j}^*} ] \right| ^4  \right] \right)^{1/2}.\]
Again using Lemma 9.1 of \citeasnoun{Lounici-etal2011_AoS}, and Assumptions \ref{iid} and \ref{bounded}, we bound the expectation in the second term above as follows:
	\[\E \left[ \maxj  \left| \En [ v_{\t,i} {x_{i,j}^*} ] \right| ^4  \right]  	 \leq \frac{64 \log(12p)^2 \mathcal{X}^4 }{n^2}.\]
Collecting these results proves the Lemma.
\end{proof}

\begin{lemma}[Estimate Sparsity]
	\label{appx-mlogit S-hat}
	With probability at least $1 - \mathcal{P}$
	\[  | \tilde{S}^\D |  \leq  \frac{4}{\lambda_\D^2} \overline{\phi}\{Q, \tilde{S}^\D \}  \sumT   \En \left[ (\pttilde - p_\t(x_i))^2 \right] .\]
\end{lemma}
\begin{proof}
From the Karush-Kuhn-Tucker conditions for \eqref{grplasso}, for all $\t \in \NT$, if $\gtildej \neq 0$ it must satisfy
	\[\En[{x_{i,j}^*} (\pttilde - \dt)] = \lambda_\D \frac{\gtildetj}{\|\gtildej \|_2}.\]
Taking the $\ell_2$-norm over $\t \in \NT$ for fixed $j \in \tilde{S}^\D$, adding and subtracting the true propensity score, using the triangle inequality, the score bound \eqref{appx-mlogit score bound}, collecting terms, squaring both sides, and summing over $j \in \tilde{S}^\D$ (i.e. applying $\| \cdot \|_2^2$ over $j \in \tilde{S}^\D$ to both sides) yields
	\[ \sum_{j \in \tilde{S}^\D} \lambda_\D^2  \leq 4 \sum_{j \in \tilde{S}^\D} \sumT \En[{x_{i,j}^*}( \pttilde  -  p_\t(x_i) ) ]^2 \leq 4 \overline{\phi}\{Q, \tilde{S}^\D \}  \sumT   \En \left[ (\pttilde  -  p_\t(x_i))^2 \right]. \]
The result now follows, as the left-hand side is equal to $|\tilde{S}^\D| \lambda_\D^2$.
\end{proof}

\begin{lemma}[Bounds in $\ell_2/\ell_1$ norm]
	\label{appx-mlogit cone}
	With probability $1 - \mathcal{P}$ the vector $\dtilde = \gtilde - \gtrue$ satisfies $\vertiii{\dtilde} \leq 5 a_n$ and $\vertiii{ \tilde{\delta}_{\newdot, S_*}} \leq a_n$ where $a_n := \max \left\{   \kappa_\D^{-1} \sqrt{|S_*|},  2 \lambda_\D^{-1} \bias^d \sqrt{\T} \right\}\En [ \| \{ {x_i^*}' \dtildet\}_{\NT}  \|_2^2]^{1/2}$.
\end{lemma}
\begin{proof}
By the Cauchy-Schwarz inequality and Lemma \ref{appx-mlogit score bound}, 
\begin{align}
	\sumT \En \left[(p_\t(x_i) - \dt ) {x_i^*}' \dtildet \right]   
	& \leq \sumj \sqrt{\sumT \En \left[(p_\t(x_i) - \dt) {x_{i,j}^*} \right]^2} \sqrt{\sumT \dtildetj^2}		\nonumber \\
	& \leq \maxj \left\{ \left\| \En \left[(p_\t(x_i) - \dt) {x_{i,j}^*} \right] \right\|_2 \right\}  \sumj  \left\| \dtildej  \right\|_2			
	& \leq \frac{\lambda_\D}{2} \vertiii{ \dtilde }, 	\label{appx-mlogit score variance}
\end{align}
with probability at least $1 - \mathcal{P}$. Applying the Cauchy-Schwarz inequality, the bias condition of Assumption \ref{sparsity}, and Cauchy-Schwarz again yields
\begin{align}
	\sumT \En \left[(\pttrue - p_\t(x_i)) {x_i^*}' \dtildet \right]    
	& \leq    \sumT \En \left[(\pttrue - p_\t(x_i) )^2\right]^{1/2}  \En \left[({x_i^*}' \dtildet)^2 \right]^{1/2}		\nonumber \\
	& \leq    \bias^d \sumT  \En \left[({x_i^*}' \dtildet)^2 \right]^{1/2}		\nonumber  \\
	& \leq     \bias^d \sqrt{\T} \En [ \| \{ {x_i^*}' \dtildet\}_{\NT}  \|_2^2]^{1/2}.		\label{appx-mlogit score bias} 
\end{align}

Combining Equations \eqref{appx-mlogit score variance} and \eqref{appx-mlogit score bias}, we have, probability at least $1 - \mathcal{P}$,
\begin{align}
	\sumT \En \left[(\pttrue - \dt ) {x_i^*}' \dtildet \right] &  = \sumT \En \left[(p_\t(x_i) - \dt ) {x_i^*}' \dtildet \right] 		\nonumber \\
		& \qquad  + \sumT \En \left[(\pttrue - p_\t(x_i)) {x_i^*}' \dtildet \right]       	\nonumber \\
	& \leq  \frac{\lambda_\D}{2} \vertiii{ \dtilde } + \bias^d \sqrt{\T} \En [ \| \{ {x_i^*}' \dtildet\}_{\NT}  \|_2^2]^{1/2}.	\label{appx-mlogit score}
\end{align}

By the optimality of $\dtilde$, 
$\M(\gtrue + \dtilde) + \lambda_\D \vertiii{\gtrue + \dtilde} \leq  \M(\gtrue) + \lambda_\D \vertiii{\gtrue}$,
and so
	\[\lambda_\D \left\{ \vertiii{\gtrue}  -  \vertiii{\gtrue + \dtilde}  \right\}  \geq \M(\gtrue + \dtilde) -   \M(\gtrue)		
	 \geq \sumT \En \left[(\pttrue - \dt) {x_i^*}' \dtildet \right],\]
applying the convexity of $\M$. Using the bound in \Eqref{appx-mlogit score} and rearranging we find that
	\[0  \leq  \lambda_\D \left\{ \vertiii{\gtrue}  -  \vertiii{\gtrue + \dtilde}  \right\} + \frac{\lambda_\D}{2} \vertiii{ \dtilde } + \bias^d \sqrt{\T} \En [ \| \{ {x_i^*}' \dtildet\}_{\NT}  \|_2^2]^{1/2}.\]
Dividing through $\lambda_\D$ and decomposing the supports, we find that
\begin{align*}
	0 & \leq  \frac{1}{2} \vertiii{ \dtilde }    +  \left\{ \vertiii{\gtrue} -  \vertiii{\gtrue + \dtilde} \right\} +  \frac{\bias^d \sqrt{\T}}{\lambda_\D} \En [ \| \{ {x_i^*}' \dtildet\}_{\NT}  \|_2^2]^{1/2}		\\
	& = \frac{1}{2} \vertiii{ \tilde{\delta}_{\newdot, S_*}} +  \frac{1}{2} \vertiii{ \tilde{\delta}_{\newdot, S_*^c}} + \vertiii{\gamma^*_{\newdot, S_*}} - \vertiii{\gamma^*_{\newdot, S_*} + \tilde{\delta}_{\newdot, S_*}} - \vertiii{ \tilde{\delta}_{\newdot, S_*^c}} +  \frac{\bias^d \sqrt{\T}}{\lambda_\D} \En [ \| \{ {x_i^*}' \dtildet\}_{\NT}  \|_2^2]^{1/2},
\end{align*}
because $\gamma^*_{\newdot, S_*^c} = 0$. Collecting terms and applying the triangle inequality yields
\begin{align*}
	\frac{1}{2} \vertiii{ \tilde{\delta}_{\newdot, S_*^c}} & \leq \frac{1}{2} \vertiii{ \tilde{\delta}_{\newdot, S_*}} + \left| \vertiii{\gamma^*_{\newdot, S_*}} - \vertiii{\gamma^*_{\newdot, S_*} + \tilde{\delta}_{\newdot, S_*}}\right|  +  \frac{\bias^d \sqrt{\T}}{\lambda_\D} \En [ \| \{ {x_i^*}' \dtildet\}_{\NT}  \|_2^2]^{1/2}  		\\
	& \leq \frac{1}{2} \vertiii{ \tilde{\delta}_{\newdot, S_*}} +  \vertiii{\gamma^*_{\newdot, S_*} - \left( \gamma^*_{\newdot, S_*} + \tilde{\delta}_{\newdot, S_*} \right)}  +  \frac{\bias^d \sqrt{\T}}{\lambda_\D} \En [ \| \{ {x_i^*}' \dtildet\}_{\NT}  \|_2^2]^{1/2}	  		\\
	& =  \frac{1}{2} \vertiii{ \tilde{\delta}_{\newdot, S_*}} +  \vertiii{ \tilde{\delta}_{\newdot, S_*}} +  \frac{\bias^d \sqrt{\T}}{\lambda_\D} \En [ \| \{ {x_i^*}' \dtildet\}_{\NT}  \|_2^2]^{1/2}.
\end{align*}
Therefore with probability at least $1 - \mathcal{P}$
\begin{equation}
	\label{appx-mlogit cases}
	\vertiii{ \tilde{\delta}_{\newdot, S_*^c}} \leq  3 \vertiii{ \tilde{\delta}_{\newdot, S_*}} +  \frac{2 \bias^d \sqrt{\T}}{\lambda_\D} \En [ \| \{ {x_i^*}' \dtildet\}_{\NT}  \|_2^2]^{1/2}.
\end{equation}

Consider two cases based on the upper bound in \eqref{appx-mlogit cases}. First, suppose that $\dtilde$ obeys the cone constraint of \Eqref{RE full} in the definition of $\kappa_\D^2$. This implies
\begin{align}
	\vertiii{\dtilde} & \leq 5 \vertiii{ \tilde{\delta}_{\newdot, S_*}}  \leq 5 \sqrt{|S_*|} \left\| \tilde{\delta}_{\newdot, S_*}\right\|_2 			 \leq  \frac{ 5 \sqrt{|S_*|} }{\kappa_\D} \En [ \| \{ {x_i^*}' \dtildet\}_{\NT}  \|_2^2]^{1/2},     \label{appx-mlogit group norm 1}
\end{align}
by the Cauchy-Schwarz inequality, the restricted eigenvalue definition of \Eqref{RE full}, and noting that $\sumT \dtildet' Q \dtildet =  \En [ \| \{ {x_i^*}' \dtildet\}_{\NT}  \|_2^2]$. Collecting across the second and third inequalities yields 
\begin{equation}
	\label{appx-mlogit group norm 2}
	\vertiii{ \tilde{\delta}_{\newdot, S_*}} \leq \frac{\sqrt{|S_*|} }{\kappa_\D} \En [ \| \{ {x_i^*}' \dtildet\}_{\NT}  \|_2^2]^{1/2}.
\end{equation}

On the other hand, if the cone constraint fails, then $\vertiii{ \tilde{\delta}_{\newdot, S_*}}  <  \frac{1}{4} \vertiii{ \tilde{\delta}_{\newdot, S_*^c}}$. Using this for the first and third inequalities, and \Eqref{appx-mlogit cases} for the second, we have
\begin{align*}
	\vertiii{\dtilde}  \leq \frac{5}{4} \vertiii{ \tilde{\delta}_{\newdot, S_*^c}} 	&  \leq \frac{15}{4} \vertiii{ \tilde{\delta}_{\newdot, S_*}} +  \frac{5}{2} \frac{ \bias^d \sqrt{\T}}{\lambda_\D} \En [ \| \{ {x_i^*}' \dtildet\}_{\NT}  \|_2^2]^{1/2}		 \\
	& \leq \frac{15}{16}\vertiii{ \tilde{\delta}_{\newdot, S_*^c}}  +  \frac{5}{2} \frac{ \bias^d \sqrt{\T}}{\lambda_\D} \En [ \| \{ {x_i^*}' \dtildet\}_{\NT}  \|_2^2]^{1/2}
\end{align*}
Combining the right hand side of the first line with third lines yields $\vertiii{ \tilde{\delta}_{\newdot, S_*^c}}  \leq \linebreak 8 \bias^d \sqrt{\T} \En [ \| \{ {x_i^*}' \dtildet\}_{\NT}  \|_2^2]^{1/2} \big/ \lambda_\D$. 
Plugging this back into the last line we obtain the bound
\begin{equation}
	\label{appx-mlogit group norm 3}
	\vertiii{\dtilde}  \leq  10 \frac{ \bias^d \sqrt{\T}}{\lambda_\D} \En [ \| \{ {x_i^*}' \dtildet\}_{\NT}  \|_2^2]^{1/2},
\end{equation}
while instead, plugging it into the failure of the cone constraint yields
\begin{equation}
	\label{appx-mlogit group norm 4}
	\vertiii{ \tilde{\delta}_{\newdot, S_*}} \leq 2 \frac{ \bias^d \sqrt{\T}}{\lambda_\D} \En [ \| \{ {x_i^*}' \dtildet\}_{\NT}  \|_2^2]^{1/2}.
\end{equation}
Combining Equations \eqref{appx-mlogit group norm 1} and \eqref{appx-mlogit group norm 3} gives the first claim of the lemma and Equations \eqref{appx-mlogit group norm 2} and \eqref{appx-mlogit group norm 4} give the second. 
\end{proof}

%%%%%%%%%%%%%%%%%%%%%%%%%%%%
\subsection{Proof of Theorem \ref{thm-mlogit}}

Define $\dtilde = \gtilde - \gtrue$. By the optimality of $\dtilde$, we have
\[ \M(\gtrue + \dtilde)   + \lambda_\D \vertiii{\gtrue + \dtilde} \leq  \M(\gtrue) + \lambda_\D \vertiii{\gtrue}.\]
Rearranging and subtracting the score, we have
\begin{align}
	\begin{split}
		\label{appx-mlogit start}
		& \M(\gtrue + \dtilde)  -  \M(\gtrue)  -  \sumT \En \left[(\pttrue - \dt) {x_i^*}' \right] \dtildet		\\
		& \qquad \quad \leq  \lambda_\D \left\{ \vertiii{\gtrue}  -  \vertiii{\gtrue + \dtilde}  \right\}  -  \sumT \En \left[(\pttrue - \dt) {x_i^*}' \right] \dtildet.
	\end{split}
\end{align}
The proof proceeds by deriving a further upper bound to the right and a quadratic lower bound of the left. The combination of these will yield a bound on $\En[({x_i^*}'\dtildet )^2] ^{1/2}$. 

Begin with the right side of \Eqref{appx-mlogit start}. For the penalized difference of coefficients we have
$\vertiii{\gamma^*_{\newdot, S_*^c}} - \vertiii{\gamma^*_{\newdot, S_*^c} + \tilde{\delta}_{\newdot, S_*^c}} = \vertiii{ \tilde{\delta}_{\newdot, S_*^c}}$,
because $\gamma^*_{\newdot, S_*^c} = 0$. Therefore,
\begin{align*}
	\vertiii{\gtrue} - \vertiii{\gtrue + \dtilde} & = \vertiii{\gamma^*_{\newdot, S_*}} - \vertiii{\gamma^*_{\newdot, S_*} + \tilde{\delta}_{\newdot, S_*}} - \vertiii{ \tilde{\delta}_{\newdot, S_*^c}}		 \\
	& \leq \vertiii{\gamma^*_{\newdot, S_*}} - \vertiii{\gamma^*_{\newdot, S_*} + \tilde{\delta}_{\newdot, S_*}}		 \\
	& \leq \left| \vertiii{\gamma^*_{\newdot, S_*}} - \vertiii{\gamma^*_{\newdot, S_*} + \tilde{\delta}_{\newdot, S_*}} \right|		 \\
	& \leq \vertiii{\gamma^*_{\newdot, S_*} - \left(\gamma^*_{\newdot, S_*} + \tilde{\delta}_{\newdot, S_*}\right) }		 = \vertiii{ \tilde{\delta}_{\newdot, S_*}}, 
\end{align*}
where the first inequality reflects dropping the nonpositive final term (the norm is nonnegative) and the third inequality follows from the triangle inequality. Using this result for the first term and the bound \eqref{appx-mlogit score} for the second, the right side of \Eqref{appx-mlogit start} is bounded by
\begin{align}
	& \lambda_\D \vertiii{ \tilde{\delta}_{\newdot, S_*}}   +  \frac{\lambda_\D}{2} \vertiii{ \dtilde } + \bias^d \sqrt{\T} \En [ \| \{ {x_i^*}' \dtildet\}_{\NT}  \|_2^2]^{1/2} 		\nonumber \\
	& \qquad \leq \left(         \lambda_\D \left\{ \frac{  \sqrt{|S_*|} }{\kappa_\D}  \vee  \frac{ 2 \bias^d \sqrt{\T}}{\lambda_\D} \right\}         +        \frac{\lambda_\D}{2}\left\{ \frac{ 5 \sqrt{|S_*|} }{\kappa_\D}  \vee  \frac{ 10 \bias^d \sqrt{\T}}{\lambda_\D} \right\}       +      \bias^d \sqrt{\T}             \right)   \En [ \| \{ {x_i^*}' \dtildet\}_{\NT}  \|_2^2]^{1/2} 		\nonumber \\
	& \qquad \leq \left(    6 \frac{ \lambda_\D \sqrt{|S_*|} }{\kappa_\D}        +     8  \bias^d \sqrt{\T}             \right)   \En [ \| \{ {x_i^*}' \dtildet\}_{\NT}  \|_2^2]^{1/2},			\label{appx-mlogit upper} 	
\end{align}
where the second inequality applies Lemma \ref{appx-mlogit cone} and the third bounds the maximum by the sum.

Now turn to the left side of \Eqref{appx-mlogit start}. Our goal is to show that this is bounded below by a quadratic function. We apply the bounds for \possessivecite{Bach2010_EJS} modified self-concordant functions. To show that $\M(\cdot)$ belongs to this class, we must bound the third derivative in terms of the Hessian. Recall that $\ptany = \exp\{ {x_i^*}'\ganyt\}/\left(1 + \sum_{\NT} \exp\{ {x_i^*}'\ganyt\} \right)$ and the $\T$-square matrix $\H(\{{x_i^*}'\ganyt\}_{\NT})$ has $(\t, \t') \in \NT^2$ entry given by
\begin{equation*}
	\H(\{{x_i^*}'\ganyt\}_{\NT})_{[\t,\t']} = \begin{cases}
				\ptany (1- \ptany) 	& \text{ if } \t = \t'	\\
				- \ptany \hat{p}_{\t'}(\{{x_i^*}'\ganyt\}_{\NT}) 	 & \text{ if } \t \neq \t'
			\end{cases}
\end{equation*}
First, note that $\M(\gany)$ can be written as
\[\M(\gany) = \En\biggl[ \log\biggl(1 + \sumT \exp\{ {x_i^*}'\ganyt\} \biggr)   -  \sumT \dt ({x_i^*}'\ganyt)\biggr].\]
Define $F: \mathbb{R}^\T \to \mathbb{R}$ as $F(w) = \log\left(1 + \sumT \exp(w_\t)\right)$, so that $\M(\gany) = \En\left[ F(w_i) - \sumT \dt w_{i,\t} \right]$, where $w_{i,\t} = {x_i^*}'\ganyt$ and $w_i = \{w_{i,\t}\}_{\NT}$. Then for any $w \in \mathbb{R}^\T$, $v \in \mathbb{R}^\T$, and scalar $\alpha$, define $g(\alpha) = F(w + \alpha v): \mathbb{R} \to \mathbb{R}$. We verify the conditions of \citeasnoun[Lemma 1]{Bach2010_EJS} for this $g(\alpha)$ and $F(w)$. This involves finding the third derivative of $g(\alpha)$, and bounding it in terms of the second (i.e. the Hessian). To this end, note that the multinomial function has the property that $\partial \ptany / \partial \ganyt = \ptany (1- \ptany) {x_i^*}$ and $\partial \ptany / \partial \gamma_{\t',\newdot} = - \ptany \hat{p}_{\t'}(\{{x_i^*}'\ganyt\}_{\NT}) {x_i^*}$. From these, we find 	\[g'(\alpha) = v'F'(w + \alpha v) = \sumT v_\t \hat{p}_\t(w + \alpha v) \qquad \text{and}  \qquad g''(\alpha) = v'F''(w + \alpha v) v = v' \H(w + \alpha v) v.\]
To bound $g'''(\alpha)$, we again use the derivatives of $\ptany$ to find the derivatives of elements $\H(w)$. Routine calculations give, for any $r \neq s \neq \t$:
\begin{align*}
	\partial \H(w)_{\t,\t}/ \partial w_\t & =  \hat{p}_\t(w) ( 1- \hat{p}_\t(w)) (1 - 2 \hat{p}_\t(w)) = \H(w)_{\t,\t} (1 - 2 \hat{p}_\t(w))		\\
	\partial \H(w)_{\t,\t}/ \partial w_r & = - \hat{p}_\t(w) \hat{p}_r(w) ( 1- \hat{p}_\t(w))  + \hat{p}_\t(w)^2\hat{p}_r(w)  = \H(w)_{\t,\t} (\hat{p}_\t(w)\hat{p}_r(w)( 1- \hat{p}_\t(w))^{-1} -  \hat{p}_r(w))		\\
	\partial \H(w)_{\t,s}/ \partial w_\t & =  - \hat{p}_\t(w) \hat{p}_s(w)  (1 - 2 \hat{p}_\t(w)) = \H(w)_{\t,s} (1 - 2 \hat{p}_\t(w))		\\
	\partial \H(w)_{\t,s}/ \partial w_r & =  - \hat{p}_\t(w) \hat{p}_s(w)  ( - 2 \hat{p}_r(w)) = \H(w)_{\t,s} ( -  2 \hat{p}_r(w)).
\end{align*}
Each derivative returns the same Hessian element multiplied by term bounded by 2 in absolute value. Let $a_r$ represent this factor. Then we bound
\begin{multline*}
	g'''(\alpha)  = \left| \sum_{r \in \NT} v_r  \left. \frac{\partial v' \H(\tilde{w}) v  }{\partial w_r}\right|_{\tilde{w} = w + \alpha v} \right| = \left| \sum_{r \in \NT} v_r  v' \H(w + \alpha v) v a_r \right|  	\\	 \leq    \sum_{r \in \NT} v' \H(w + \alpha v) v |v_r|   |a_r |  \leq   2  v' \H(w + \alpha v) v \sum_{r \in \NT}  |v_r|    = 2 \| v \|_1 g''(\alpha) \leq 2 \sqrt{\T}\| v \|_2 g''(\alpha) .
\end{multline*}
Applying \possessivecite{Bach2010_EJS} Lemma 1 to each observation, as in \citeasnoun{Belloni-Chernozhukov-Wei2013_logit}, with $w_i = \{{x_i^*}'\gtruet \}_{\NT}$ and $v_i = \{ {x_i^*}' \dtildet\}_{\NT}$ we get the lower bound
\begin{align}
	M(\gtrue + \dtilde) &  - \M(\gtrue) -  \sumT \En \left[(\pttrue - \dt) {x_i^*}' \right] \dtildet  		\nonumber \\
	&   \geq     \En \left[ \frac{v_i' \H(\{{x_i^*}'\ganyt\}_{\NT}) v_i}{ 4  \T  \| v_i  \|_2^2} \left( e^{-2\| v_i  \|_2} + 2 \| v_i  \|_2 - 1\right) \right]		\nonumber \\
	& \geq  \En \left[ \frac{v_i' \H(\{{x_i^*}'\ganyt\}_{\NT}) v_i}{ 4  \T  \| v_i  \|_2^2} \left( 2 \| v_i  \|_2^2 - \frac{4}{3} \| v_i  \|_2^3\right) \right],	\label{appx-mlogit Bach}
\end{align}
where the second inequality follows from \citeasnoun[Lemma 9]{Belloni-Chernozhukov-Wei2013_logit}. 

\citeasnoun[Theorem 1]{Tanabe-Sagae1992_JRSSB} give $\H(\{{x_i^*}'\gtruet\}_{\NT}) \geq \phi_{\min}\{ \H(\{{x_i^*}'\gtruet\}_{\NT})\} \mathcal{I}_\T $, in the positive definite sense, where $\phi_{\min}(A)$ denotes the smallest eigenvalue of $A$ and $\mathcal{I}_T$ is the $\T \times \T$ identity matrix. Then
\[ \phi_{\min}\{ \H(\{{x_i^*}\gtruet\}_{\NT})\} \geq \det\{ \H(\{{x_i^*}'\ganyt\}_{\NT}) \} = \prod_{\t \in \NTbar} \pttrue \geq \left(p_{\min} \big/ A_p\right)^{\Tbar},\]
where $p_0(\{{x_i^*}'\gtruet\}_{\NT}) = 1 - \sum_{\t \in \NT} \pttrue$ and the first inequality is also due to \citeasnoun{Tanabe-Sagae1992_JRSSB}. These results imply that $v_i' \H(\{{x_i^*}'\ganyt\}_{\NT}) v_i \geq (p_{\min} / A_p)^{\Tbar} v_i'\mathcal{I}_\T v_i = (p_{\min} / A_p)^{\Tbar} \| v_i  \|_2^2$ and therefore
\begin{align}
	\En \left[ \frac{v_i' \H(\{{x_i^*}'\ganyt\}_{\NT}) v_i}{ 4  \T  \| v_i  \|_2^2} \left( 2 \| v_i  \|_2^2 - \frac{4}{3} \| v_i  \|_2^3\right) \right]   &  \geq    \left(p_{\min} \big/ A_p\right)^{\Tbar} \frac{1}{4  \T } \En \left[ 2 \| v_i  \|_2^2 - \frac{4}{3} \| v_i  \|_2^3 \right] 		\nonumber \\
	& = \left(p_{\min} \big/ A_p\right)^{\Tbar} \frac{1}{\T } \frac{\En [ \| v_i  \|_2^2]}{2} \left( 1 - \frac{2}{3}\frac{\En [ \| v_i  \|_2^3]}{\En [ \| v_i  \|_2^2]} \right).		\label{appx-mlogit Hessian cases}
\end{align}

Recall that $v_i = \{ {x_i^*}' \dtildet\}_{\NT}$. To prove a quadratic lower bound, consider two cases, depending on whether 
\[\frac{1}{2} \left( 1 - \frac{2}{3}\frac{\En [ \| \{ {x_i^*}' \dtildet\}_{\NT}  \|_2^3]}{\En [ \| \{ {x_i^*}' \dtildet\}_{\NT}  \|_2^2]} \right)\]
is above or below $1 / A_K$. In the first case, combining Equations \eqref{appx-mlogit Bach} and \eqref{appx-mlogit Hessian cases} gives
\begin{equation}
	\label{appx-mlogit case 1}
	\M(\gtrue + \dtilde)   - \M(\gtrue) -  \sumT \En \left[(\pttrue - \dt) {x_i^*}' \right] \dtildet    \geq   \left(p_{\min} \big/ A_p\right)^{\Tbar} \frac{1}{\T } \frac{\En [ \| \{ {x_i^*}' \dtildet\}_{\NT}  \|_2^2]}{A_K}.
\end{equation}
Now consider the second case, where this bound does not hold. By Assumption \ref{bounded}, the Cauchy-Schwarz inequality, and the conclusion of Lemma \ref{appx-mlogit cone}
\begin{align*}
	\| \{ {x_i^*}' \dtildet \}_{\NT}  \|_1 = \sumT \sumj \left| {x_{i,j}^*} \dtildetj \right|      \leq      \mathcal{X} \left\| \dtilde \right\|_1       &  \leq      \sqrt{\T} \mathcal{X} \vertiii{\dtilde}      		\\
	&  \leq      \sqrt{\T} \mathcal{X}  \left\{ \frac{ 5 \sqrt{|S_*|} }{\kappa_\D}  \vee  \frac{ 10 \bias^d \sqrt{\T}}{\lambda_\D} \right\}\En [ \| \{ {x_i^*}' \dtildet\}_{\NT}  \|_2^2]^{1/2}.
\end{align*}
Hence, by subadditivity (to bound the $\ell_2$ norm by the $\ell_1$ norm), 
\[ \En [ \| \{ {x_i^*}' \dtildet\}_{\NT}  \|_2^3] \leq \En [ \| \{ {x_i^*}' \dtildet\}_{\NT}  \|_2^2 \| \{ {x_i^*}' \dtildet\}_{\NT}  \|_1] \leq \En [ \| \{ {x_i^*}' \dtildet\}_{\NT}  \|_2^2]^{3/2}   \sqrt{\T} \mathcal{X}  \left\{ \frac{ 5 \sqrt{|S_*|} }{\kappa_\D}  \vee  \frac{ 10 \bias^d \sqrt{\T}}{\lambda_\D} \right\}.  \]
Thus
\[\frac{1}{A_K} > \frac{1}{2} \left( 1 - \frac{2}{3}\frac{\En [ \| \{ {x_i^*}' \dtildet\}_{\NT}  \|_2^3]}{\En [ \| \{ {x_i^*}' \dtildet\}_{\NT}  \|_2^2]} \right)  \geq  \frac{1}{2} \left( 1 -    \frac{2}{3} \frac{\mathcal{X} \sqrt{\T} }{\kappa_\D \lambda_\D} \left( 5 \lambda_\D \sqrt{|S_*|}  +   10 \kappa_\D \bias^d \sqrt{\T}  \right)  \En [ \| \{ {x_i^*}' \dtildet\}_{\NT}  \|_2^2]^{1/2} \right),   \]
which is equivalent to
\[\En [ \| \{ {x_i^*}' \dtildet\}_{\NT}  \|_2^2]^{1/2} > \left( 1 - \frac{2}{A_K}\right) \frac{3}{2}  \frac{\kappa_\D \lambda_\D}{\mathcal{X} \sqrt{\T} } \left( 5 \lambda_\D \sqrt{|S_*|}  +   10 \kappa_\D \bias^d \sqrt{\T}  \right)^{-1}     := r_n.\]
Because $\M(\gtrue + \dany) - \M(\gany) - \sumT \En \left[(\pttrue - \dt) {x_i^*}' \right] \danyt$ is convex in $\dany$, and hence any line segment lies above the function, we know that $\En [ \| \{ {x_i^*}' \dtildet\}_{\NT}  \|_2^2]^{1/2} > r_n$, so we have
\begin{align*}
\M(\gtrue + \dtilde) - \M(\gany) - \sumT \En \left[(\pttrue - \dt) {x_i^*}' \right] \dtildet \geq r_n^2 & \geq  r_n^2 \frac{\En [ \| \{ {x_i^*}' \dtildet \}_{\NT}  \|_2^2]^{1/2}}{r_n} 		\\
& = r_n \En [ \| \{ {x_i^*}' \dtildet \}_{\NT}  \|_2^2]^{1/2}.
\end{align*}
Combining this result with Equations \eqref{appx-mlogit start} and \eqref{appx-mlogit upper}, we have
\begin{multline*}
	\left( 1 - \frac{2}{A_K}\right) \frac{3}{2} \frac{\kappa_\D \lambda_\D}{\mathcal{X} \sqrt{\T} } \left( 5 \lambda_\D \sqrt{|S_*|}  +   10 \kappa_\D \bias^d \sqrt{\T}  \right)^{-1}  \En [ \| \{ {x_i^*}' \danyt\}_{\NT}  \|_2^2]^{1/2}      		\\		\leq      \left(    6 \frac{ \lambda_\D \sqrt{|S_*|} }{\kappa_\D}        +     8  \bias^d \sqrt{\T}             \right)   \En [ \| \{ {x_i^*}' \dtildet\}_{\NT}  \|_2^2]^{1/2},
\end{multline*}
which is impossible under the restriction on $A_K$. Therefore, \Eqref{appx-mlogit case 1} must hold.\footnote{This analysis is conceptually similar to using \possessivecite{Belloni-Chernozhukov2011_AoS} restricted nonlinearity impact coefficient, but our characterization is different.} Combining this with Equations \eqref{appx-mlogit start} and \eqref{appx-mlogit upper}, we find that
\[\left(p_{\min} \big/ A_p\right)^{\Tbar} \frac{1}{\T } \frac{\En [ \| \{ {x_i^*}' \dtildet\}_{\NT}  \|_2^2]}{A_K} \leq  \left(    6 \frac{ \lambda_\D \sqrt{|S_*|} }{\kappa_\D}        +     8  \bias^d \sqrt{\T}             \right)   \En [ \| \{ {x_i^*}' \dtildet\}_{\NT}  \|_2^2]^{1/2}.\]
Thus, dividing through and applying the union bound we find that
\begin{equation}
	\label{appx-mlogit log odds rate}
	\max_{\t \in \NT} \En [  ({x_i^*}' \dtildet)^2]^{1/2}      \leq        \En [ \| \{ {x_i^*}' \dtildet\}_{\NT}  \|_2^2]^{1/2}         \leq         \left(A_p \big/ p_{\min}\right)^{\Tbar}     \T A_K \left(    6 \frac{ \lambda_\D \sqrt{|S_*|} }{\kappa_\D}    +   8  \bias^d \sqrt{\T}   \right)   .
\end{equation}

To bound the propensity score error, we apply the mean value theorem and the form of $\partial \ptany / \partial \ganyt$. We must linearize with respect to $\t$ only (recall that $\pttilde$ depends on all of $\gtilde$). To this end, define $M_\t$ as the $\T$-vector with entry $\t$ given by ${x_i^*}'\gtruet + \tilde{m}_\t{x_i^*}'\gtildet$ for a scalar $\tilde{m}_\t \in [0,1]$ and entries $\t' \in \NT \setminus \{\t\}$ equal to ${x_i^*}'\gamma_{\t'}$. Then we have
\begin{equation}
	\label{appx-mlogit rate mvt}
	\left|\pttilde - \pttrue \right|   =   \left| \hat{p}_\t(M_\t) [1 - \hat{p}_\t(M_\t)]{x_i^*}'\dtildet\right|   \leq   \left|{x_i^*}'\dtildet\right|.
\end{equation}
Using this result coupled with the triangle inequality, the bias condition, and \Eqref{appx-mlogit log odds rate}, we find
\begin{align*}
	\En[(\pttilde - p_\t(x_i))^2] ^{1/2} & \leq \En[(\pttilde - \pttrue)^2] ^{1/2} + \En[(\pttrue - p_\t(x_i))^2] ^{1/2} 		\\
	& \leq \En\left[({x_i^*}'\dtildet )^2\right] ^{1/2} + \bias^d		\\
	& \leq \left(A_p \big/ p_{\min}\right)^{\Tbar}     \T A_K \left(    6 \frac{ \lambda_\D \sqrt{|S_*|} }{\kappa_\D}        +     8  \bias^d \sqrt{\T}             \right)   +   \bias^d.
\end{align*}

The $\ell_1$ bound follows from \Eqref{appx-mlogit log odds rate}, the Cauchy-Schwarz inequality, and \Eqref{SE}:
\begin{align*}
	\left\| \gtildet - \gtruet \right\|_1 \leq \sqrt{|\tilde{S}^\D \cup S_\D^*|} \left\| \gtildet - \gtruet \right\|_{2,p} \leq \left(\frac{ |\tilde{S}^\D \cup S_\D^*| }{ \underline{\phi}\{Q, \tilde{S}^\D \cup S_\D^*\} } \right)^{1/2} \En[({x_i^*}'(\gtildet - \gtruet))^2]^{1/2}.	
\end{align*}

Finally, we bound the size of the selected set of coefficients. First, note that optimality of $\gtilde$ ensures that $| \tilde{S}^\D | \leq n$. Then, restating the conclusion Lemma \ref{appx-mlogit S-hat} using the notation of the Theorem and the rate result \eqref{appx-mlogit log odds rate}, then bounding $\overline{\phi}$ by $\dbarphi$ we find that
\[ | \tilde{S}^\D |    \leq    |S_\D^* |4 L_n \dbarphi\{Q, |\tilde{S}^\D|\}.\]
The argument now parallels that used by \citeasnoun{Belloni-Chernozhukov2013_Bern}, relying on their result on the sublinearity of sparse eigenvalues. Let $\lceil m \rceil$ be the ceiling function and note that $\lceil{m}\rceil \leq 2m$. For any $m \in \N_Q^\D$, suppose that $|\tilde{S}^\D| > m$. Then, 
\begin{align*}
	| \tilde{S}^\D |& \leq |S_\D^* |4 L_n \dbarphi\{Q, m (|\tilde{S}^\D| / m) \} 		\\
	& \leq  \left\lceil |\tilde{S}^\D| / m \right \rceil  |S_\D^* |4 L_n   \dbarphi\{Q, m  \}		\\
	& \leq  ( |\tilde{S}^\D| / m) |S_\D^* | 8 L_n   \dbarphi\{Q, m\}.
\end{align*}
Rearranging gives 
$m \leq |S_\D^* | 8 L_n   \dbarphi\{Q, m\}$
whence $m \not\in \N_Q^\D$. Minimizing over $\N_Q^\D$ gives the result.	 \qed

%%%%%%%%%%%%%%%%%%%%%%%%%%%%
\subsection{Proof of Theorem \ref{thm-post mlogit}}

Define $\dhat = \ghat - \gtrue$. Many of the arguments parallel those for Theorem \ref{thm-mlogit}. The key differences are that a quadratic lower bound for $\M(\gtrue + \dhat) - \M(\gtrue) - \sumT \En \left[(\pttrue - \dt) {x_i^*}' \right] \dhatt$ may occur, but is not necessary, and $\dhat$ may not belong to the cone of the restricted eigenvalues, but obeys the sparse eigenvalue constraints. 

We first give a suitable upper bound for $\M(\gtrue + \dhat) - \M(\gtrue) - \sumT \En \left[(\pttrue - \dt) {x_i^*}' \right] \dhatt$. By the Cauchy-Schwarz inequality and the definition of the sparse eigenvalues of \Eqref{SE},
\begin{align}
	\vertiii{ \dhat } & \leq  \sqrt{ \left| \hat{S}_\D \cup S_\D^* \right|} \sqrt{ \sumT \sum_{j \in \hat{S}_\D \cup S_\D^*}  \dhattj^2 }	 		\nonumber  \\
	& \leq  \sqrt{ \left| \hat{S}_\D \cup S_\D^* \right|} \sqrt{ \sumT \underline{\phi}\left\{Q, \hat{S}_\D \cup S_\D^*  \right\}^{-2} \dhatt' Q \dhatt }		\nonumber  \\
	& =  \sqrt{ \left| \hat{S}_\D \cup S_\D^* \right|}  \underline{\phi}\left\{Q, \hat{S}_\D \cup S_\D^*  \right\}^{-1} \En [ \| \{ {x_i^*}' \dhatt\}_{\NT}  \|_2^2]^{1/2}.				\label{appx-post mlogit norm}
\end{align}
Following identical steps to Equations \eqref{appx-mlogit score variance}, \eqref{appx-mlogit score bias}, and \eqref{appx-mlogit score}, but with $\dhat$ in place of $\dtilde$, and then using the above bound, we have
\begin{align}
	\left| \sumT \En \left[(\pttrue - \dt) {x_i^*}' \right] \dhatt \right|  & \leq   \frac{\lambda_\D}{2} \vertiii{ \dhat } + \bias^d \sqrt{\T} \En [ \| \{ {x_i^*}' \dhatt\}_{\NT}  \|_2^2]^{1/2} 		 \nonumber \\
	& \leq    \left(  \frac{\lambda_\D}{2}  \frac{ \sqrt{ | \hat{S}_\D \cup S_\D^* |} }{  \underline{\phi}\{Q, \hat{S}_\D \cup S_\D^*  \} }    + \bias^d \sqrt{\T}    \right)   \En [ \| \{ {x_i^*}' \dhatt\}_{\NT}  \|_2^2]^{1/2}.		\label{appx-post mlogit score}
\end{align}

Next we turn to $\M(\gtrue + \dhat)  - \M(\gtrue)$. By optimality of the post selection estimator $\M(\ghat) \leq \M(\gtilde)$, as $\tilde{S}^\D \subset \hat{S}_\D$ by construction, and hence $\M(\gtrue + \dhat)  - \M(\gtrue) \leq \M(\gtilde)  - \M(\gtrue)$. By the mean value theorem, for scalars $\{m_\t \in [0,1]\}_{\NT}$ we have
\begin{align}
	\M(\gtrue + \dtilde)  - \M(\gtrue)  &  =  \sumT \En \left[(\dt - \hat{p}_\t(\{{x_i^*}'\gtruet + m_\t{x_i^*}'\dtildet\}) ) {x_i^*}'\dtildet \right]		\nonumber \\
	& = \sumT \En \left[(\dt  - \pttrue) {x_i^*}'\dtildet \right]  		\nonumber \\
	& \quad  +   \sumT \En \left[(\pttrue - \hat{p}_\t(\{{x_i^*}'\gtruet + m_\t{x_i^*}'\dtildet\}) ) {x_i^*}'\dtildet \right],		\nonumber \\
	& \leq   \frac{\lambda_\D}{2} \vertiii{ \dtilde } + \bias^d \sqrt{\T} \En [ \| \{ {x_i^*}' \dtildet\}_{\NT}  \|_2^2]^{1/2}    + \sumT \En \left[m_\t ({x_i^*}'\dtildet)^2 \right].		\nonumber \\
	& \leq \left( \frac{\lambda_\D}{2}  \frac{\sqrt{ | \hat{S}_\D \cup S_\D^* |}} {  \underline{\phi}\{Q, \hat{S}_\D \cup S_\D^*  \}}     + \bias^d \sqrt{\T} \right)   \En [ \| \{ {x_i^*}' \dtildet\}_{\NT}  \|_2^2]^{1/2}    +  \En [ \| \{ {x_i^*}' \dtildet\}_{\NT}  \|_2^2],		\label{appx-mlogit loss fcn}		
\end{align}
where the first inequality follows from \Eqref{appx-mlogit score} and the same steps as in \eqref{appx-mlogit rate mvt} while the second applies \eqref{appx-post mlogit norm} with $\dtilde$ and $m_\t \leq 1$.\footnote{Applying the steps of \Eqref{appx-post mlogit norm} to $\dtilde$ is preferred to using the results of Lemma \ref{appx-mlogit cone} because it leads to the tidier expression involving $\underline{\phi}\{Q, \hat{S}_\D \cup S_\D^*  \}$, but the latter method could be substituted.}

Collecting the bounds of \eqref{appx-post mlogit score} and \eqref{appx-mlogit loss fcn}, and the definition of $R_\M$ gives
\begin{multline}
	\label{appx-post mlogit upper}
	\M(\gtrue + \dhat) - \M(\gtrue) - \sumT \En \left[(\pttrue - \dt) {x_i^*}' \right] \dhatt 	 		\\		
	\leq   \left( \frac{\lambda_\D}{2}  \frac{\sqrt{ | \hat{S}_\D \cup S_\D^* |}} {  \underline{\phi}\{Q, \hat{S}_\D \cup S_\D^*  \}}     + \bias^d \sqrt{\T} \right)  \left( \En [ \| \{ {x_i^*}' \dhatt\}_{\NT}  \|_2^2]^{1/2} +  R_\M  \right) +  R_\M^2 .
\end{multline}

Next, we turn to a lower bound. Consider the same two cases as in the proof of Theorem \ref{thm-mlogit}. In the first case, we have the quadratic lower bound:
\begin{equation}
	\label{appx-post mlogit case 1}
	M(\gtrue + \dhat)   - \M(\gtrue) -  \sumT \En \left[(\pttrue - \dt) {x_i^*}' \right] \dhatt    \geq   \left(p_{\min} \big/ A_p\right)^{\Tbar}  \frac{\En [ \| \{ {x_i^*}' \dhatt\}_{\NT}  \|_2^2]}{ \T A_K}.
\end{equation}
In the other case, this bound may not hold. Arguing as in the proof of Theorem \ref{thm-mlogit}, but applying \Eqref{appx-post mlogit norm}, we get
\[\| \{ {x_i^*}' \dhatt \}_{\NT}  \|_1 \leq \sqrt{\T} \mathcal{X} \sqrt{ | \hat{S}_\D \cup S_\D^* |}  \underline{\phi}\{Q, \hat{S}_\D \cup S_\D^*  \}^{-1} \En [ \| \{ {x_i^*}' \dtildet\}_{\NT}  \|_2^2]^{1/2}.\]
Therefore, as above, we find 
\begin{equation}
	\label{appx-post mlogit case 2}
	\M(\gtrue + \dhat) - \M(\gany) - \sumT \En \left[(\pttrue - \dt) {x_i^*}' \right] \dhatt \geq r_n \En [ \| \{ {x_i^*}' \dhatt \}_{\NT}  \|_2^2]^{1/2},
\end{equation}

\[\text{with} \qquad \qquad r_n  =  \frac{3}{2} \left( 1 - \frac{2}{A_K}\right) \frac{ \underline{\phi}\{Q, \hat{S}_\D \cup S_\D^*  \} }{\mathcal{X} \sqrt{\T} \sqrt{ | \hat{S}_\D \cup S_\D^* |}}.\]
Collecting the upper bound of \eqref{appx-post mlogit upper} and the lower bounds \eqref{appx-post mlogit case 1} and \eqref{appx-post mlogit case 2} we have
\begin{multline}
	\label{appx-post mlogit cases}
	\left\{ \left(p_{\min} \big/ A_p\right)^{\Tbar} \frac{1}{\T } \frac{\En [ \| \{ {x_i^*}' \dhatt\}_{\NT}  \|_2^2] }{A_K} \right\}    \wedge   \left\{  r_n \En [ \| \{ {x_i^*}' \dhatt \}_{\NT}  \|_2^2]^{1/2} \right\}     		\\
	\leq      \left( \frac{\lambda_\D}{2}  \frac{\sqrt{ | \hat{S}_\D \cup S_\D^* |}} {  \underline{\phi}\{Q, \hat{S}_\D \cup S_\D^*  \}}     + \bias^d \sqrt{\T} \right)  \left( \En [ \| \{ {x_i^*}' \dhatt\}_{\NT}  \|_2^2]^{1/2} +  R_\M  \right) +  R_\M^2.	
\end{multline}
Suppose the linear term is the minimum. The restrictions on $A_K$ imply, algebraically, that \Eqref{appx-post mlogit cases} yields
\begin{align*}
	r_n \En [ \| \{ {x_i^*}' \dhatt \}_{\NT}  \|_2^2]^{1/2}   &   \leq (r_n / 3) \left(  \En [ \| \{ {x_i^*}' \dhatt \}_{\NT}  \|_2^2]^{1/2} +  R_\M \right) + R_\M^2  		\\
	& \leq  (r_n / 3) \left(  \En [ \| \{ {x_i^*}' \dhatt \}_{\NT}  \|_2^2]^{1/2} +  2R_\M \right).
\end{align*}
Canceling the $r_n$ and solving yields
$\En [ \| \{ {x_i^*}' \dhatt \}_{\NT}  \|_2^2]^{1/2}   \leq    R_\M $.
On the other hand, if the quadratic term is the minimum, define
\[R_\M' = \left(A_p \big/ p_{\min}\right)^{\Tbar}\T A_K \left( 2^{-1} \lambda_\D \sqrt{ | \hat{S}_\D \cup S_\D^* |} \big/   \underline{\phi}\{Q, \hat{S}_\D \cup S_\D^*  \}     + \bias^d \sqrt{\T} \right).\]
With this notation and the quadratic term being the minimum, \Eqref{appx-post mlogit cases} becomes
\[ \En [ \| \{ {x_i^*}' \dhatt\}_{\NT}  \|_2^2] \leq R_\M'  \En [ \| \{ {x_i^*}' \dhatt\}_{\NT}  \|_2^2]^1/2 + R_\M' R_\M + \left(A_p \big/ p_{\min}\right)^{\Tbar} \T A_K R_\M^2.\]
Then, because $a^2 \leq ab + c$ implies that $a \leq b + \sqrt{c}$, we have
	\[\En [ \| \{ {x_i^*}' \dhatt \}_{\NT}  \|_2^2]^{1/2}  \leq R_\M' + \left(  R_\M' R_\M + \left(A_p \big/ p_{\min}\right)^{\Tbar} \T A_K R_\M^2 \right)^{1/2}.\]
Combining the bounds on $\En [ \| \{ {x_i^*}' \dhatt \}_{\NT}  \|_2^2]^{1/2}$ from the two cases gives
\begin{equation*}
	\label{appx-post log odds rate}
	\En [ \| \{ {x_i^*}' \dhatt \}_{\NT}  \|_2^2]^{1/2}  \leq   \left\{ R_\M \right\}  \vee  \left\{  R_\M' + \left(  R_\M' R_\M + \left(A_p \big/ p_{\min}\right)^{\Tbar} \T A_K R_\M^2 \right)^{1/2} \right\}.
\end{equation*}

From this bound on the log-odds, we bound the propensity score and the $\ell_1$ rate:
	\[\max_{\t \in \NT} \En[(\pthat - p_\t(x_i))^2] ^{1/2} \leq   \left\{R_\M \right\}  \vee  \left\{  R_\M' + \left(  R_\M' R_\M + \left(A_p \big/ p_{\min}\right)^{\Tbar} \T A_K R_\M^2 \right)^{1/2} \right\}   + \bias^d;\]
	\[\max_{\t \in \NT} \left\| \ghatt - \gtruet \right\|_1  \leq \left(\frac{ |\tilde{S}^\D \cup S_\D^*| }{ \underline{\phi}\{Q, \tilde{S}^\D \cup S_\D^*\} } \right)^{1/2} \left\{ R_\M \right\}  \vee  \left\{  R_\M' + \left(  R_\M' R_\M + \left(A_p \big/ p_{\min}\right)^{\Tbar} \T A_K R_\M^2 \right)^{1/2} \right\},\]
by arguments parallel to those used in the proof of Theorem \ref{thm-mlogit}. \qed

%%%%%%%%%%%%%%%%%%%%%%%%%%%%
%%%%%%%%%%%%%%%%%%%%%%%%%%%%
\section{Proofs for Group Lasso Selection and Estimation of Linear Models}

SEE SUPPLEMENTAL APPENDIX.

\end{appendices}

%%%%%%%%%%%%%%%%%%%%%%%%%%%%%%%%%%%
%%%%%%%%%%%%%%%%%%%%%%%%%%%%%%%%%%%

%% These two lines will create a standard bibliography, with a header 'References' that does NOT appear in the table of contents
%	\bibliography{MasterBibliography}{}
%	\bibliographystyle{econometrica}

%% Alternatively, locally redefine the command '\section' to get something that appears in the table of contents.
	\normalsize
	\section{References}
	\small
	\singlespacing
	\begingroup
	\renewcommand{\section}[2]{}	%this removes the standard 'References' header that BibTeX puts in.
	\bibliography{Farrell2015_JoE--Corrigendum-Bibliography}{}
	\bibliographystyle{joe}
	\endgroup

\clearpage

\pagestyle{empty}

\begin{sidewaystable}
	\begin{center}
		\begin{threeparttable}
			\caption{Analysis of NSW Demonstration: Treatment Effects on the Treated and Confidence Intervals for Various Specifications}
			\label{table-lalonde}
			\small
			\begin{tabular}{l r r c r r c r c r   }
				\hline
				\hline\noalign{\smallskip}
				& \multicolumn{2}{c}{Number of Variables}  &  & \multicolumn{2}{c}{Sample Sizes$^{\text{(c)}}$}   \\
				\cline{2-3}  \cline{5-6}  \noalign{\smallskip}
				 & Before & After  & &  &  & &  & &  \\
				Specifications: & selection$^{\text{(a)}}$ & selection$^{\text{(b)}}$  & & Control & Treated & & ATT & & 95\% CI  \\
				\hline
				\noalign{\smallskip}
				\emph{Experimental Benchmark}  & -- & -- &   & 260 & 185 & & 	1794 &     & [110, 3479]  \\

				\noalign{\medskip}
				\emph{Doubly-Robust Estimates} \\
				Specification 1 (No Selection)	 & N/A & 11 &   & 1211 & 185   & &   1664   & &   [-276, 3604] \\
				DW02 (Informal Selection)	 & ?? & 15 &   & 1058 & 185   & &   2528   & &   [149, 4908] \\

				\noalign{\smallskip}
				{\bf Refitting after Group Lasso Selection} & 171 & 20/6 &   & 1735 & 185   & &   1737   & &   [33, 3441]  \\

				\hline
			\end{tabular}
			\begin{tablenotes}[online]
				\item[Notes:] All analyses use the DW99 subsample and PSID comparison group. Specifications vary, but all estimates and standard errors of from the method defined in Section \ref{sec-ate} with the exception of the partially linear model. 
				\item[(a)] Not counting the intercept. The total set of variables considered by DW02 is not known.
				\item[(b)] For the group lasso estimators, the two numbers given are for those used in the outcome regressions and propensity score, respectively. For other doubly-robust estimators all variables are used in the propensity score and outcome models.
				\item[(c)] The full sample begins with 2490 comparisons and 185 treated units. Control observations outside the range of estimated propensity scores in the treated sample are discarded. 
			\end{tablenotes}
		\end{threeparttable}
	\end{center}
\end{sidewaystable}

\begin{figure}
	\begin{center}
		\caption{Empirical Coverage of 95\% Confidence Intervals, Varying Signal Strength and Sparsity of $p_\t(x)$ and $\mut(x)$}
		\label{fig-plot-main}
		\includegraphics[scale=1]{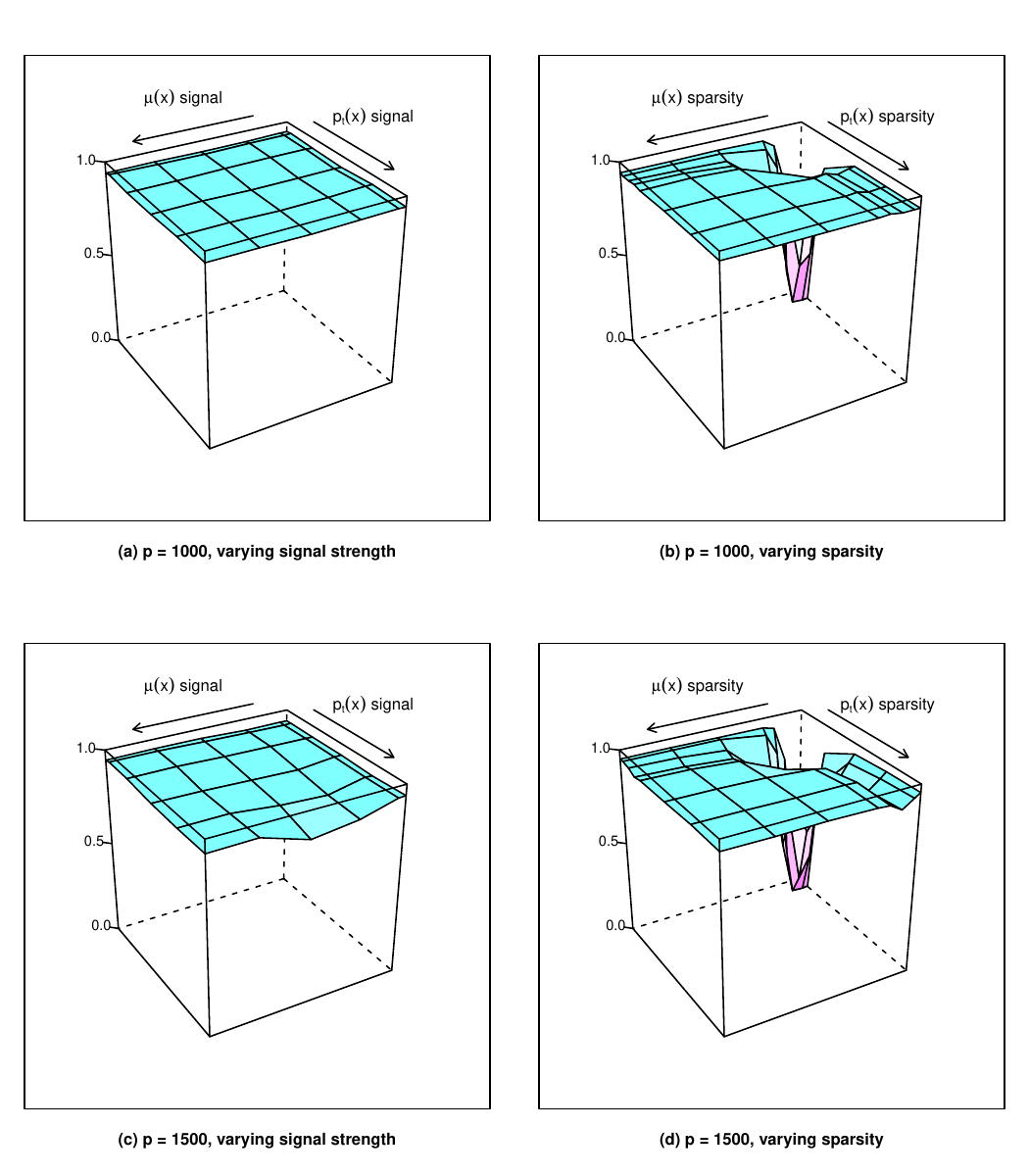}
	\end{center}
\end{figure}

\begin{figure}
	\begin{center}
		\caption{Empirical Coverage of 95\% Confidence Intervals, Penalty Chosen with Cross-validation, Varying Signal Strength of $p_\t(x)$ and $\mut(x)$}
		\label{fig-plot-cross}
		\includegraphics[scale=1]{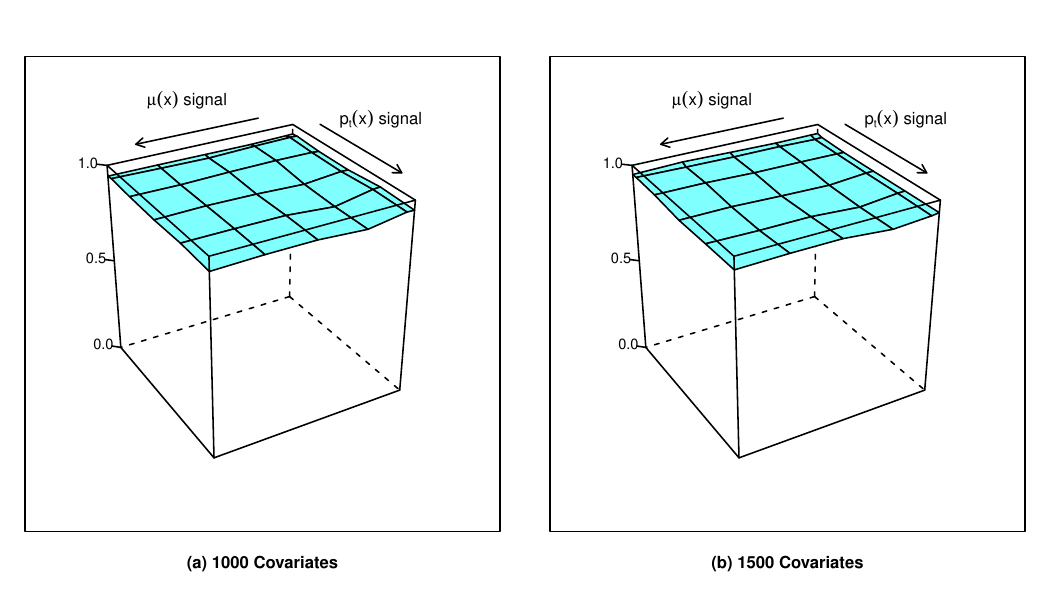}
	\end{center}
\end{figure}

\clearpage

%%%%%%%%%%%%%%%%%%%%%%%%%%%%%%%%%%%
%%%%%%%%%%%%%%%%%%%%%%%%%%%%%%%%%%%
%%%%%%%%%%%%%%%%%%%%%%%%%%%%%%%%%%%
%%%%										%%%%
%%%%										%%%%
%%%%										%%%%
%%%%			Supplement starts here			%%%%
%%%%										%%%%
%%%%										%%%%
%%%%										%%%%
%%%%%%%%%%%%%%%%%%%%%%%%%%%%%%%%%%%
%%%%%%%%%%%%%%%%%%%%%%%%%%%%%%%%%%%
%%%%%%%%%%%%%%%%%%%%%%%%%%%%%%%%%%%

\begin{center}
{\Large Supplemental Appendix for ``Robust Inference on Average Treatment Effects with Possibly More Covariates than Observations''}

\bigskip 
Max H. Farrell 

March 27, 2015

Updated: February 1, 2018\footnote{The published version of the supplement, \citeasnoun{Farrell2015_JoE--Supplement}, contains an error in the proof which is rectified here. See the author's website for further detail. I am grateful to Whitney Newey for alerting me to this error.}
\end{center}
\bigskip
\begin{center}
{\small\bf Summary}
\end{center}
{\small This is a supplemental appendix for ``Robust Inference on Average Treatment Effects with Possibly More Covariates than Observations'' containing complete proofs. Notation is kept in line with the main text, though equation numbers may change. This file may serve as a drop-in replacement for the main appendix. Additional simulations results are also presented.}

%%%%%%%%%%%%%%%%%%%%%%%%%%%%
%%%%%%%%%%%%%%%%%%%%%%%%%%%%
\begin{appendices}
\newpage
%\cleardoublepage	%If you want a blank even numbered page between them, so the Appendixes start fresh.

%\begin{center}
%	\LARGE{Supplemental Appendixes for \\ Post Model Selection Inference for Average Treatment Effects with High-Dimensional Controls} 
%\end{center}

\small
\singlespacing

\numberwithin{lemma}{section}
\numberwithin{equation}{section}
\numberwithin{figure}{section}
\numberwithin{table}{section}
\numberwithin{theorem}{section}

%%%%%%%%%%%%%%%%%%%%%%%%%%%%
%%%%%%%%%%%%%%%%%%%%%%%%%%%%
\section{Proofs for Treatment Effect Inference}
	\label{supp-ate proofs}

The proofs in this section are asymptotic in nature, compared to the nonasymptotic bounds of the next section. It shall be understood that asymptotic order symbols hold for the sequence being considered, as a shorthand for the more formal versions given in the assumptions (e.g. Assumption \ref{first stage}). $C$ will denote a generic positive constant, which may be a matrix. Define the set of indexes $\It = \{i : \d = \t\}$.

\begin{proof}[Proof of Theorem \ref{thm-ate consistency}.]
First, we have $\muthat - \mut  = \En[ \psi_\t (y_i, \dt, \mut^0(x_i), p_\t^0(x_i), \mut) ] + R_1 + R_2 + R_3 + R_4$, where
\begin{gather*}
	R_1 = \En[\muthat(x_i) - \mut(x_i)], \qquad R_2 = \En\left[\frac{\dt y_i }{\hat{p}_\t(x_i) p_\t^0(x_i)} ( p_\t^0(x_i) - \hat{p}_\t(x_i)) \right],		\\
	R_3 = - \En\left[\frac{\dt }{\hat{p}_\t(x_i) } ( \muthat(x_i) - \mut^0(x_i)) \right], \qquad \text{and}  \qquad  R_4 = -  \En\left[\frac{\dt \mut^0(x_i) }{\hat{p}_\t(x_i) p_\t^0(x_i)} ( p_\t^0(x_i) - \hat{p}_\t(x_i)) \right].
\end{gather*}
Under Assumptions \ref{overlap} and \ref{dgp asmpts}, if follows from the Cauchy-Schwarz inequality, the consistency condition of the theorem and the inequality of \citeasnoun{vonBahr-Esseen1965_AoMS}, it follows that $|R_k| = o_{P_n}(1)$ for $k = 1, 2, 3, 4$. Finally, by the same von Bahr and Esseen inequality $| \En[ \psi_\t (y_i, \dt, \mut^0(x_i), p_\t^0(x_i), \mut) ]| = o_{P_n}(1)$ if either $\mut(x)$ or $p_\t(x)$ is correctly specified.
\end{proof}

We first prove Theorem \ref{thm-ate linearity} assuming there is no additional randomness injected into the support estimates. Following this, we redo the proof to account for additional randomness. We then turn to the remaining portions of Theorem \ref{thm-ate} and to Corollary \ref{thm-ate uniform}, which require shorter arguments. 

We make frequent use of the linearization
\begin{equation}
	\label{supp-linearization}
	\frac{1}{a} = \frac{1}{b} + \frac{b - a}{ab} =  \frac{1}{b} + \frac{b - a}{b^2} + \frac{(b - a)^2}{ab^2},
\end{equation}
where the first inequality is readily verified, and the second re-applies the first.

\begin{proof}[Proof of Theorem \ref{thm-ate linearity} without Additional Randomness.]
With $\psi_\t (\cdot)$ defined in \Eqref{eqn-ate moments}, we have $ \sqrt{n} ( \muthat - \mut ) = \sqrt{n} \En[ \psi_\t (y_i, \dt, \mut(x_i), p_\t(x_i), \mut) ] + R_1 + R_2$, where
\[R_1 = \frac{1}{\sqrt{n}} \sumi \dt (y_i - \mut(x_i)) \left( \frac{1}{ \hat{p}_\t(x_i) }  -  \frac{1}{ p_\t(x_i) } \right)\]
and
\[ R_2 = \frac{1}{\sqrt{n}} \sumi (\muthat(x_i) - \mut(x_i)) \left( 1 - \frac{ \dt }{ \hat{p}_\t(x_i) } \right).\]
The proof proceeds by showing that both $R_1$ and $R_2$ are $o_{P_n}(1)$.

For $R_1$, applying the first equality in \Eqref{supp-linearization}, we rewrite $R_1$ as 
	\[R_1 = \frac{1}{\sqrt{n}} \sumi \dt u_i \left( \frac{p_\t(x_i) - \hat{p}_\t(x_i)}{ \hat{p}_\t(x_i) p_\t(x_i)}  \right).\]
Applying Assumptions \ref{overlap} and \ref{fourth moments} and the first-stage consistency condition of Assumption \ref{consistent}:
\[ \E\left[ R_1^2 \vert \{x_i, \d\}_{i = 1}^n \right]  = \En \left[ \frac{ \dt \sigma_t^2(x_i) }{\hat{p}_\t(x_i)^2 p_\t(x_i)^2} \left( p_\t(x_i) - \hat{p}_\t(x_i) \right)^2 \right] \leq C \En[(p_\t(x_i) - \hat{p}_\t(x_i))^2] = o_{P_n}(1). \]

Next, again using \Eqref{supp-linearization} we have
\[1 - \frac{\dt}{\hat{p}_\t(x_i) }= \frac{ p_t(x_i)  - \dt}{p_t(x_i)} + \frac{\dt( \hat{p}_t(x_i) -  p_t(x_i)) }{\hat{p}_t(x_i) p_t(x_i)}.\]
We use this to re-write $R_{2} = R_{21} + R_{22}$, where
\[R_{21} =  \frac{1}{\sqrt{n}} \sumi (\muthat(x_i) - \mut(x_i)) \left(\frac{ p_\t(x_i)  -  \dt }{ p_\t(x_i) } \right) \]
and
\[R_{22} = \frac{1}{\sqrt{n}} \sumi (\muthat(x_i) - \mut(x_i))(\hat{p}_\t(x_i)  -  p_\t(x_i))\left(  \frac{   \dt }{ \hat{p}_\t(x_i) p_\t(x_i) }  \right)  .   \]
For the first term, $ R_{21} = \sqrt{n}  \En[ (\muthat(x_i) - \mut(x_i)) (1 - d_i^t/ p_\t(x_i))] = o_{P_n}(1)$ by Assumption \ref{new}. Next, by H\"older's inequality, Assumption \ref{overlap} and the rate condition of Assumption \ref{ATE RATES}
\begin{align*}
	|R_{22}| & \leq \sqrt{n} \left( \maxi \frac{  1 }{ \hat{p}_\t(x_i) p_\t(x_i) } \right)\sqrt{ \En[ (\muthat(x_i) - \mut(x_i))^2]\En[ (\hat{p}_\t(x_i) - p_\t(x_i))^2]}		\\
	& = O_{P_n} (1) \sqrt{n} \sqrt{ \En[ (\muthat(x_i) - \mut(x_i))^2]\En[ (\hat{p}_\t(x_i) - p_\t(x_i))^2]} = o_{P_n}(1).
\end{align*}
This completes the proof, as $|R_1 + R_2| = o_{P_n}(1)$ by Markov's inequality and the triangle inequality.
\end{proof}

\begin{proof}[Proof of Theorem \ref{thm-ate linearity} with Additional Randomness.]
We must reconsider the remainders $R_1$ and $R_2$. For the former, applying \Eqref{supp-linearization}, we find $R_1 = R_{11} + R_{12}$, where
\[R_{11} = \frac{1}{\sqrt{n}} \sumi \frac{\dt u_i}{p_\t(x_i)^2} \left( p_\t(x_i) - \hat{p}_\t(x_i)\right) 	\qquad \text{ and } \qquad     	 R_{12} =  \frac{1}{\sqrt{n}} \sumi \frac{\dt u_i}{p_\t(x_i)^2 \hat{p}_\t(x_i) } \left( \hat{p}_\t(x_i)   -  p_\t(x_i)\right)^2. \]
For $R_{11}$, we first add and subtract the parametric representation to get $R_{11} = R_{111} + R_{112}$, where, 
\[R_{111} = \frac{1}{\sqrt{n}} \sumi \frac{\dt u_i}{p_\t(x_i)^2} \left( \pttrue - \pthat\right) 	\quad \text{ and } \quad     	R_{112} = \frac{1}{\sqrt{n}} \sumi \frac{\dt u_i}{p_\t(x_i)^2} \left( p_\t(x_i) - \pttrue\right).\]

By a two-term mean-value expansion $R_{111} = R_{111a} + R_{111b}$, with
\[R_{111a} = \frac{1}{\sqrt{n}} \sumi \frac{\dt u_i}{p_\t(x_i)^2} \sumT \left\{\pttrue ( 1- \pttrue) \left({x_i^*}'(\ghatt - \gtruet) \right) \right\} \]
and
\[R_{111b} = \frac{1}{2 \sqrt{n}} \sumi \frac{\dt u_i}{p_\t(x_i)^2} v_i' \bar{\H} v_i,\]
where $v_i  = \{ {x_i^*}'(\ghatt - \gtruet)\}_{\NT}$ and $\overline{\H}  =  \H(\{{x_i^*}'\gtruet  +  m_\t {x_i^*}'\ghatt\}_{\NT})$ for appropriate scalars $m_\t$ where the $\T$-square matrix $\H(\{{x_i^*}'\ganyt\}_{\NT})$ is defined as having the $(\t, \t') \in \NT^2$ entry given by
\begin{equation*}
	\H(\{{x_i^*}'\ganyt\}_{\NT})_{[\t,\t']} = \begin{cases}
				\ptany (1- \ptany) 	& \text{ if } \t = \t'	\\
				- \ptany \hat{p}_{\t'}(\{{x_i^*}'\ganyt\}_{\NT}) 	 & \text{ if } \t \neq \t'
			\end{cases}
\end{equation*}

For $R_{111a}$, consider each term in the sum over $\NT$ one at a time; let $R_{111a} = \sumT R_{111a}(\t)$. Let $\t'$ denote the original treatment under consideration. Define $\Sigma_{\t,j} = \E\left[ (x_{i,j}^*)^2 \sigma_{\t'}^2(x_i)  \pttrue^2 ( 1- \pttrue)^2 / p_{\t'}(x_i)^3 \right]$. Then proceed as follows
\begin{align*}
	 R_{111a}(\t)  & = \frac{1}{\sqrt{n}} \sumi \left(\frac{\d^{\t'} u_i \pttrue ( 1- \pttrue)}{p_{\t'}(x_i)^2} \right) \sum_{j \in \hat{S}_\D} x_{i,j}^*  (\ghatt - \gtruet)		\\
	& =  \sum_{j \in \hat{S}_\D} \left\{ \frac{1}{\sqrt{n}} \sumi \left(x_{i,j}^* \frac{\d^{\t'} u_i \pttrue ( 1- \pttrue)}{p_{\t'}(x_i)^2 \Sigma_{\t,j}^{1/2}}  \right) \right\} \Sigma_{\t,j}^{1/2}  (\ghattj - \gtruetj)		\\
	& \leq \left( \maxj \Sigma_{\t,j}^{1/2} \right) \left( \maxj \frac{1}{\sqrt{n}} \sumi x_{i,j}^* \frac{\d^{\t'} u_i \pttrue ( 1- \pttrue)}{p_{\t'}(x_i)^2 \Sigma_{\t,j}^{1/2}}    \right) \left\| \ghatt - \gtruet \right\|_1		\\
	& = O(1) O_{P_n}( \log(p) )  \left\| \ghatt - \gtruet \right\|_1= o_{P_n}(1).
\end{align*}
Convergence follows under Assumption \ref{ATE union}. For the penultimate equality, it follows from Assumptions \ref{overlap}, \ref{bounded}, and \ref{fourth moments} that $\maxj \Sigma_{t,j} = O(1)$. Finally, the center factor is shown to be $O_{P_n}( \log(p) )$ by applying the moderate deviation theory for self-normalized sums of \citeasnoun[Theorem 7.4]{delaPena-Lai-Shao2009_book} and in particular \citeasnoun[Lemma 5]{BCCH2012_Ecma}. To apply this lemma, first note that the summand of the center factor has bounded third moment and second moment bounded away from zero, from Assumptions \ref{overlap}, \ref{bounded}, \ref{fourth moments}, and the requirements of Assumptions \ref{first stage} and \ref{ATE union}. $ \Sigma_{t,j}$ normalizes the second moment, and the lemma applies under Assumptions \ref{sparsity} and the first restriction of Assumption \ref{ATE union}.

For $R_{111b}$, the results of \citeasnoun{Tanabe-Sagae1992_JRSSB} coupled with Assumption \ref{first stage} give $v_i' \bar{\H} v_i \leq C \| v_i \|_2^2$. Thus, using Assumption \ref{overlap} to bound $\maxi p_\t(x_i)^{-2} < C$, we find $R_{111b}$ may be bounded as follows:
\begin{align*}
	|R_{111b}| & \leq C \sumT  \sqrt{n} (\maxt |u_i| )  \En \left[ |{x_i^*}'(\ghatt - \gtruet)|^2\right] 		\\
	& \leq C \T \maxT \left| \sqrt{n} (\maxt |u_i| )  \En \left[ | \pthat - \pttrue |^2 \right] \right|= o_{P_n}(1),
\end{align*}
by the union bound and Assumption \ref{ATE union}, using the Assumptions \ref{overlap} and \ref{consistent} to apply \Eqref{supp-mlogit rate mvt} with the inequality reversed.

A variance bound may be applied to $R_{112}$ as in the previous proof, and we have $|R_{112}| = O_{P_n}(\bias) = o_{P_n}(1)$ by Markov's inequality.

Next, $R_{12}$ is simply bounded by
\begin{align*}
	|R_{12}| & \leq \sqrt{n} (\maxt |u_i| ) \left( \maxt \frac{1}{p_\t(x_i)^2 \hat{p}_\t(x_i) } \right) \En \left[ \left( \hat{p}_\t(x_i)   -  p_\t(x_i)\right)^2 \right]  		\\
	& \leq O_{P_n}(1) \sqrt{n} (\maxt |u_i| ) \En \left[ \left( \hat{p}_\t(x_i)   -  p_\t(x_i)\right)^2 \right] = o_{P_n}(1),
\end{align*}
where the rate follows from Assumptions \ref{overlap}, \ref{dgp asmpts}, and \ref{first stage}, and this tends to zero by Assumption \ref{ATE union}.

As in the prior proof, write $R_2 = R_{21} + R_{22}$. The same bound is used for $R_{22}$. However, for $R_{21}$, add and subtract the pseudotrue values to get $R_{21} = R_{211} + R_{212}$, where
\[R_{211} =  \frac{1}{\sqrt{n}} \sumi ({x_i^*}'\bhatt - {x_i^*}\btruet) \left(\frac{ p_\t(x_i)  -  \dt }{ p_\t(x_i) } \right)  	 \qquad \text{ and } \qquad   	R_{212} =  \frac{1}{\sqrt{n}} \sumi ({x_i^*}\btruet - \mut(x_i)) \left(\frac{ p_\t(x_i)  -  \dt }{ p_\t(x_i) } \right) \]
For the first term, define $\tilde{\Sigma}_{\t,j} = \E\left[ (x_{i,j}^*)^2 (\dt - p_\t(x_i))^2 / p_\t(x_i)^2 \right]$ and then proceed as follows:
\begin{align*}
	 R_{211}  & = \frac{1}{\sqrt{n}} \sumi \left(\frac{ p_\t(x_i)  -  \dt }{ p_\t(x_i) } \right) \sum_{j \in \hat{S}_Y} x_{i,j}^*  (\bhattj - \btruetj)		\\
	& =  \sum_{j \in \hat{S}_Y} \left\{ \frac{1}{\sqrt{n}} \sumi \frac{x_{i,j}^*  (p_\t(x_i) - \dt) / p_\t(x_i)}{  \tilde{\Sigma}_{\t,j}^{1/2} } \right\} \tilde{\Sigma}_{\t,j}^{1/2}  (\bhattj - \btruetj)		\\
	& \leq \left( \maxj \tilde{\Sigma}_{\t,j}^{1/2} \right) \left( \maxj \frac{1}{\sqrt{n}} \sumi \frac{x_{i,j}^*  (p_\t(x_i) - \dt) / p_\t(x_i)}{  \tilde{\Sigma}_{\t,j}^{1/2} }   \right) \left\| \bhatt - \btruet \right\|_1		\\
	& = O(1) O_{P_n}( \log(p) )  \left\| \bhatt - \btruet \right\|_1 = o_{P_n}(1),
\end{align*}
where the final line follows exactly as above.

A variance bound may be applied to $R_{212}$ as in the previous proof, and we have $|R_{212}| = O_{P_n}(\bias) = o_{P_n}(1)$ by Markov's inequality.\end{proof}

\begin{proof}[Proof of Theorem \ref{thm-ate normality}.]
	This claim follows directly from the prior result under the moment conditions of Assumption \ref{ATE moments}.
\end{proof}

\begin{proof}[Proof of Theorem \ref{thm-ate variance}.]
We begin with $\hat{V}_{W}(\t)$. Expanding the square and using \Eqref{supp-linearization}, rewrite $\hat{V}_{\drf}^W(\t) = \En[ \dt u_i^2 p_\t(x_i) ^{-2}] + R_{W,1} + R_{W,2} + R_{W,3}$ where
\begin{gather*}
	R_{W,1} = \En\left[ \frac{\dt u_i^2}{\hat{p}_\t(x_i)^2 p_\t(x_i)^2} \left( \hat{p}_\t(x_i) - p_\t(x_i)\right)\left( \hat{p}_\t(x_i) + p_\t(x_i)\right) \right],		\\
	R_{W,2} = \En\left[ \frac{\dt (\mut(x_i)  -  \muthat(x_i))^2 }{\hat{p}_\t(x_i)^2} \right],	\qquad \text{and} \qquad R_{W,3} = 2 \En\left[ \frac{\dt u_i (\mut(x_i)  -  \muthat(x_i)) }{\hat{p}_\t(x_i)^2} \right].
\end{gather*}
Using H\"older's inequality, Assumptions \ref{overlap}, \ref{ATE moments}, and \ref{consistent}, we have the following
	\[R_{W,1}  \leq \left( \maxt \frac{  \hat{p}_\t(x_i)  +  p_\t(x_i) }{ \hat{p}_\t(x_i)^2 p_\t(x_i)^2 } \right) \En[\dt |u_i|^4]^{1/2} \En[ \dt (\hat{p}_\t(x_i) - p_\t(x_i))^2]^{1/2} = o_{P_n}(1), 		\]
	\[R_{W,2}  \leq \left( \maxt \frac{  1 }{ \hat{p}_\t(x_i)^2 } \right) \En[ \dt (\muthat(x_i) - \mut(x_i))^2] = o_{P_n}(1),\]
and,
	\[R_{W,3}  \leq 2 \left( \maxt \frac{  1 }{ \hat{p}_\t(x_i)^2 } \right) \En[\dt |u_i|^2]^{1/2}  \En[ \dt (\muthat(x_i) - \mut(x_i))^2]^{1/2} = o_{P_n}(1),\]
where $ \En[|u_i|^4] = O_{P_n}(1)$ from the inequality of \citeasnoun{vonBahr-Esseen1965_AoMS}. From the same inequality it follows that $\En[ \dt u_i^2 p_\t(x_i) ^{-2}] - V_{\drf}^W(\t)| = o_{P_n}(1)$, under Assumptions \ref{overlap} and \ref{fourth moments}.

Next consider the ``between'' variance estimator, $\hat{V}_{\drf}^B$. For any $\t \NTbar$ and $\t' \in \NTbar$, define 
	\[R_{B,1}(\t, \t') =  \En\left[ (\muthat(x_i) - \mut(x_i)) (\hat{\mu}_{\t'}(x_i) - \mu_{\t'}(x_i)) \right],\]
	\[R_{B,2}(\t, \t') =  \muthat \En\left[ \hat{\mu}_{\t'}(x_i) - \mu_{\t'}(x_i) \right],   \quad \text{and} \quad  R_{B,3}(\t, \t') =  \En\left[ \mu_{\t}(x_i) (\hat{\mu}_{\t'}(x_i) - \mu_{\t'}(x_i)) \right].\]
From H\"older's inequality, Assumption \ref{consistent}, Theorem \ref{thm-ate normality}, the von Bahr and Esseen inequality, and Assumptions \ref{fourth moments} and \ref{ATE moments} it follows that $R_{B,k}(\t,\t') = o_{P_n}(1)$ for $k \in \N_3$ and all pairs $(\t, \t') \in \N_\t^2$. With this in mind, we decompose 
\begin{align*}
	\hat{V}_{\drf}^B(\t,\t') & = \En\left[ \mut(x_i) \mu_{\t'}(x_i) \right] - \muthat \En\left[ \mu_{\t'}(x_i)\right]  - \hat{\mu}_{\t'} \En\left[ \mut(x_i)\right] +  \muthat \hat{\mu}_{\t'}		\\
	& \qquad + R_{B,1}(\t,\t') + R_{B,2}(\t, \t') + R_{B,2}(\t', \t) + R_{B,3}(\t, \t') + R_{B,3}(\t', \t).
\end{align*}
Consistency of $\hat{V}_{\drf}^B(\t,\t') $ now follows from the von Bahr and Esseen inequality and Theorem \ref{thm-ate normality}.
\end{proof}

\begin{proof}[Proof of Corollary \ref{thm-ate uniform}.]
	Suppose the result did not hold. Then, there would exist a subsequence $P_m \in \bm{P}_m$, for each $m$, such that 
	\[\lim_{m \to \infty} \left| \P_{P_m} \left[  G(\drf) \in \left\{ G(\drfhat) \pm c_\alpha \sqrt{\nabla_G(\drfhat)  \hat{V} \nabla_G'(\drfhat)  / n}\right\} \right] - (1 - \alpha) \right| >0 .\]
	But this contradicts Theorem \ref{thm-ate}, under which $ (\nabla_G(\drfhat)  \hat{V} \nabla_G'(\drfhat)/n)^{-1/2} ( G(\drfhat) - G(\drf))$ is asymptotically standard normal under the sequence $P_m$.
\end{proof}

%%%%%%%%%%%%%%%%%%%%%%%%%%%%
%%%%%%%%%%%%%%%%%%%%%%%%%%%%
\section{Proofs for Group Lasso Selection and Estimation of Multinomial Logistic Models}

Unless otherwise noted, all bounds in this section are nonasymptotic. We will use generic notation $X^*$, $\delta$, $s$, etc, as this section deals only with multinomial logistic models.

%%%%%%%%%%%%%%%%%%%%%%%%%%%%
\subsection{Lemmas}

\begin{lemma}[Score Bound]
	\label{supp-mlogit score bound}
	For $\lambda_\D$ and $\mathcal{P}$ defined respectively in \Eqref{lambda} and \Eqref{probability} we have
	\[\P \left[ \maxj \| \En [(p_\t(x_i) - \dt)  x_{i,j}^* ] \|_2 \geq \frac{\lambda_\D}{2} \right] \leq \mathcal{P}. \]
\end{lemma}
\begin{proof}
The residuals $v_{\t,i} = p_\t(x_i) - \dt$ are conditionally mean-zero by definition and satisfy $\E[v_{\t,i}^2 \vert x_i ] \leq 1$. Using this, Assumption \ref{iid}, and the definition of $\mathcal{X}$, we find that
	\[\E\left[ \| \En [v_{\t,i} x_{i,j}^* ] \|_2^2 \right]  =  \sumT \E\left[ \En[ v_{\t,i} x_{i,j}^* ]^2 \right] = \sumT \frac{1}{n} \E [v_{\t,i}^2 (x_{i,j}^*)^2]  \leq \frac{\mathcal{X}^2 \T}{n} \]
uniformly in $j \in \Np$. Define the mean-zero random variables $\xi_{t,j}$ as: 
	\[\xi_{t,j} = (\En[ v_{\t,i} x_{i,j}^* ])^2 - \frac{1}{n} \E[V_\t^2 {X_j^*}^2].\]
Using this definition and the above bound after inserting the definition of $\lambda_\D$, setting $r_n = \T^{-1/2} \log(p \vee n)^{3/2 + \delta}$, and squaring both sides, we have
\begin{align}
	\P  \left[ \maxj \| \En [(p_\t(x_i) - \dt)  x_{i,j}^* ] \|_2 \geq \frac{\lambda_\D}{2} \right] 
	& = \P \left[ \maxj \| \En [v_{\t,i} x_{i,j}^* ] \|_2  \geq   \frac{ \mathcal{X} \sqrt{\T}   }{\sqrt{n} } \left( 1 + r_n \right)^{1/2}\right]		\nonumber  \\
	& = \P \left[ \maxj \| \En [v_{\t,i} x_{i,j}^* ] \|_2^2  \geq   \frac{ \mathcal{X}^2 \T   }{ n  } \left( 1 + r_n \right)\right]		\nonumber  \\
	& = \P \left[ \maxj \| \En [v_{\t,i} {x_{i,j}^*}] \|_2^2  -  \frac{ \mathcal{X}^2  \T }{n}     \geq  \frac{\mathcal{X}^2  \T r_n }{n} \right]		\nonumber  \\
	& \leq \P \left[ \maxj \sumT \xi_{t,j}  \geq  \frac{\mathcal{X}^2  \T r_n }{n} \right]		\nonumber  \\
	& \leq \E \left[ \maxj \left| \sumT \xi_{t,j} \right|\right]   \frac{n}{ \mathcal{X}^2  \T r_n },			\label{supp-mlogit maximal1}
\end{align}
where final line follows from Markov's inequality.

Next, applying Lemma 9.1 of \citeasnoun{Lounici-etal2011_AoS} (with their $m=1$ and hence $c(m) = 2$) followed by Jensen's inequality and Assumption \ref{fourth moments}, we find that
\begin{align}
	\E \left[ \maxj \left| \sumT \xi_{t,j} \right|\right] & \leq ( 8 \log(2p))^{1/2} \E \left[ \left( \sumT \maxj \xi_{t,j}^2 \right)^{1/2} \right]	\nonumber  \\
	& \leq  ( 8 \log(2p))^{1/2} \left( \E \left[  \sumT \maxj \xi_{t,j}^2 \right] \right)^{1/2}  	\nonumber  \\
	& \leq  4 \log(2p)^{1/2} \left(  \sumT  \frac{\mathcal{X}^4 }{n^2} + \sumT \E \left[ \maxj \left| \En[ v_{\t,i} {x_{i,j}^*} ] \right| ^4  \right] \right)^{1/2}.	\label{supp-mlogit maximal2}
\end{align}
The leading 4 is $\sqrt{8} \sqrt{2}$, where $\sqrt{2}$ is a byproduct of applying the inequality $(a - b)^2 \leq 2(a^2 + b^2)$ to $\xi_{t,j}^2$. Again using Lemma 9.1 of \citeasnoun{Lounici-etal2011_AoS} (with their $m=4$, and $c(m)=12$ since $c(4) \geq (e^{4-1} - 1 )/2 + 2 \approx 11.54$), we bound the expectation in the second term above as follows:
\begin{align}
	 & \E \left[ \maxj  \left| \En [ v_{\t,i} {x_{i,j}^*} ] \right| ^4  \right]  \leq [ 8 \log(12p) ] ^{4/2} \E \left[ \left( \sumi \maxj \left| \frac{v_{\t,i} {x_{i,j}^*} }{n}\right|^2  \right)^{4/2} \right] 		 \leq \frac{64 \log(12p)^2 \mathcal{X}^4 }{n^2},	\label{supp-mlogit maximal3}
\end{align}
using Assumptions \ref{iid} and \ref{bounded}.

Now, inserting the results of Eqns.\ \eqref{supp-mlogit maximal2} and \eqref{supp-mlogit maximal3} into \Eqref{supp-mlogit maximal1}, we have
\begin{align*}
	\P \left[ \maxj \| \En [v_{\t,i} {x_{i,j}^*}] \|_2 \geq \frac{\lambda_\D }{4} \right] & \leq \frac{ 4 n  \log(2p)^{1/2} }{\T \mathcal{X}^2  r_n }  \left(  \sumT  \frac{\mathcal{X}^4 }{n^2} + \sumT \frac{64 \log(12p)^2 \mathcal{X}^4 }{n^2} \right)^{1/2} \\
& \leq \frac{ 4\log(2p)^{1/2}  }{ r_n \sqrt{\T}  } [ 1 + 64 \log(12 p)^2]^{1/2} = \mathcal{P},
\end{align*}
using the choice $r_n = \T^{-1/2} \log(p \vee n)^{3/2 + \delta}$.
\end{proof}

\begin{lemma}[Estimate Sparsity]
	\label{supp-mlogit S-hat}
	With probability at least $1 - \mathcal{P}$
	\[  | \tilde{S}^\D |  \leq  \frac{4}{\lambda_\D^2} \overline{\phi}\{Q, \tilde{S}^\D \}  \sumT   \En \left[ (\pttilde - p_\t(x_i))^2 \right] .\]
\end{lemma}
\begin{proof}
First, by Karush-Kuhn-Tucker conditions for \eqref{grplasso}, for all $\t \in \NT$, if $\gtildej \neq 0$ it must satisfy
\begin{equation}
	\label{supp-mlogit KKT}
	\En[{x_{i,j}^*} (\pttilde - \dt)] = \lambda_\D \frac{\gtildetj}{\|\gtildej \|_2}.
\end{equation}
Hence, taking the $\ell_2$-norm over $\t \in \NT$ for fixed $j \in \tilde{S}^\D$, adding and subtracting the true propensity score, using the triangle inequality, and the score bound \eqref{supp-mlogit score bound}, we find that
\begin{align*}
	\lambda_\D & =  \left\| \En[{x_{i,j}^*} (\pttilde - \dt)] \right\|_2		\\
	& \leq  \left\| \En[{x_{i,j}^*} (p_\t(x_i) - \dt)] \right\|_2 + \left\| \En[{x_{i,j}^*}( \pttilde - p_\t(x_i)) ] \right\|_2		\\
	& \leq  \lambda_\D/2  + \left\| \En[{x_{i,j}^*}( \pttilde - p_\t(x_i) ) ] \right\|_2.
\end{align*}
Let $\tilde{\bm{P}}_\t$ be the vector of $\{\pttilde\}_{i=1}^n$ and $\bm{P}_\t$ collect $\{p_\t(x_i) \}_{i=1}^n$. Collecting terms, then squaring both sides and summing over $j \in \tilde{S}^\D$ (i.e. applying $\| \cdot \|_2^2$ over $j \in \tilde{S}^\D$ to both sides) yields
\begin{align*}
	 \sum_{j \in \tilde{S}^\D} \lambda_\D^2 & \leq 4 \sum_{j \in \tilde{S}^\D} \sumT \En[{x_{i,j}^*}( \pttilde  -  p_\t(x_i) ) ]^2	 		\\
	& = 4 \sumT  \frac{1}{n^2}  \left\|  \left[ \Xn' (\tilde{\bm{P}}_\t  -  \bm{P}_\t)\right]_{j \in \tilde{S}^\D} \right\|_2^2	\\
	& \leq 4 \sumT  \frac{ \overline{\phi}\{Q, \tilde{S}^\D \} }{n}   \left\| \tilde{\bm{P}}_\t  -  \bm{P}_\t  \right\|_{2,n}^2	\\
	& \leq 4 \overline{\phi}\{Q, \tilde{S}^\D \}  \sumT   \En \left[ (\pttilde  -  p_\t(x_i))^2 \right].
\end{align*}
The result now follows, as the left-hand side is equal to $|\tilde{S}^\D| \lambda_\D^2$.
\end{proof}

\begin{lemma}[Bounds in $\ell_2/\ell_1$ norm]
	\label{supp-mlogit cone}
	With probability $1 - \mathcal{P}$ the vector $\dtilde = \gtilde - \gtrue$ satisfies
\begin{equation*}
	\vertiii{\dtilde} \leq \left\{ \frac{ 5 \sqrt{|S_*|} }{\kappa_\D}  \vee  \frac{ 10 \bias^d \sqrt{\T}}{\lambda_\D} \right\}\En [ \| \{ {x_i^*}' \dtildet\}_{\NT}  \|_2^2]^{1/2}.
\end{equation*}
and
\begin{equation*}
	\vertiii{ \tilde{\delta}_{\newdot, S_*}} \leq  \left\{ \frac{  \sqrt{|S_*|} }{\kappa_\D}  \vee  \frac{ 2 \bias^d \sqrt{\T}}{\lambda_\D} \right\}\En [ \| \{ {x_i^*}' \dtildet\}_{\NT}  \|_2^2]^{1/2}.
\end{equation*}
\end{lemma}
\begin{proof}
By the Cauchy-Schwarz inequality and Lemma \ref{supp-mlogit score bound}, 
\begin{align}
	\sumT \En \left[(p_\t(x_i) - \dt ) {x_i^*}' \dtildet \right]   &  = \sumj \sumT \En \left[(p_\t(x_i) - \dt) {x_{i,j}^*} \right] \dtildetj		\nonumber \\
	& \leq \sumj \sqrt{\sumT \En \left[(p_\t(x_i) - \dt) {x_{i,j}^*} \right]^2} \sqrt{\sumT \dtildetj^2}		\nonumber \\
	& \leq \maxj \left\{ \left\| \En \left[(p_\t(x_i) - \dt) {x_{i,j}^*} \right] \right\|_2 \right\}  \sumj  \left\| \dtildej  \right\|_2			\nonumber \\
	& \leq \frac{\lambda_\D}{2} \vertiii{ \dtilde }, 	\label{supp-mlogit score variance}
\end{align}
with probability at least $1 - \mathcal{P}$. Applying the Cauchy-Schwarz inequality, the bias condition of Assumption \ref{sparsity}, and Cauchy-Schwarz again yield
\begin{align}
	\sumT \En \left[(\pttrue - p_\t(x_i)) {x_i^*}' \dtildet \right]    
	& \leq    \sumT \En \left[(\pttrue - p_\t(x_i) )^2\right]^{1/2}  \En \left[({x_i^*}' \dtildet)^2 \right]^{1/2}		\nonumber \\
	& \leq    \bias^d \sumT  \En \left[({x_i^*}' \dtildet)^2 \right]^{1/2}		\nonumber  \\
	& \leq    \bias^d \sqrt{\T} \sqrt{ \sumT  \En \left[({x_i^*}' \dtildet)^2 \right] }		\nonumber  \\
	& =    \bias^d \sqrt{\T} \En [ \| \{ {x_i^*}' \dtildet\}_{\NT}  \|_2^2]^{1/2}.		\label{supp-mlogit score bias} 
\end{align}

Combining Equations \eqref{supp-mlogit score variance} and \eqref{supp-mlogit score bias}, we have, probability at least $1 - \mathcal{P}$,
\begin{align}
	\sumT \En \left[(\pttrue - \dt ) {x_i^*}' \dtildet \right] &  = \sumT \En \left[(p_\t(x_i) - \dt ) {x_i^*}' \dtildet \right] 		\nonumber \\
		& \qquad  + \sumT \En \left[(\pttrue - p_\t(x_i)) {x_i^*}' \dtildet \right]       	\nonumber \\
	& \leq  \frac{\lambda_\D}{2} \vertiii{ \dtilde } + \bias^d \sqrt{\T} \En [ \| \{ {x_i^*}' \dtildet\}_{\NT}  \|_2^2]^{1/2}.	\label{supp-mlogit score}
\end{align}

By the optimality of $\dtilde$, we have
\[\M(\gtrue + \dtilde) + \lambda_\D \vertiii{\gtrue + \dtilde} \leq  \M(\gtrue) + \lambda_\D \vertiii{\gtrue},\]
implying
\begin{align*}
	\lambda_\D \left\{ \vertiii{\gtrue}  -  \vertiii{\gtrue + \dtilde}  \right\}  & \geq \M(\gtrue + \dtilde) -   \M(\gtrue)		\\
	& \geq \sumT \En \left[(\pttrue - \dt) {x_i^*}' \dtildet \right],
\end{align*}
applying the convexity of $\M$. Using the bound in \Eqref{supp-mlogit score} and rearranging we find that
	\[0  \leq  \lambda_\D \left\{ \vertiii{\gtrue}  -  \vertiii{\gtrue + \dtilde}  \right\} + \frac{\lambda_\D}{2} \vertiii{ \dtilde } + \bias^d \sqrt{\T} \En [ \| \{ {x_i^*}' \dtildet\}_{\NT}  \|_2^2]^{1/2}.\]
Dividing through $\lambda_\D$ and decomposing the supports, we find that
\begin{align*}
	0 & \leq  \frac{1}{2} \vertiii{ \dtilde }    +  \left\{ \vertiii{\gtrue} -  \vertiii{\gtrue + \dtilde} \right\} +  \frac{\bias^d \sqrt{\T}}{\lambda_\D} \En [ \| \{ {x_i^*}' \dtildet\}_{\NT}  \|_2^2]^{1/2}		\\
	& = \frac{1}{2} \vertiii{ \tilde{\delta}_{\newdot, S_*}} +  \frac{1}{2} \vertiii{ \tilde{\delta}_{\newdot, S_*^c}} + \vertiii{\gamma^*_{\newdot, S_*}} - \vertiii{\gamma^*_{\newdot, S_*} + \tilde{\delta}_{\newdot, S_*}} - \vertiii{ \tilde{\delta}_{\newdot, S_*^c}} +  \frac{\bias^d \sqrt{\T}}{\lambda_\D} \En [ \| \{ {x_i^*}' \dtildet\}_{\NT}  \|_2^2]^{1/2},
\end{align*}
where the second line follows because $\gamma^*_{\newdot, S_*^c} = 0$. Collecting terms and applying the triangle inequality yields
\begin{align*}
	\frac{1}{2} \vertiii{ \tilde{\delta}_{\newdot, S_*^c}} & \leq \frac{1}{2} \vertiii{ \tilde{\delta}_{\newdot, S_*}} + \vertiii{\gamma^*_{\newdot, S_*}} - \vertiii{\gamma^*_{\newdot, S_*} + \tilde{\delta}_{\newdot, S_*}} +  \frac{\bias^d \sqrt{\T}}{\lambda_\D} \En [ \| \{ {x_i^*}' \dtildet\}_{\NT}  \|_2^2]^{1/2}   		\\
	& \leq \frac{1}{2} \vertiii{ \tilde{\delta}_{\newdot, S_*}} + \left| \vertiii{\gamma^*_{\newdot, S_*}} - \vertiii{\gamma^*_{\newdot, S_*} + \tilde{\delta}_{\newdot, S_*}}\right|  +  \frac{\bias^d \sqrt{\T}}{\lambda_\D} \En [ \| \{ {x_i^*}' \dtildet\}_{\NT}  \|_2^2]^{1/2}  		\\
	& \leq \frac{1}{2} \vertiii{ \tilde{\delta}_{\newdot, S_*}} +  \vertiii{\gamma^*_{\newdot, S_*} - \left( \gamma^*_{\newdot, S_*} + \tilde{\delta}_{\newdot, S_*} \right)}  +  \frac{\bias^d \sqrt{\T}}{\lambda_\D} \En [ \| \{ {x_i^*}' \dtildet\}_{\NT}  \|_2^2]^{1/2}	  		\\
	& =  \frac{1}{2} \vertiii{ \tilde{\delta}_{\newdot, S_*}} +  \vertiii{ \tilde{\delta}_{\newdot, S_*}} +  \frac{\bias^d \sqrt{\T}}{\lambda_\D} \En [ \| \{ {x_i^*}' \dtildet\}_{\NT}  \|_2^2]^{1/2}.
\end{align*}
Therefore with probability at least $1 - \mathcal{P}$
\begin{equation}
	\label{supp-mlogit cases}
	\vertiii{ \tilde{\delta}_{\newdot, S_*^c}} \leq  3 \vertiii{ \tilde{\delta}_{\newdot, S_*}} +  \frac{2 \bias^d \sqrt{\T}}{\lambda_\D} \En [ \| \{ {x_i^*}' \dtildet\}_{\NT}  \|_2^2]^{1/2}
\end{equation}

Consider two cases based on the upper bound in \eqref{supp-mlogit cases}. First, suppose that $\dtilde$ obeys the cone constraint of \Eqref{RE full} in the definition of $\kappa_\D^2$, such that
	\[ \vertiii{ \tilde{\delta}_{\newdot, S_*^c}} \leq  4 \vertiii{ \tilde{\delta}_{\newdot, S_*}}.\]
This implies
\begin{align}
	\vertiii{\dtilde} = \vertiii{ \tilde{\delta}_{\newdot, S_*}} + \vertiii{ \tilde{\delta}_{\newdot, S_*^c}} & \leq 5 \vertiii{ \tilde{\delta}_{\newdot, S_*}} 		\nonumber \\
	& \leq 5 \sqrt{|S_*|} \left\| \tilde{\delta}_{\newdot, S_*}\right\|_2 			\nonumber \\
	& \leq  \frac{ 5 \sqrt{|S_*|} }{\kappa_\D} \En [ \| \{ {x_i^*}' \dtildet\}_{\NT}  \|_2^2]^{1/2},     \label{supp-mlogit group norm 1}
\end{align}
by the Cauchy-Schwarz inequality, the restricted eigenvalue definition of \Eqref{RE full}, and noting that $\sumT \dtildet' Q \dtildet =  \En [ \| \{ {x_i^*}' \dtildet\}_{\NT}  \|_2^2]$. Collecting across the second and third inequalities yields
\begin{equation}
	\label{supp-mlogit group norm 2}
	\vertiii{ \tilde{\delta}_{\newdot, S_*}} \leq \frac{\sqrt{|S_*|} }{\kappa_\D} \En [ \| \{ {x_i^*}' \dtildet\}_{\NT}  \|_2^2]^{1/2}.
\end{equation}

On the other hand, if the cone constraint fails, then 
	\[  \vertiii{ \tilde{\delta}_{\newdot, S_*}}  <  \frac{1}{4} \vertiii{ \tilde{\delta}_{\newdot, S_*^c}}.\]
Using this for the first and third inequalities, and \Eqref{supp-mlogit cases} for the second, we have
\begin{align*}
	\vertiii{\dtilde} = \vertiii{ \tilde{\delta}_{\newdot, S_*}} + \vertiii{ \tilde{\delta}_{\newdot, S_*^c}} & \leq \frac{5}{4} \vertiii{ \tilde{\delta}_{\newdot, S_*^c}} 		 \\
	& \leq \frac{15}{4} \vertiii{ \tilde{\delta}_{\newdot, S_*}} +  \frac{5}{2} \frac{ \bias^d \sqrt{\T}}{\lambda_\D} \En [ \| \{ {x_i^*}' \dtildet\}_{\NT}  \|_2^2]^{1/2}		 \\
	& \leq \frac{15}{16}\vertiii{ \tilde{\delta}_{\newdot, S_*^c}}  +  \frac{5}{2} \frac{ \bias^d \sqrt{\T}}{\lambda_\D} \En [ \| \{ {x_i^*}' \dtildet\}_{\NT}  \|_2^2]^{1/2}
\end{align*}
Combining the right hand side of the first line with third lines yields 
	\[ \vertiii{ \tilde{\delta}_{\newdot, S_*^c}}  \leq 8 \frac{ \bias^d \sqrt{\T}}{\lambda_\D} \En [ \| \{ {x_i^*}' \dtildet\}_{\NT}  \|_2^2]^{1/2}.\]
Plugging this back into the last line we obtain the bound
\begin{equation}
	\label{supp-mlogit group norm 3}
	\vertiii{\dtilde}  \leq  10 \frac{ \bias^d \sqrt{\T}}{\lambda_\D} \En [ \| \{ {x_i^*}' \dtildet\}_{\NT}  \|_2^2]^{1/2},
\end{equation}
while instead, plugging it into the failure of the cone constraint yields
\begin{equation}
	\label{supp-mlogit group norm 4}
	\vertiii{ \tilde{\delta}_{\newdot, S_*}} \leq 2 \frac{ \bias^d \sqrt{\T}}{\lambda_\D} \En [ \| \{ {x_i^*}' \dtildet\}_{\NT}  \|_2^2]^{1/2}.
\end{equation}
Combining Equations \eqref{supp-mlogit group norm 1} and \eqref{supp-mlogit group norm 3} gives the first claim of the lemma and Equations \eqref{supp-mlogit group norm 2} and \eqref{supp-mlogit group norm 4} give the second. 
\end{proof}

%%%%%%%%%%%%%%%%%%%%%%%%%%%%
\subsection{Proof of Theorem \ref{thm-mlogit}}

Define $\dtilde = \gtilde - \gtrue$. By the optimality of $\dtilde$, we have
\[ \M(\gtrue + \dtilde)   + \lambda_\D \vertiii{\gtrue + \dtilde} \leq  \M(\gtrue) + \lambda_\D \vertiii{\gtrue}.\]
Rearranging and subtracting the score, we have
\begin{align}
	\begin{split}
		\label{supp-mlogit start}
		& \M(\gtrue + \dtilde)  -  \M(\gtrue)  -  \sumT \En \left[(\pttrue - \dt) {x_i^*}' \right] \dtildet		\\
		& \qquad \quad \leq  \lambda_\D \left\{ \vertiii{\gtrue}  -  \vertiii{\gtrue + \dtilde}  \right\}  -  \sumT \En \left[(\pttrue - \dt) {x_i^*}' \right] \dtildet.
	\end{split}
\end{align}
The proof proceeds by deriving a further upper bound to the right and a quadratic lower bound of the left. The combination of these will yield a bound on $\En[({x_i^*}'\dtildet )^2] ^{1/2}$. 

Let us begin with the right side of \Eqref{supp-mlogit start}. For the penalized difference of coefficients we have
\[\vertiii{\gamma^*_{\newdot, S_*^c}} - \vertiii{\gamma^*_{\newdot, S_*^c} + \tilde{\delta}_{\newdot, S_*^c}} = \vertiii{ \tilde{\delta}_{\newdot, S_*^c}},\]
because $\gamma^*_{\newdot, S_*^c} = 0$. Therefore,
\begin{align*}
	\vertiii{\gtrue} - \vertiii{\gtrue + \dtilde} & = \vertiii{\gamma^*_{\newdot, S_*}} - \vertiii{\gamma^*_{\newdot, S_*} + \tilde{\delta}_{\newdot, S_*}} - \vertiii{ \tilde{\delta}_{\newdot, S_*^c}}		 \\
	& \leq \vertiii{\gamma^*_{\newdot, S_*}} - \vertiii{\gamma^*_{\newdot, S_*} + \tilde{\delta}_{\newdot, S_*}}		 \\
	& \leq \left| \vertiii{\gamma^*_{\newdot, S_*}} - \vertiii{\gamma^*_{\newdot, S_*} + \tilde{\delta}_{\newdot, S_*}} \right|		 \\
	& \leq \vertiii{\gamma^*_{\newdot, S_*} - \left(\gamma^*_{\newdot, S_*} + \tilde{\delta}_{\newdot, S_*}\right) }		 = \vertiii{ \tilde{\delta}_{\newdot, S_*}}, 
\end{align*}
where the first inequality reflects dropping the nonpositive final term (the norm is nonnegative) and the third inequality follows from the triangle inequality. Using this result for the first term and the bound \eqref{supp-mlogit score} for the second, the right side of \Eqref{supp-mlogit start} is bounded by
\begin{align}
	& \lambda_\D \vertiii{ \tilde{\delta}_{\newdot, S_*}}   +  \frac{\lambda_\D}{2} \vertiii{ \dtilde } + \bias^d \sqrt{\T} \En [ \| \{ {x_i^*}' \dtildet\}_{\NT}  \|_2^2]^{1/2} 		\nonumber \\
	& \qquad \leq \left(         \lambda_\D \left\{ \frac{  \sqrt{|S_*|} }{\kappa_\D}  \vee  \frac{ 2 \bias^d \sqrt{\T}}{\lambda_\D} \right\}         +        \frac{\lambda_\D}{2}\left\{ \frac{ 5 \sqrt{|S_*|} }{\kappa_\D}  \vee  \frac{ 10 \bias^d \sqrt{\T}}{\lambda_\D} \right\}       +      \bias^d \sqrt{\T}             \right)   \En [ \| \{ {x_i^*}' \dtildet\}_{\NT}  \|_2^2]^{1/2} 		\nonumber \\
	& \qquad \leq \left(    6 \frac{ \lambda_\D \sqrt{|S_*|} }{\kappa_\D}        +     8  \bias^d \sqrt{\T}             \right)   \En [ \| \{ {x_i^*}' \dtildet\}_{\NT}  \|_2^2]^{1/2},			\label{supp-mlogit upper} 	
\end{align}
where the second inequality applies the results of Lemma \ref{supp-mlogit cone} and the third bounds the maximum by the sum.

Now turn to the left side of \Eqref{supp-mlogit start}. Our goal is to show that this is bounded below by a quadratic function. We apply the bounds for \possessivecite{Bach2010_EJS} modified self-concordant functions. To show that $\M(\cdot)$ belongs to this class, we must bound the third derivative in terms of the Hessian. Recall that $\ptany = \exp\{ {x_i^*}'\ganyt\}/\left(1 + \sum_{\NT} \exp\{ {x_i^*}'\ganyt\} \right)$ and the $\T$-square matrix $\H(\{{x_i^*}'\ganyt\}_{\NT})$ has $(\t, \t') \in \NT^2$ entry given by
\begin{equation*}
	\H(\{{x_i^*}'\ganyt\}_{\NT})_{[\t,\t']} = \begin{cases}
				\ptany (1- \ptany) 	& \text{ if } \t = \t'	\\
				- \ptany \hat{p}_{\t'}(\{{x_i^*}'\ganyt\}_{\NT}) 	 & \text{ if } \t \neq \t'
			\end{cases}
\end{equation*}
First, note that $\M(\gany)$ can be written as
\[\M(\gany) = \En\biggl[ \log\biggl(1 + \sumT \exp\{ {x_i^*}'\ganyt\} \biggr)   -  \sumT \dt ({x_i^*}'\ganyt)\biggr].\]
Define $F: \mathbb{R}^\T \to \mathbb{R}$ as $F(w) = \log\left(1 + \sumT \exp(w_\t)\right)$, so that $\M(\gany) = \En\left[ F(w_i) - \sumT \dt w_{i,\t} \right]$, where $w_{i,\t} = {x_i^*}'\ganyt$ and $w_i = \{w_{i,\t}\}_{\NT}$. Then for any $w \in \mathbb{R}^\T$, $v \in \mathbb{R}^\T$, and scalar $\alpha$, define $g(\alpha) = F(w + \alpha v): \mathbb{R} \to \mathbb{R}$. We verify the conditions of \citeasnoun[Lemma 1]{Bach2010_EJS} for this $g(\alpha)$ and $F(w)$. This involves finding the third derivative of $g(\alpha)$, and bounding it in terms of the second (i.e. the Hessian). To this end, note that the multinomial function has the property that $\partial \ptany / \partial \ganyt = \ptany (1- \ptany) {x_i^*}$ and $\partial \ptany / \partial \gamma_{\t',\newdot} = - \ptany \hat{p}_{\t'}(\{{x_i^*}'\ganyt\}_{\NT}) {x_i^*}$. From these, we find that 
	\[ g'(\alpha) = v'F'(w + \alpha v) = \sumT v_\t \hat{p}_\t(w + \alpha v)\]
and
	\[g''(\alpha) = v'F''(w + \alpha v) v = v' \H(w + \alpha v) v.\]
To bound $g'''(\alpha)$, we again use the derivatives of $\ptany$ to find the derivatives of elements $\H(w)$. Routine calculations give, for any $r \neq s \neq \t$:
\begin{align*}
	\partial \H(w)_{\t,\t}/ \partial w_\t & =  \hat{p}_\t(w) ( 1- \hat{p}_\t(w)) (1 - 2 \hat{p}_\t(w)) = \H(w)_{\t,\t} (1 - 2 \hat{p}_\t(w))		\\
	\partial \H(w)_{\t,\t}/ \partial w_r & = - \hat{p}_\t(w) \hat{p}_r(w) ( 1- \hat{p}_\t(w))  + \hat{p}_\t(w)^2\hat{p}_r(w)  = \H(w)_{\t,\t} (\hat{p}_\t(w)\hat{p}_r(w)( 1- \hat{p}_\t(w))^{-1} -  \hat{p}_r(w))		\\
	\partial \H(w)_{\t,s}/ \partial w_\t & =  - \hat{p}_\t(w) \hat{p}_s(w)  (1 - 2 \hat{p}_\t(w)) = \H(w)_{\t,s} (1 - 2 \hat{p}_\t(w))		\\
	\partial \H(w)_{\t,s}/ \partial w_r & =  - \hat{p}_\t(w) \hat{p}_s(w)  ( - 2 \hat{p}_r(w)) = \H(w)_{\t,s} ( -  2 \hat{p}_r(w)).
\end{align*}
Each derivative returns the same Hessian element multiplied by term bounded by 2 in absolute value. Let $a_r$ represent this factor. Then we bound
\begin{multline*}
	g'''(\alpha)  = \left| \sum_{r \in \NT} v_r  \left. \frac{\partial v' \H(\tilde{w}) v  }{\partial w_r}\right|_{\tilde{w} = w + \alpha v} \right| = \left| \sum_{r \in \NT} v_r  v' \H(w + \alpha v) v a_r \right|  	\\	 \leq    \sum_{r \in \NT} v' \H(w + \alpha v) v |v_r|   |a_r |  \leq   2  v' \H(w + \alpha v) v \sum_{r \in \NT}  |v_r|    = 2 \| v \|_1 g''(\alpha) \leq 2 \sqrt{\T}\| v \|_2 g''(\alpha) .
\end{multline*}
Applying \possessivecite{Bach2010_EJS} Lemma 1 to each observation, as in \citeasnoun{Belloni-Chernozhukov-Wei2013_logit}\footnote{\citeasnoun{Kwemou2012_logit} also applied \possessivecite{Bach2010_EJS} to study sparse logistic regression}, with $w_i = \{{x_i^*}'\gtruet \}_{\NT}$ and $v_i = \{ {x_i^*}' \dtildet\}_{\NT}$ we get the lower bound
\begin{align}
	M(\gtrue + \dtilde) &  - \M(\gtrue) -  \sumT \En \left[(\pttrue - \dt) {x_i^*}' \right] \dtildet  		\nonumber \\
	&   \geq     \En \left[ \frac{v_i' \H(\{{x_i^*}'\ganyt\}_{\NT}) v_i}{ 4  \T  \| v_i  \|_2^2} \left( e^{-2\| v_i  \|_2} + 2 \| v_i  \|_2 - 1\right) \right]		\nonumber \\
	& \geq  \En \left[ \frac{v_i' \H(\{{x_i^*}'\ganyt\}_{\NT}) v_i}{ 4  \T  \| v_i  \|_2^2} \left( 2 \| v_i  \|_2^2 - \frac{4}{3} \| v_i  \|_2^3\right) \right],	\label{supp-mlogit Bach}
\end{align}
where the second inequality follows from \citeasnoun[Lemma 9]{Belloni-Chernozhukov-Wei2013_logit}. 

\citeasnoun[Theorem 1]{Tanabe-Sagae1992_JRSSB} give $\H(\{{x_i^*}'\gtruet\}_{\NT}) \geq \phi_{\min}\{ \H(\{{x_i^*}'\gtruet\}_{\NT})\} \mathcal{I}_\T $, in the positive definite sense, where $\phi_{\min}(A)$ denotes the smallest eigenvalue of $A$ and $\mathcal{I}_T$ is the $\T \times \T$ identity matrix. Then
\[ \phi_{\min}\{ \H(\{{x_i^*}\gtruet\}_{\NT})\} \geq \det\{ \H(\{{x_i^*}'\ganyt\}_{\NT}) \} = \prod_{\t \in \NTbar} \pttrue \geq \left(\frac{p_{\min}}{A_p}\right)^{\Tbar},\]
where $p_0(\{{x_i^*}'\gtruet\}_{\NT}) = 1 - \sum_{\t \in \NT} \pttrue$ and the first inequality is also due to \citeasnoun{Tanabe-Sagae1992_JRSSB}. These results imply that $v_i' \H(\{{x_i^*}'\ganyt\}_{\NT}) v_i \geq (p_{\min} / A_p)^{\Tbar} v_i'\mathcal{I}_\T v_i = (p_{\min} / A_p)^{\Tbar} \| v_i  \|_2^2$ and therefore
\begin{align}
	\En \left[ \frac{v_i' \H(\{{x_i^*}'\ganyt\}_{\NT}) v_i}{ 4  \T  \| v_i  \|_2^2} \left( 2 \| v_i  \|_2^2 - \frac{4}{3} \| v_i  \|_2^3\right) \right]   &  \geq    \left(\frac{p_{\min}}{A_p}\right)^{\Tbar} \frac{1}{4  \T } \En \left[ 2 \| v_i  \|_2^2 - \frac{4}{3} \| v_i  \|_2^3 \right] 		\nonumber \\
	& = \left(\frac{p_{\min}}{A_p}\right)^{\Tbar} \frac{1}{\T } \frac{\En [ \| v_i  \|_2^2]}{2} \left( 1 - \frac{2}{3}\frac{\En [ \| v_i  \|_2^3]}{\En [ \| v_i  \|_2^2]} \right).		\label{supp-mlogit Hessian cases}
\end{align}

Recall that $v_i = \{ {x_i^*}' \dtildet\}_{\NT}$. To prove a quadratic lower bound, consider two cases, depending on whether 
\[\frac{1}{2} \left( 1 - \frac{2}{3}\frac{\En [ \| \{ {x_i^*}' \dtildet\}_{\NT}  \|_2^3]}{\En [ \| \{ {x_i^*}' \dtildet\}_{\NT}  \|_2^2]} \right)\]
is above or below $1 / A_K$. 

In the first case, combining Equations \eqref{supp-mlogit Bach} and \eqref{supp-mlogit Hessian cases} gives
\begin{equation}
	\label{supp-mlogit case 1}
	\M(\gtrue + \dtilde)   - \M(\gtrue) -  \sumT \En \left[(\pttrue - \dt) {x_i^*}' \right] \dtildet    \geq   \left(\frac{p_{\min}}{A_p}\right)^{\Tbar} \frac{1}{\T } \frac{\En [ \| \{ {x_i^*}' \dtildet\}_{\NT}  \|_2^2]}{A_K}.
\end{equation}

Now consider the second case, where this bound does not hold. By Assumption \ref{bounded}, the Cauchy-Schwarz inequality, and the conclusion of Lemma \ref{supp-mlogit cone}
\begin{align*}
	\| \{ {x_i^*}' \dtildet \}_{\NT}  \|_1 = \sumT \sumj \left| {x_{i,j}^*} \dtildetj \right|      \leq      \mathcal{X} \left\| \dtilde \right\|_1       &  \leq      \sqrt{\T} \mathcal{X} \vertiii{\dtilde}      		\\
	&  \leq      \sqrt{\T} \mathcal{X}  \left\{ \frac{ 5 \sqrt{|S_*|} }{\kappa_\D}  \vee  \frac{ 10 \bias^d \sqrt{\T}}{\lambda_\D} \right\}\En [ \| \{ {x_i^*}' \dtildet\}_{\NT}  \|_2^2]^{1/2}.
\end{align*}
Hence, by subadditivity (to bound the $\ell_2$ norm by the $\ell_1$ norm), 
\[ \En [ \| \{ {x_i^*}' \dtildet\}_{\NT}  \|_2^3] \leq \En [ \| \{ {x_i^*}' \dtildet\}_{\NT}  \|_2^2 \| \{ {x_i^*}' \dtildet\}_{\NT}  \|_1] \leq \En [ \| \{ {x_i^*}' \dtildet\}_{\NT}  \|_2^2]^{3/2}   \sqrt{\T} \mathcal{X}  \left\{ \frac{ 5 \sqrt{|S_*|} }{\kappa_\D}  \vee  \frac{ 10 \bias^d \sqrt{\T}}{\lambda_\D} \right\}.  \]
Thus
\[\frac{1}{A_K} > \frac{1}{2} \left( 1 - \frac{2}{3}\frac{\En [ \| \{ {x_i^*}' \dtildet\}_{\NT}  \|_2^3]}{\En [ \| \{ {x_i^*}' \dtildet\}_{\NT}  \|_2^2]} \right)  \geq  \frac{1}{2} \left( 1 -    \frac{2}{3} \frac{\mathcal{X} \sqrt{\T} }{\kappa_\D \lambda_\D} \left( 5 \lambda_\D \sqrt{|S_*|}  +   10 \kappa_\D \bias^d \sqrt{\T}  \right)  \En [ \| \{ {x_i^*}' \dtildet\}_{\NT}  \|_2^2]^{1/2} \right),   \]
which is equivalent to
\[\En [ \| \{ {x_i^*}' \dtildet\}_{\NT}  \|_2^2]^{1/2} > \left( 1 - \frac{2}{A_K}\right) \frac{3}{2}  \frac{\kappa_\D \lambda_\D}{\mathcal{X} \sqrt{\T} } \left( 5 \lambda_\D \sqrt{|S_*|}  +   10 \kappa_\D \bias^d \sqrt{\T}  \right)^{-1}     := r_n.\]
Because $\M(\gtrue + \dany) - \M(\gany) - \sumT \En \left[(\pttrue - \dt) {x_i^*}' \right] \danyt$ is convex in $\dany$, and hence any line segment lies above the function, we know that $\En [ \| \{ {x_i^*}' \dtildet\}_{\NT}  \|_2^2]^{1/2} > r_n$, so we have
\[\M(\gtrue + \dtilde) - \M(\gany) - \sumT \En \left[(\pttrue - \dt) {x_i^*}' \right] \dtildet \geq r_n^2 \geq  r_n^2 \frac{\En [ \| \{ {x_i^*}' \dtildet \}_{\NT}  \|_2^2]^{1/2}}{r_n} = r_n \En [ \| \{ {x_i^*}' \dtildet \}_{\NT}  \|_2^2]^{1/2}.\]
Combining this result with Equations \eqref{supp-mlogit start} and \eqref{supp-mlogit upper}, we have
\[\left( 1 - \frac{2}{A_K}\right) \frac{3}{2} \frac{\kappa_\D \lambda_\D}{\mathcal{X} \sqrt{\T} } \left( 5 \lambda_\D \sqrt{|S_*|}  +   10 \kappa_\D \bias^d \sqrt{\T}  \right)^{-1}  \En [ \| \{ {x_i^*}' \danyt\}_{\NT}  \|_2^2]^{1/2}      \leq      \left(    6 \frac{ \lambda_\D \sqrt{|S_*|} }{\kappa_\D}        +     8  \bias^d \sqrt{\T}             \right)   \En [ \| \{ {x_i^*}' \dtildet\}_{\NT}  \|_2^2]^{1/2},\]
which is impossible under the restriction on $A_K$ because it is equivalent to
\begin{align*}
	1 - \frac{2}{A_K} & \leq \frac{2}{3} \frac{\mathcal{X} \sqrt{\T} }{\kappa_\D^2 \lambda_\D} \left( 30 \lambda_\D^2 |S_*| + 100 \lambda_\D \sqrt{|S_*|} \kappa_\D \bias^d \sqrt{\T} + 80 \kappa_\D^2 (\bias^d)^2 \T \right)		\\
	& = \frac{2}{3} \frac{\mathcal{X} \sqrt{\T} }{\kappa_\D^2} \left( 30 \lambda_\D |S_*| + 100 \sqrt{|S_*|} \kappa_\D \bias^d \sqrt{\T} + 80 \kappa_\D^2 (\bias^d)^2 \T \lambda_\D^{-1} \right).
\end{align*}
Solving this for $A_K$ would require that
\[A_K \leq 2 \frac{\kappa_\D^2}{\kappa_\D^2 - (2/3)  \mathcal{X} \sqrt{\T} \left( 30 \lambda_\D |S_*| + 100 \sqrt{|S_*|} \kappa_\D \bias^d \sqrt{\T} + 80 \kappa_\D^2 (\bias^d)^2 \T \lambda_\D^{-1} \right)}  \]
which contradicts the condition in the Theorem.

Therefore, \Eqref{supp-mlogit case 1} must hold.\footnote{This analysis is conceptually similar to using \possessivecite{Belloni-Chernozhukov2011_AoS} restricted nonlinearity impact coefficient, but our characterization is different.} Combining this with Equations \eqref{supp-mlogit start} and \eqref{supp-mlogit upper}, we find that
\[\left(\frac{p_{\min}}{A_p}\right)^{\Tbar} \frac{1}{\T } \frac{\En [ \| \{ {x_i^*}' \dtildet\}_{\NT}  \|_2^2]}{A_K} \leq  \left(    6 \frac{ \lambda_\D \sqrt{|S_*|} }{\kappa_\D}        +     8  \bias^d \sqrt{\T}             \right)   \En [ \| \{ {x_i^*}' \dtildet\}_{\NT}  \|_2^2]^{1/2}.\]
Thus, dividing through and applying the union bound we find that
\begin{equation}
	\label{supp-mlogit log odds rate}
	\max_{\t \in \NT} \En [  ({x_i^*}' \dtildet)^2]^{1/2}      \leq        \En [ \| \{ {x_i^*}' \dtildet\}_{\NT}  \|_2^2]^{1/2}         \leq         \left(\frac{A_p}{p_{\min}}\right)^{\Tbar}     \T A_K \left(    6 \frac{ \lambda_\D \sqrt{|S_*|} }{\kappa_\D}    +   8  \bias^d \sqrt{\T}   \right)   .
\end{equation}

To bound the propensity score error, we apply the mean value theorem and the form of $\partial \ptany / \partial \ganyt$. We must linearize with respect to $\t$ only (recall that $\pttilde$ depends on all of $\gtilde$). To this end, define $M_\t$ as the $\T$-vector with entry $\t$ given by ${x_i^*}'\gtruet + \tilde{m}_\t{x_i^*}'\gtildet$ for a scalar $\tilde{m}_\t \in [0,1]$ and entries $\t' \in \NT \setminus \{\t\}$ equal to ${x_i^*}'\gamma_{\t'}$. Then we have
\begin{equation}
	\label{supp-mlogit rate mvt}
	\left|\pttilde - \pttrue \right|   =   \left| \hat{p}_\t(M_\t) [1 - \hat{p}_\t(M_\t)]{x_i^*}'\dtildet\right|   \leq   \left|{x_i^*}'\dtildet\right|.
\end{equation}
Using this result coupled with the triangle inequality, the bias condition, and \Eqref{supp-mlogit log odds rate}, we find
\begin{align*}
	\En[(\pttilde - p_\t(x_i))^2] ^{1/2} & \leq \En[(\pttilde - \pttrue)^2] ^{1/2} + \En[(\pttrue - p_\t(x_i))^2] ^{1/2} 		\\
	& \leq \En\left[({x_i^*}'\dtildet )^2\right] ^{1/2} + \bias^d		\\
	& \leq \left(\frac{A_p}{p_{\min}}\right)^{\Tbar}     \T A_K \left(    6 \frac{ \lambda_\D \sqrt{|S_*|} }{\kappa_\D}        +     8  \bias^d \sqrt{\T}             \right)   +   \bias^d.
\end{align*}

The $\ell_1$ bound follows from \Eqref{supp-mlogit log odds rate} by the Cauchy-Schwarz inequality and the definition in \Eqref{SE}:
\begin{align*}
	\left\| \gtildet - \gtruet \right\|_1 \leq \sqrt{|\tilde{S}^\D \cup S_\D^*|} \left\| \gtildet - \gtruet \right\|_{2,p} \leq \left(\frac{ |\tilde{S}^\D \cup S_\D^*| }{ \underline{\phi}\{Q, \tilde{S}^\D \cup S_\D^*\} } \right)^{1/2} \En[({x_i^*}'(\gtildet - \gtruet))^2]^{1/2}.	
\end{align*}

Finally, we bound the size of the selected set of coefficients. First, note that optimality of $\gtilde$ ensures that $| \tilde{S}^\D | \leq n$. Then, restating the conclusion Lemma \ref{supp-mlogit S-hat} using the notation of the Theorem and the rate result \eqref{supp-mlogit log odds rate}, then bounding $\overline{\phi}$ by $\dbarphi$ we find that
\[ | \tilde{S}^\D |    \leq    |S_\D^* |4 L_n \dbarphi\{Q, |\tilde{S}^\D|\}.\]
The argument now parallels that used by \citeasnoun{Belloni-Chernozhukov2013_Bern}, relying on their result on the sublinearity of sparse eigenvalues. Let $\lceil m \rceil$ be the ceiling function and note that $\lceil{m}\rceil \leq 2m$. For any $m \in \N_Q^\D$, suppose that $|\tilde{S}^\D| > m$. Then, 
\begin{align*}
	| \tilde{S}^\D |& \leq |S_\D^* |4 L_n \dbarphi\{Q, m (|\tilde{S}^\D| / m) \} 		\\
	& \leq  \left\lceil |\tilde{S}^\D| / m \right \rceil  |S_\D^* |4 L_n   \dbarphi\{Q, m  \}		\\
	& \leq  ( |\tilde{S}^\D| / m) |S_\D^* | 8 L_n   \dbarphi\{Q, m\}.
\end{align*}
Rearranging gives 
\[m \leq |S_\D^* | 8 L_n   \dbarphi\{Q, m\}\]
whence $m \not\in \N_Q^\D$. Minimizing over $\N_Q^\D$ gives the result.	 \qed

%%%%%%%%%%%%%%%%%%%%%%%%%%%%
\subsection{Proof of Theorem \ref{thm-post mlogit}}

Define $\dhat = \ghat - \gtrue$. Many of the arguments parallel those for Theorem \ref{thm-mlogit}. The key differences are that a quadratic lower bound for $\M(\gtrue + \dhat) - \M(\gtrue) - \sumT \En \left[(\pttrue - \dt) {x_i^*}' \right] \dhatt$ may occur, but is not necessary, and $\dhat$ may not belong to the cone of the restricted eigenvalues, but obeys the sparse eigenvalue constraints. 

We first give a suitable upper bound for $\M(\gtrue + \dhat) - \M(\gtrue) - \sumT \En \left[(\pttrue - \dt) {x_i^*}' \right] \dhatt$. By the Cauchy-Schwarz inequality and the definition of the sparse eigenvalues of \Eqref{SE},
\begin{align}
	\vertiii{ \dhat } & = \sum_{j \in \hat{S}_\D \cup S_\D^*} \left\| \dhatj \right\|_2		\nonumber \\
	& \leq  \sqrt{ \left| \hat{S}_\D \cup S_\D^* \right|} \sqrt{ \sumT \sum_{j \in \hat{S}_\D \cup S_\D^*}  \dhattj^2 }		\nonumber \\
	& =  \sqrt{ \left| \hat{S}_\D \cup S_\D^* \right|} \sqrt{ \sumT \left\| \dhatj \right\|_2^2 }		\nonumber  \\
	& \leq  \sqrt{ \left| \hat{S}_\D \cup S_\D^* \right|} \sqrt{ \sumT \underline{\phi}\left\{Q, \hat{S}_\D \cup S_\D^*  \right\}^{-2} \dhatt' Q \dhatt }		\nonumber  \\
	& =  \sqrt{ \left| \hat{S}_\D \cup S_\D^* \right|}  \underline{\phi}\left\{Q, \hat{S}_\D \cup S_\D^*  \right\}^{-1} \En [ \| \{ {x_i^*}' \dhatt\}_{\NT}  \|_2^2]^{1/2}.				\label{supp-post mlogit norm}
\end{align}
Following identical steps to Equations \eqref{supp-mlogit score variance}, \eqref{supp-mlogit score bias}, and \eqref{supp-mlogit score}, but with $\dhat$ in place of $\dtilde$, and then using the above bound, we have
\begin{align}
	\left| \sumT \En \left[(\pttrue - \dt) {x_i^*}' \right] \dhatt \right|  & \leq   \frac{\lambda_\D}{2} \vertiii{ \dhat } + \bias^d \sqrt{\T} \En [ \| \{ {x_i^*}' \dhatt\}_{\NT}  \|_2^2]^{1/2} 		 \nonumber \\
	& \leq    \left(  \frac{\lambda_\D}{2}  \frac{ \sqrt{ | \hat{S}_\D \cup S_\D^* |} }{  \underline{\phi}\{Q, \hat{S}_\D \cup S_\D^*  \} }    + \bias^d \sqrt{\T}    \right)   \En [ \| \{ {x_i^*}' \dhatt\}_{\NT}  \|_2^2]^{1/2}.		\label{supp-post mlogit score}
\end{align}

Next we turn to $\M(\gtrue + \dhat)  - \M(\gtrue)$. By optimality of the post selection estimator $\M(\ghat) \leq \M(\gtilde)$, as $\tilde{S}^\D \subset \hat{S}_\D$ by construction, and hence $\M(\gtrue + \dhat)  - \M(\gtrue) \leq \M(\gtilde)  - \M(\gtrue)$. By the mean value theorem, for scalars $\{m_\t \in [0,1]\}_{\NT}$ we have
\begin{align}
	\M(\gtrue + \dtilde)  - \M(\gtrue)  &  =  \sumT \En \left[(\dt - \hat{p}_\t(\{{x_i^*}'\gtruet + m_\t{x_i^*}'\dtildet\}) ) {x_i^*}'\dtildet \right]		\nonumber \\
	& = \sumT \En \left[(\dt  - \pttrue) {x_i^*}'\dtildet \right]  		\nonumber \\
	& \quad  +   \sumT \En \left[(\pttrue - \hat{p}_\t(\{{x_i^*}'\gtruet + m_\t{x_i^*}'\dtildet\}) ) {x_i^*}'\dtildet \right],		\nonumber \\
	& \leq   \frac{\lambda_\D}{2} \vertiii{ \dtilde } + \bias^d \sqrt{\T} \En [ \| \{ {x_i^*}' \dtildet\}_{\NT}  \|_2^2]^{1/2}    + \sumT \En \left[m_\t ({x_i^*}'\dtildet)^2 \right].		\nonumber \\
	& \leq \left( \frac{\lambda_\D}{2}  \frac{\sqrt{ | \hat{S}_\D \cup S_\D^* |}} {  \underline{\phi}\{Q, \hat{S}_\D \cup S_\D^*  \}}     + \bias^d \sqrt{\T} \right)   \En [ \| \{ {x_i^*}' \dtildet\}_{\NT}  \|_2^2]^{1/2}    +  \En [ \| \{ {x_i^*}' \dtildet\}_{\NT}  \|_2^2],		\label{supp-mlogit loss fcn}		
\end{align}
where the first inequality follows from \Eqref{supp-mlogit score} and the same steps as in \eqref{supp-mlogit rate mvt} while the second applies \eqref{supp-post mlogit norm} with $\dtilde$ and $m_\t \leq 1$.\footnote{Applying the steps of \Eqref{supp-post mlogit norm} to $\dtilde$ is preferred to using the results of Lemma \ref{supp-mlogit cone} because it leads to the tidier expression involving $\underline{\phi}\{Q, \hat{S}_\D \cup S_\D^*  \}$, but the latter method could be substituted.}

Collecting the bounds of \eqref{supp-post mlogit score} and \eqref{supp-mlogit loss fcn}, and the definition of $R_\M$ (that is, \Eqref{supp-mlogit log odds rate}) gives
\begin{multline}
	\label{supp-post mlogit upper}
	\M(\gtrue + \dhat) - \M(\gtrue) - \sumT \En \left[(\pttrue - \dt) {x_i^*}' \right] \dhatt 	 		\\		
	\leq   \left( \frac{\lambda_\D}{2}  \frac{\sqrt{ | \hat{S}_\D \cup S_\D^* |}} {  \underline{\phi}\{Q, \hat{S}_\D \cup S_\D^*  \}}     + \bias^d \sqrt{\T} \right)  \left( \En [ \| \{ {x_i^*}' \dhatt\}_{\NT}  \|_2^2]^{1/2} +  R_\M  \right) +  R_\M^2 .
\end{multline}

Next, we turn to a lower bound. Consider the same two cases as in the proof of Theorem \ref{thm-mlogit}. In the first case, we have the quadratic lower bound:
\begin{equation}
	\label{supp-post mlogit case 1}
	M(\gtrue + \dhat)   - \M(\gtrue) -  \sumT \En \left[(\pttrue - \dt) {x_i^*}' \right] \dhatt    \geq   \left(\frac{p_{\min}}{A_p}\right)^{\Tbar} \frac{1}{\T } \frac{\En [ \| \{ {x_i^*}' \dhatt\}_{\NT}  \|_2^2]}{A_K}.
\end{equation}
In the other case, this bound may not hold. Arguing as in the proof of Theorem \ref{thm-mlogit}, but applying \Eqref{supp-post mlogit norm}, we get
\[\| \{ {x_i^*}' \dhatt \}_{\NT}  \|_1 \leq \sqrt{\T} \mathcal{X} \sqrt{ | \hat{S}_\D \cup S_\D^* |}  \underline{\phi}\{Q, \hat{S}_\D \cup S_\D^*  \}^{-1} \En [ \| \{ {x_i^*}' \dtildet\}_{\NT}  \|_2^2]^{1/2}.\]
Therefore, as above, we find 
\begin{equation}
	\label{supp-post mlogit case 2}
	\M(\gtrue + \dhat) - \M(\gany) - \sumT \En \left[(\pttrue - \dt) {x_i^*}' \right] \dhatt \geq r_n \En [ \| \{ {x_i^*}' \dhatt \}_{\NT}  \|_2^2]^{1/2},
\end{equation}
with
\[r_n  =  \frac{3}{2} \left( 1 - \frac{2}{A_K}\right) \frac{ \underline{\phi}\{Q, \hat{S}_\D \cup S_\D^*  \} }{\mathcal{X} \sqrt{\T} \sqrt{ | \hat{S}_\D \cup S_\D^* |}}.\]
Collecting the upper bound of \eqref{supp-post mlogit upper} and the lower bounds \eqref{supp-post mlogit case 1} and \eqref{supp-post mlogit case 2} we have
\begin{multline}
	\label{supp-post mlogit cases}
	\left\{ \left(\frac{p_{\min}}{A_p}\right)^{\Tbar} \frac{1}{\T } \frac{\En [ \| \{ {x_i^*}' \dhatt\}_{\NT}  \|_2^2] }{A_K} \right\}    \wedge   \left\{  r_n \En [ \| \{ {x_i^*}' \dhatt \}_{\NT}  \|_2^2]^{1/2} \right\}     		\\
	\leq      \left( \frac{\lambda_\D}{2}  \frac{\sqrt{ | \hat{S}_\D \cup S_\D^* |}} {  \underline{\phi}\{Q, \hat{S}_\D \cup S_\D^*  \}}     + \bias^d \sqrt{\T} \right)  \left( \En [ \| \{ {x_i^*}' \dhatt\}_{\NT}  \|_2^2]^{1/2} +  R_\M  \right) +  R_\M^2.	
\end{multline}

For some $A_1 > 1$, replace the restriction on $A_K$ in the Theorem with the requirement that
\begin{multline*}
	A_K > 2 \left\{ \frac{\underline{\phi}\{Q, \hat{S}_\D \cup S_\D^*  \}^2}{ \underline{\phi}\{Q, \hat{S}_\D \cup S_\D^*  \}^2  -  (A_1/3)\mathcal{X} \sqrt{\T} \left( \lambda_\D  | \hat{S}_\D \cup S_\D^* | + \bias^d \underline{\phi}\{Q, \hat{S}_\D \cup S_\D^*  \} \sqrt{ | \hat{S}_\D \cup S_\D^* |}  \sqrt{\T} \right)  } \right\} 		\\		\vee \left\{  \frac{\underline{\phi}\{Q, \hat{S}_\D \cup S_\D^*  \}}{ \underline{\phi}\{Q, \hat{S}_\D \cup S_\D^*  \}  -  (A_1/3) 2 R_\M  \mathcal{X} \sqrt{\T}  \sqrt{ | \hat{S}_\D \cup S_\D^* |}} \right\}.
\end{multline*}
Suppose the linear term is the minimum. The first restriction on $A_K$ implies, by simple algebraic manipulations, that
	\[ \left( \frac{\lambda_\D}{2}  \frac{\sqrt{ | \hat{S}_\D \cup S_\D^* |}} {  \underline{\phi}\{Q, \hat{S}_\D \cup S_\D^*  \}}     + \bias^d \sqrt{\T} \right) < \frac{r_n}{A_1},\]
while the second gives $R_\M < (r_n / A_1)$. Plugging the former into \Eqref{supp-post mlogit cases} and then applying the latter yields
\begin{align*}
	r_n \En [ \| \{ {x_i^*}' \dhatt \}_{\NT}  \|_2^2]^{1/2}   &   \leq (r_n / A_1) \left(  \En [ \| \{ {x_i^*}' \dhatt \}_{\NT}  \|_2^2]^{1/2} +  R_\M \right) + R_\M^2  		\\
	& \leq  (r_n / A_1) \left(  \En [ \| \{ {x_i^*}' \dhatt \}_{\NT}  \|_2^2]^{1/2} +  2R_\M \right).
\end{align*}
Canceling the $r_n$ and solving yields
\[\En [ \| \{ {x_i^*}' \dhatt \}_{\NT}  \|_2^2]^{1/2}   \leq   \frac{2 R_\M }{ A_1 - 1}.\]
On the other hand, if the quadratic term is the minimum, define
\[R_\M' = \left(\frac{A_p}{p_{\min}}\right)^{\Tbar}\T A_K \left( \frac{\lambda_\D}{2}  \frac{\sqrt{ | \hat{S}_\D \cup S_\D^* |}} {  \underline{\phi}\{Q, \hat{S}_\D \cup S_\D^*  \}}     + \bias^d \sqrt{\T} \right).\]
With this notation and the quadratic term being the minimum, \Eqref{supp-post mlogit cases} becomes
\[ \En [ \| \{ {x_i^*}' \dhatt\}_{\NT}  \|_2^2] \leq R_\M'  \En [ \| \{ {x_i^*}' \dhatt\}_{\NT}  \|_2^2]^1/2 + R_\M' R_\M + \left(\frac{A_p}{p_{\min}}\right)^{\Tbar} \T A_K R_\M^2.\]
Then, because $a^2 \leq ab + c$ implies that $a \leq b + \sqrt{c}$, we have
	\[\En [ \| \{ {x_i^*}' \dhatt \}_{\NT}  \|_2^2]^{1/2}  \leq R_\M' + \left(  R_\M' R_\M + \left(\frac{A_p}{p_{\min}}\right)^{\Tbar} \T A_K R_\M^2 \right)^{1/2}.\]
Taking $A_1 = 3$ and combining the bounds on $\En [ \| \{ {x_i^*}' \dhatt \}_{\NT}  \|_2^2]^{1/2}$ from the two cases gives
\begin{equation*}
	\label{supp-post log odds rate}
	\En [ \| \{ {x_i^*}' \dhatt \}_{\NT}  \|_2^2]^{1/2}  \leq   \left\{ R_\M \right\}  \vee  \left\{  R_\M' + \left(  R_\M' R_\M + \left(\frac{A_p}{p_{\min}}\right)^{\Tbar} \T A_K R_\M^2 \right)^{1/2} \right\}.
\end{equation*}

From this bound on the log-odds estimates, we obtain the bound on the propensity score estimates and the $\ell_1$ rate, given by,
	\[\max_{\t \in \NT} \En[(\pthat - p_\t(x_i))^2] ^{1/2} \leq   \left\{R_\M \right\}  \vee  \left\{  R_\M' + \left(  R_\M' R_\M + \left(\frac{A_p}{p_{\min}}\right)^{\Tbar} \T A_K R_\M^2 \right)^{1/2} \right\}   + \bias^d,\]
and
	\[\max_{\t \in \NT} \left\| \ghatt - \gtruet \right\|_1  \leq \left(\frac{ |\tilde{S}^\D \cup S_\D^*| }{ \underline{\phi}\{Q, \tilde{S}^\D \cup S_\D^*\} } \right)^{1/2} \left\{ R_\M \right\}  \vee  \left\{  R_\M' + \left(  R_\M' R_\M + \left(\frac{A_p}{p_{\min}}\right)^{\Tbar} \T A_K R_\M^2 \right)^{1/2} \right\},\]
by arguments parallel to those used in the proof of Theorem \ref{thm-mlogit}.

%%%%%%%%%%%%%%%%%%%%%%%%%%%%
%%%%%%%%%%%%%%%%%%%%%%%%%%%%
\section{Proofs for Group Lasso Selection and Estimation of Linear Models}

Unless otherwise noted, all bounds in this section are nonasymptotic. We will use generic notation $X^*$, $\delta$, $s$, etc, as this section deals only with linear models.

%%%%%%%%%%%%%%%%%%%%%%%%%%%%
\subsection{Lemmas}

\begin{lemma}[Score Bound]
	\label{supp-ols score bound}
	For $\lambda_Y$ and $\mathcal{P}$ defined respectively in \Eqref{lambda} and \Eqref{probability} we have
	\[\P \left[ \maxj \| \E_{n,\newdot} [u_i x_{i,j}^* ] \|_2 \geq \frac{\lambda_Y}{4} \right] \leq \mathcal{P}. \]
\end{lemma}
\begin{proof}
The residuals $u_i$ are conditionally mean-zero by definition. Using this, Assumption \ref{iid}, the definitions of $\mathcal{X}$ and $\mathcal{U}$, and the Cauchy-Schwarz inequality, we find that
	\[\E\left[ \| \E_{n,\newdot} [u_i x_{i,j}^* ] \|_2^2 \right]  =  \sumTbar \E\left[ \Ent[u_i x_{i,j}^* ]^2 \right] = \sumTbar \frac{1}{n_\t} \E [u_i^2 (x_{i,j}^*)^2]  \leq \sumTbar \frac{1}{n_\t}  \E[ | X_{i,j}^*|^4]^{1/2} \E[ | U_i|^4]^{1/2} \leq  \frac{\mathcal{X}^2 \mathcal{U}^2 \Tbar}{\underline{n}}  \]
uniformly in $j \in \Np$. Define the mean-zero random variables $\xi_{t,j}$ as: 
	\[\xi_{t,j} = (\Ent[u_i x_{i,j}^* ])^2 - \frac{1}{n_\t} \E[U^2 {X_j^*}^2].\]
Let $r_n = \Tbar^{-1/2} \log(p \vee \underline{n})^{3/2 + \delta}$. Then, using the definition of $\lambda_Y$:
\begin{align}
	\P \left[ \maxj \| \E_{n,\newdot} [u_i {x_{i,j}^*}] \|_2 \geq \frac{\lambda_Y}{4} \right] & = \P \left[ \maxj \| \E_{n,\newdot} [u_i {x_{i,j}^*}] \|_2^2 \geq \frac{\lambda_Y^2 }{16} \right]	\nonumber  \\
	& = \P \left[ \maxj \| \E_{n,\newdot} [u_i {x_{i,j}^*}] \|_2^2 \geq \frac{ \mathcal{X}^2 \mathcal{U}^2 \Tbar}{\underline{n}}  + \frac{ \mathcal{X}^2 \mathcal{U}^2 \Tbar r_n }{\underline{n}} \right]		\nonumber  \\
	& = \P \left[ \maxj \| \E_{n,\newdot} [u_i {x_{i,j}^*}] \|_2^2  -  \frac{ \mathcal{X}^2 \mathcal{U}^2 \Tbar }{\underline{n}}     \geq  \frac{\mathcal{X}^2 \mathcal{U}^2 \Tbar r_n }{\underline{n}} \right]	\nonumber  \\
	& \leq \P \left[ \maxj \sumTbar \xi_{t,j}  \geq  \frac{\mathcal{X}^2 \mathcal{U}^2 \Tbar r_n }{\underline{n}} \right]	\nonumber  \\
	& \leq \E \left[ \maxj \left| \sumTbar \xi_{t,j} \right|\right]   \frac{\underline{n}}{ \mathcal{X}^2 \mathcal{U}^2 \Tbar r_n },			\label{supp-maximal1}
\end{align}
where final line follows from Markov's inequality.

Next, applying Lemma 9.1 of \citeasnoun{Lounici-etal2011_AoS} (with their $m=1$ and hence $c(m) = 2$) followed by Jensen's inequality and Assumption \ref{fourth moments}, we find that
\begin{align}
	\E \left[ \maxj \left| \sumTbar \xi_{t,j} \right|\right] & \leq ( 8 \log(2p))^{1/2} \E \left[ \left( \sumTbar \maxj \xi_{t,j}^2 \right)^{1/2} \right]	\nonumber  \\
	& \leq  ( 8 \log(2p))^{1/2} \left( \E \left[  \sumTbar \maxj \xi_{t,j}^2 \right] \right)^{1/2}  	\nonumber  \\
	& \leq  4 \log(2p)^{1/2} \left(  \sumTbar  \frac{\mathcal{X}^4 \mathcal{U}^4}{\underline{n}^2} + \sumTbar \E \left[ \maxj \left| \Ent[u _i {x_{i,j}^*}]  \right| ^4  \right] \right)^{1/2}.	\label{supp-maximal2}
\end{align}
The leading 4 is $\sqrt{8} \sqrt{2}$, where $\sqrt{2}$ is a byproduct of applying the inequality $(a - b)^2 \leq 2(a^2 + b^2)$ to $\xi_{t,j}^2$. Again using Lemma 9.1 of \citeasnoun{Lounici-etal2011_AoS} (with their $m=4$, and $c(m)=12$ since $c(4) \geq (e^{4-1} - 1 )/2 + 2 \approx 11.54$), we bound the expectation in the second term above as follows:
\begin{align}
	 & \E \left[ \maxj  \left|  \Ent[u _i {x_{i,j}^*}]  \right| ^4  \right]  \leq [ 8 \log(12p) ] ^{4/2} \E \left[ \left( \sumt \maxj \left| \frac{u_i {x_{i,j}^*} }{n_\t}\right|^2  \right)^{4/2} \right] 		\nonumber  \\
	& \qquad \leq \frac{64 \log(12p)^2 \mathcal{X}^4 }{n_\t^4} \left( \sumt \E\left[ |u_i|^4  \right] + \sumt \sum_{ k \in \It \setminus \{i\}} \E [  |u_i|^2 |u_k|^2] \right)			\nonumber  \\
	& \qquad \leq \frac{64 \log(12p)^2 \mathcal{X}^4 }{n_\t^4} \left( \sumt \E\left[ |u_i|^4  \right] + \sumt \sum_{ k \in \It \setminus \{i\}} \E [  |u_i|^4 ]^{1/2} \E[ |u_k|^4]^{1/2} \right)			\nonumber  \\
	& \qquad \leq \frac{64 \log(12p)^2 \mathcal{X}^4 \mathcal{U}^4}{n_\t^4} \left( n_\t + n_\t(n_\t - 1 )  \right)			\nonumber  \\
	& \qquad = \frac{64 \log(12p)^2 \mathcal{X}^4 \mathcal{U}^4}{n_\t^2},	\label{supp-maximal3}
\end{align}
where the second inequality uses H\"older's inequality and the final inequality applies Assumptions \ref{iid} and \ref{fourth moments}.

Now, inserting the results of Eqns.\ \eqref{supp-maximal2} and \eqref{supp-maximal3} into \Eqref{supp-maximal1}, we have
\begin{align*}
	\P \left[ \maxj \| \E_{n,\newdot} [u_i {x_{i,j}^*}] \|_2 \geq \frac{\lambda_Y }{4} \right] & \leq \frac{ 4 \underline{n}  \log(2p)^{1/2} }{\Tbar \mathcal{X}^2 \mathcal{U}^2 r_n }  \left(  \sumTbar  \frac{\mathcal{X}^4 \mathcal{U}^4}{\underline{n}^2} + \sumTbar \frac{64 \log(12p)^2 \mathcal{X}^4 \mathcal{U}^4}{n_\t^2} \right)^{1/2} \\
& \leq \frac{ 4\log(2p)^{1/2}  }{ r_n \sqrt{\Tbar}  } [ 1 + 64 \log(12 p)^2]^{1/2} = \mathcal{P},
\end{align*}
using the choice $r_n = \Tbar^{-1/2} \log(p \vee \underline{n})^{3/2 + \delta}$.
\end{proof}

\begin{lemma}[Estimate Sparsity]
	\label{supp-ols S-hat}
	With probability $ 1 - \mathcal{P}$, as defined in \Eqref{probability}, the model selected by solving \eqref{grplasso} obeys
	\[ | \tilde{S}^Y |  \leq \frac{16}{\lambda_Y^2} \sumTbar \overline{\phi}\{Q_\t, \tilde{S}^Y \}  \Ent [ (  \mut(x_i)  - {x_i^*}'\btildet)^2  ]. \]
\end{lemma}
\begin{proof}
First, by Karush-Kuhn-Tucker conditions for \eqref{grplasso}, for all $\t \in \NTbar$, if $\btildej \neq 0$ it must satisfy
\begin{equation}
	\label{supp-ols KKT}
		2\Ent[{x_{i,j}^*} (y_i - {x_i^*}'\btildet)] = \lambda_Y \frac{\btildetj}{\|\btildej \|_2}.
\end{equation}
Hence, taking the $\ell_2$-norm over $\t \in \NTbar$ for fixed $j \in \tilde{S}^Y$, using $y_i = \mut(x_i) + u_i$, the triangle inequality, and Lemma \ref{supp-ols score bound}:
\begin{align*}
	\lambda_Y & = 2 \left\| \Ent[{x_{i,j}^*} (y_i - {x_i^*}'\btildet)] \right\|_2		\\
	& \leq  2 \left\| \Ent[{x_{i,j}^*} \{ \mut(x_i) - {x_i^*}'\btildet\}] \right\|_2  +  2 \left\| \Ent[{x_{i,j}^*} u_i ] \right\|_2		\\
	& \leq  2 \left\| \Ent[{x_{i,j}^*} \{ \mut(x_i) - {x_i^*}'\btildet\} ] \right\|_2  + \lambda_Y/2.
\end{align*}
Let $\bm{G}_\t$ be the vector of $\{\mut(x_i)\}_{i \in \It}$ and $\tilde{\bm{G}}_\t$ that of $\{{x_i^*}'\btildet\}_{i \in \It}$. Collecting terms, then squaring both sides and taking $\sum_{j \in \tilde{S}^Y}$ (i.e. applying $\| \cdot \|_2^2$ over $j \in \tilde{S}^Y$ to both sides) yields
\begin{align*}
	\sum_{j \in \tilde{S}^Y} \lambda_Y^2 & \leq 16 \sum_{j \in \tilde{S}^Y} \sumTbar  \left( \Ent[{x_{i,j}^*} \{ \mut(x_i) - {x_i^*}'\btildet\} ] \right)^2		\\
	& =  16 \sumTbar  \frac{1}{n_\t^2}\left\|   \left[\Xnt' (\bm{G}_\t - \tilde{\bm{G}}_\t) \right]_{j \in \tilde{S}^Y} \right\|_2^2	\\
	& \leq  16 \sumTbar \frac{ \overline{\phi}\{Q_\t, \tilde{S}^Y\} }{n_\t}   \left\|   \bm{G}_\t - \tilde{\bm{G}}_\t  \right\|_2^2	\\
	& \leq  16 \sumTbar \overline{\phi}\{Q_\t, \tilde{S}^Y \}  \Ent [ (  \mut(x_i)  - {x_i^*}'\btildet)^2  ].
\end{align*}
The claim follows, as the left-hand side is equal to  $|\tilde{S}^Y| \lambda_Y^2$.
\end{proof}

%%%%%%%%%%%%%%%%%%%%%%%%%%%%
\subsection{Proof of Theorem \ref{thm-ols}}

Let $\dtilde  = \btilde - \btrue$. First, because $\beta^*_{\newdot, S_*^c} = 0$
 \[\vertiii{\beta^*_{\newdot, S_*^c}} - \vertiii{\beta^*_{\newdot, S_*^c} + \tilde{\delta}_{\newdot, S_*^c}} =\vertiii{ \tilde{\delta}_{\newdot, S_*^c}}.\]
Therefore:
\begin{align}
	\vertiii{\btrue} - \vertiii{\btrue + \dtilde} & = \vertiii{\beta^*_{\newdot, S_*}} - \vertiii{\beta^*_{\newdot, S_*} + \tilde{\delta}_{\newdot, S_*}} - \vertiii{ \tilde{\delta}_{\newdot, S_*^c}}			\nonumber \\
	& \leq \vertiii{\beta^*_{\newdot, S_*}} - \vertiii{\beta^*_{\newdot, S_*} + \tilde{\delta}_{\newdot, S_*}}		\nonumber \\
	& \leq \left| \vertiii{\beta^*_{\newdot, S_*}} - \vertiii{\beta^*_{\newdot, S_*} + \tilde{\delta}_{\newdot, S_*}}\right|		\nonumber \\
	& \leq \vertiii{\beta^*_{\newdot, S_*} - \left(\beta^*_{\newdot, S_*} + \tilde{\delta}_{\newdot, S_*}\right) }		 = \vertiii{ \tilde{\delta}_{\newdot, S_*}}, 		\label{supp-ols norm}
\end{align}
where the first inequality reflects dropping the nonpositive final term (the norm is nonnegative) and the third inequality follows from the triangle inequality. Because $\btilde$ solves \eqref{grplasso}
\[ \SSE(\dtilde) + \lambda_Y \vertiii{\btilde} \leq   \SSE(\dtilde)  + \lambda_Y \vertiii{\btrue}.\]
Define the $i^{\text{th}}$ realization of $B_\t^Y$ as $\by = \mut(x_i) - {x_i^*}'\btruet$. Inserting $y_i = {x_i^*}'\btruet + \by + u_i$ on each side, we obtain
\begin{multline*}
	\sumTbar \Ent [({x_i^*}'\dtildet)^2] - 2 \sumTbar \Ent[(\by  +  u_i) {x_i^*}'\dtildet ]  + \sumTbar \Ent[(\by  +  u_i)^2]  + \lambda_Y \vertiii{\btrue + \dtilde} \\
		\leq  \sumTbar \Ent[(\by  +  u_i)^2]  + \lambda_Y \vertiii{\btrue}
\end{multline*}
Canceling common factors and rearranging, we find that
\begin{equation}
	\label{supp-optimality}
	\sumTbar \Ent [({x_i^*}'\dtildet)^2] \leq 2 \sumTbar \Ent[(\by  +  u_i) {x_i^*}'\dtildet ]   +   \lambda_Y \left\{ \vertiii{\btrue}  -  \vertiii{\btrue + \dtilde} \right\}
\end{equation}

By the Cauchy-Schwarz inequality and Lemma \ref{supp-ols score bound}, with probability at least $1 - \mathcal{P}$
\begin{align*}
	2 \sumTbar \Ent[u_i {x_i^*}'\dtildet] & = 2 \sumj  \sumTbar \Ent[u_i {x_{i,j}^*}\dtildet] 		\\
	& \leq \sumj 2 \sqrt{ \sumTbar ( \Ent[u_i {x_{i,j}^*} ])^2 }  \sqrt{ \sumTbar  \dtildetj^2 } 		\\
	& = \sumj 2 \| \E_{n,\newdot} [u_i {x_{i,j}^*}] \|_2  \| \dtildej  \|_2		\\
	& \leq \frac{\lambda_Y}{2} \vertiii{\dtilde}.
\end{align*}
Next, using the Cauchy-Schwarz inequality, the bias condition, and Jensen's inequality:
\begin{align*}
	2 \sumTbar \Ent[\by {x_i^*}'\dtildet] & \leq 2 \sumTbar \Ent[(\by)^2]^{1/2}  \Ent[({x_i^*}'\dtildet)^2]^{1/2} 		\\
	 & \leq 2 \bias^y \sumTbar  \Ent[({x_i^*}'\dtildet)^2]^{1/2} 		\\
	 & \leq 2 \bias^y \sqrt{ \sumTbar  \Ent[({x_i^*}'\dtildet)^2]   }.
\end{align*}

Plugging the previous two inequalities into \Eqref{supp-optimality} we find that with probability at least $1 - \mathcal{P}$:
\begin{equation}
	\label{supp-ols two cases}
	\sumTbar \Ent [({x_i^*}'\dtildet)^2]   \leq   \frac{\lambda_Y}{2} \vertiii{\dtilde}  +   \lambda_Y \left\{ \vertiii{\btrue}  -  \vertiii{\btrue + \dtilde} \right\}   +  2 \bias^y \sqrt{ \sumTbar  \Ent[({x_i^*}'\dtildet)^2]   }.
\end{equation}

Consider two cases, depending on whether
\[\sumTbar \Ent [({x_i^*}'\dtildet)^2]  - 2 \bias^y \sqrt{ \sumTbar  \Ent[({x_i^*}'\dtildet)^2]   } \]
is negative or nonnegative. In the first case, rearranging the display above  gives
\begin{equation}
	\label{supp-ols case1}
	\sumTbar \Ent [({x_i^*}'\dtildet)^2]  < 2 \bias^y \sqrt{ \sumTbar  \Ent[({x_i^*}'\dtildet)^2]   } 
\end{equation}

For the second case, returning to \Eqref{supp-ols two cases}, rearranging, and discarding positive terms (under the second case) from the left side, we have
	\[ 0  \leq   \frac{\lambda_Y}{2} \vertiii{\dtilde}  +   \lambda_Y \left\{ \vertiii{\btrue}  -  \vertiii{\btrue + \dtilde} \right\}    .\]
Canceling $\lambda_Y$, decomposing the supports
	\[ 0  \leq   \frac{1}{2} \vertiii{\tilde{\delta}_{\newdot,S_*} } +   \frac{1}{2} \vertiii{\tilde{\delta}_{\newdot,S_*^c } }  +  \vertiii{\beta^*_{\newdot,S_*} }  -  \vertiii{\beta^*_{\newdot,S_*} + \tilde{\delta}_{\newdot,S_*}}     -  \vertiii{\tilde{\delta}_{\newdot,S_*^c}},\]
Collecting terms and applying the final inequality of \Eqref{supp-ols norm} yields
	\[ \frac{1}{2} \vertiii{\tilde{\delta}_{\newdot,S_*^c } }   \leq   \frac{1}{2} \vertiii{\tilde{\delta}_{\newdot,S_*} }   +  \vertiii{\beta^*_{\newdot,S_*} }  -  \vertiii{\beta^*_{\newdot,S_*} + \tilde{\delta}_{\newdot,S_*}}     \leq \frac{3}{2}  \vertiii{\tilde{\delta}_{\newdot,S_*} }, \]
and hence $\dtilde$ obeys the cone constraint of \Eqref{RE}.

Thus, beginning with \Eqref{supp-ols two cases}, decomposing the support of $\dtilde$, using the cone constraint and the result of  \Eqref{supp-ols norm}, the Cauchy-Schwarz inequality, the definition of $\kappa_Y$ from \Eqref{RE}, 
\begin{align}
	\sumTbar \Ent [({x_i^*}'\dtildet)^2]  &  \leq   \frac{\lambda_Y}{2} \vertiii{\tilde{\delta}_{\newdot,S_*} }   +  \frac{\lambda_Y}{2} \vertiii{\tilde{\delta}_{\newdot,S_*^c} }   +   \lambda_Y \left\{ \vertiii{\btrue}  -  \vertiii{\btrue + \dtilde} \right\}   +  2 \bias^y \sqrt{ \sumTbar  \Ent[({x_i^*}'\dtildet)^2]   }		\nonumber \\
	&  \leq   \frac{\lambda_Y}{2} \vertiii{\tilde{\delta}_{\newdot,S_*} }   +  \frac{\lambda_Y}{2} 3 \vertiii{\tilde{\delta}_{\newdot,S_*} }   +   \lambda_Y \vertiii{ \tilde{\delta}_{\newdot, S_*}}  +  2 \bias^y \sqrt{ \sumTbar  \Ent[({x_i^*}'\dtildet)^2]   }		\nonumber \\
	&  \leq  3 \lambda_Y \vertiii{ \tilde{\delta}_{\newdot, S_*}}  +  2 \bias^y \sqrt{ \sumTbar  \Ent[({x_i^*}'\dtildet)^2]   }		\nonumber \\
	&  \leq  3 \lambda_Y \sqrt{|S_*|} \left\| \tilde{\delta}_{\newdot, S_*} \right\|_2  +  2 \bias^y \sqrt{ \sumTbar  \Ent[({x_i^*}'\dtildet)^2]   }		\nonumber \\
	&  \leq  \frac{3 \lambda_Y \sqrt{|S_*|}}{\kappa_Y} \sqrt{ \sumTbar \dtildet' Q_t \dtildet }   +  2 \bias^y \sqrt{ \sumTbar  \Ent[({x_i^*}'\dtildet)^2]   }		\nonumber \\
	& = \left( \frac{3 \lambda_Y \sqrt{|S_*|}}{\kappa_Y} +  2 \bias^y \right) \sqrt{ \sumTbar  \Ent[({x_i^*}'\dtildet)^2]   }. 		\label{supp-ols case2}
\end{align}

Equations \eqref{supp-ols case1} and \eqref{supp-ols case2} show that in both cases defined above, the root left side appears on the right. Thus, dividing through in both we find that
\[\sumTbar \Ent [({x_i^*}'\dtildet)^2] \leq \left( \frac{3 \lambda_Y \sqrt{|S_*|}}{\kappa_Y} +  2 \bias^y \right)^2,\]
because the bound given in \Eqref{supp-ols case2} contains that of \eqref{supp-ols case1}. From the union bound we have
\[\maxTbar \Ent [({x_i^*}'\dtildet)^2]^{1/2}  \leq \left( \frac{3 \lambda_Y \sqrt{|S_*|}}{\kappa_Y} +  2 \bias^y \right),\]
and therefore, by the triangle inequality
\begin{equation}
	\label{supp-ols final rate}
	\maxTbar \Ent [({x_i^*}'\btildet - \mut(x_i))^2]^{1/2}  \leq \left( \frac{3 \lambda_Y \sqrt{|S_*|}}{\kappa_Y} +  2 \bias^y \right) + \bias^y.
\end{equation}

The rate above pertains only to the ``with-in sample'' fit, for those observations with $\dt = 1$. To obtain a rate on the entire sample, we use the sparse eigenvalues defined in \Eqref{SE}, as follows:
\begin{align*}
	\En[({x_i^*}'\dtildet)^2]^{1/2} & = \left\|Q^{1/2} Q_\t^{-1/2} Q_\t^{1/2}\dtildet \right\|_2		\\
	& \leq \left( \frac{\overline{\phi}\{Q, \tilde{S}^Y \cup S_Y^* \} }{\underline{\phi}\{Q_\t, \tilde{S}^Y \cup S_Y^* \}} \right)^{1/2} \left\| Q_\t^{1/2}\dtildet \right\|_2		\\
	& = \left( \frac{\overline{\phi}\{Q, \tilde{S}^Y \cup S_Y^* \} }{\underline{\phi}\{Q_\t, \tilde{S}^Y \cup S_Y^* \}} \right)^{1/2}  \Ent[({x_i^*}'\dtildet)^2]^{1/2}		\\
	& \leq \left( \frac{\overline{\phi}\{Q, \tilde{S}^Y \cup S_Y^* \} }{\underline{\phi}\{Q_\t, \tilde{S}^Y \cup S_Y^* \}} \right)^{1/2} \left( \frac{3 \lambda_Y \sqrt{|S_*|}}{\kappa_Y} +  2 \bias^y \right).
\end{align*}

The first conclusion of the Theorem now follows from this rate, the triangle inequality, and the bias condition, because
\begin{align*}
	\En[({x_i^*}'\btildet - \mut(x_i))^2] ^{1/2} & \leq \En[({x_i^*}'\btildet - {x_i^*}'\btruet)^2] ^{1/2} + \En[(\by)^2] ^{1/2} 		\\
	& \leq  \left( \frac{\overline{\phi}\{Q, \tilde{S}^Y \cup S_Y^* \} }{\underline{\phi}\{Q_\t, \tilde{S}^Y \cup S_Y^* \}} \right)^{1/2} \left( \frac{3 \lambda_Y \sqrt{|S_*|}}{\kappa_Y} +  2 \bias^y \right) + \bias^y.
\end{align*}

The $\ell_1$ bound now follows by the Cauchy-Schwarz inequality and the definition in \Eqref{SE}:
\begin{align*}
	\left\| \dtildet \right\|_1 \leq \sqrt{|\tilde{S}^Y \cup S_Y^*|} \left\| \dtildet \right\|_2 \leq \left(\frac{ |\tilde{S}^Y \cup S_Y^*| }{ \underline{\phi}\{Q, \tilde{S}^Y \cup S_Y^*\} } \right)^{1/2} \En[({x_i^*}'\dtildet)^2]^{1/2}.	
\end{align*}

Finally, we bound the size of the selected set of coefficients. First, note that optimality of $\btilde$ ensures that $| \tilde{S}^Y | \leq \overline{n}$. Then, restating the conclusion Lemma \ref{supp-ols S-hat} using the notation of the Theorem and \Eqref{supp-ols final rate}, then bounding $\overline{\phi}$ by $\dbarphi$ we find that
\[ | \tilde{S}^Y |    \leq    |S_Y^* | 16 L_n \sumTbar  \overline{\phi}\{Q_\t, \tilde{S}^Y\}    \leq    |S_Y^* | 16 L_n \sumTbar  \dbarphi\{Q_\t, |\tilde{S}^Y|\} .\]
The argument now parallels that used by \citeasnoun{Belloni-Chernozhukov2013_Bern}, relying on their result on the sublinearity of sparse eigenvalues. Let $\lceil m \rceil$ be the ceiling function and note that $\lceil{m}\rceil \leq 2m$. For any $m \in \N_Q^Y$, suppose that $|\tilde{S}^Y| > m$. Then, 
\begin{align*}
	| \tilde{S}^Y | & \leq |S_Y^* | 16 L_n \sumTbar  \dbarphi\{Q_\t, m (|\tilde{S}^Y| / m) \} 		\\
	& \leq  \left\lceil |\tilde{S}^Y| / m \right \rceil  |S_Y^* | 16 L_n \sumTbar  \dbarphi\{Q_\t, m  \}		\\
	& \leq  ( |\tilde{S}^Y| / m) |S_Y^* | 32 L_n \sumTbar  \dbarphi\{Q_\t, m  \}.
\end{align*}
Rearranging gives 
\[m \leq |S_Y^* | 32 L_n \sumTbar  \dbarphi\{Q_\t, m  \}\]
whence $m \not\in \N_Q^Y$. Minimizing over $\N_Q^Y$ gives the result.	 \qed

%%%%%%%%%%%%%%%%%%%%%%%%%%% 
%%%%%%%%%%%%%%%%%%%%%%%%%%%
\subsection{Verification of Assumption 3(c) for Group Lasso Estimators}

Under conditions imposed therein, Section 6 of the paper verifies that Assumptions \ref{consistent} and \ref{ATE RATES} hold for the proposed group lasso estimators $\muthat(x_i)$ and $\hat{p}_\t(x_i)$. Here we show that \ref{new} holds also. No additional assumptions are required. 

Recall that rather than generic  $\muthat(x)$ and $\mut(x_i)$, as above, we are now explicitly considering high-dimensional approximately sparse linear models for $\mut(x_i)$. In this context, we add and subtract the pseudotrue values to write 
\[ \sqrt{n} \En[ (\muthat(x_i) - \mut(x_i)) (1 - d_i^t/ p_\t(x_i))] = A_1 + A_2,\]
where 
\[A_1 =  \frac{1}{\sqrt{n}} \sumi ({x_i^*}\btruet - \mut(x_i)) \left(\frac{ p_\t(x_i)  -  \dt }{ p_\t(x_i) } \right)     	 \qquad \text{ and } \qquad   	 A_2 =  \frac{1}{\sqrt{n}} \sumi ({x_i^*}'\bhatt - {x_i^*}\btruet) \left(\frac{ p_\t(x_i)  -  \dt }{ p_\t(x_i) } \right).   \]

For the first term, $\E[ A_1 \vert \{x_i\}_{i = 1}^n] = 0$ holds as $\btruet$ is nonrandom. From Assumption 1(b) and the definition of the bias term $\bias^y$ we find
\[ \E[ A_1^2 \vert \{x_i\}_{i = 1}^n] = \frac{1}{n} \sumi ({x_i^*}\btruet - \mut(x_i))^2 \E \left[ \left(\frac{ p_\t(x_i)  -  \dt }{ p_\t(x_i) } \right)^2 \right] \leq C \En ({x_i^*}\btruet - \mut(x_i))^2 \leq C (\bias^y)^2. \]
Therefore $|A_1| = O_{P_n}(\bias^y) = o_{P_n}(1)$, where the second equality is assumed in the bias condition of Assumption 4 and the first equality follows from Markov's inequality.

For the second term, define $\tilde{\Sigma}_{\t,j} = \E\left[ (x_{i,j}^*)^2 (\dt - p_\t(x_i))^2 / p_\t(x_i)^2 \right]$ and then proceed as follows:
\begin{align*}
	 A_{1}  & = \frac{1}{\sqrt{n}} \sumi \left(\frac{ p_\t(x_i)  -  \dt }{ p_\t(x_i) } \right) \sum_{j \in \hat{S}_Y} x_{i,j}^*  (\bhattj - \btruetj)		\\
	& =  \sum_{j \in \hat{S}_Y} \left\{ \frac{1}{\sqrt{n}} \sumi \frac{x_{i,j}^*  (p_\t(x_i) - \dt) / p_\t(x_i)}{  \tilde{\Sigma}_{\t,j}^{1/2} } \right\} \tilde{\Sigma}_{\t,j}^{1/2}  (\bhattj - \btruetj)		\\
	& \leq \left( \maxj \tilde{\Sigma}_{\t,j}^{1/2} \right) \left( \maxj \frac{1}{\sqrt{n}} \sumi \frac{x_{i,j}^*  (p_\t(x_i) - \dt) / p_\t(x_i)}{  \tilde{\Sigma}_{\t,j}^{1/2} }   \right) \left\| \bhatt - \btruet \right\|_1		\\
	& = O(1) O_{P_n}( \log(p) )  \left\| \bhatt - \btruet \right\|_1.
\end{align*}
This quantity is $o_{P_n}(1)$ by Corollary 5 in the original paper, which among other results, gives a rate for the $\ell_1$ norm of the estimated coefficients. For the final equality, Assumptions 1(b), 2(b), and 2(c) imply that $\maxj \tilde{\Sigma}_{\t,j} = O(1)$, while the center factor is bounded by applying the moderate deviation theory for self-normalized sums of \citeasnoun[Theorem 7.4]{delaPena-Lai-Shao2009_book} and in particular \citeasnoun[Lemma 5]{BCCH2012_Ecma}. To apply this theory, first note that the summand of the center factor has bounded third moment and second moment bounded away from zero from Assumptions 1(b) and 2. $\Sigma_{t,j}$ normalizes the second moment, and the theory applies under Assumption 4.

%%%%%%%%%%%%%%%%%%%%%%%%%%%%
%%%%%%%%%%%%%%%%%%%%%%%%%%%%
\section{Additional Simulation Results}

The DGP is as described in the main text. The mean comparison group sample sizes for the various DGPs are in Table \ref{supp-table-n-controls}. Figure \ref{supp-fig-2000} shows the analogue of Figure \ref{fig-plot-main} Panels (a) and (b) with 2000 covariates. The manually chosen values of $\delta_\D$ and $\delta_Y$ do not appear well-suited to one particular DGP. Figures \ref{supp-fig-cross-p1000}, \ref{supp-fig-cross-p1500}, and \ref{supp-fig-cross-p2000} show the coverage results based using 10-fold cross validation to choose the penalty parameters, for 1000, 1500, and 2000 covariates, respectively. For these three, the exponents $\alpha_\beta$ and $\alpha_\gamma$ range from one to four, and hence the functions are always sparse (to a certain degree). For nonsparse functions, the current R routines are not reliable. This will be explored in future software development. Cross-validation choices perform very well.

\begin{table}
	\begin{center}
		\begin{threeparttable}
			\caption{Mean Comparison Group Sample Sizes for Various Specifications}
			\label{supp-table-n-controls}
			\small
			\begin{tabular}{l r r r r r r}
				\hline
				\hline\noalign{\smallskip}
				& \multicolumn{6}{c}{Multiplier $\rho_\gamma$ (exponent $\alpha_\gamma=2$)}  	  \\
				\cline{2-7} \noalign{\smallskip}
				No.\ of Covariates: 	&  0.01 &  0.05  &  0.25  &  0.50  &  0.75  &  1		\\
				\hline
				1000  	&  498.368  &  488.311 &   439.159  &  384.060 &   338.106  &  301.793		\\
				1500  	&  497.832   &  487.987   &  438.608   &  383.966   &  338.408  &  301.924		\\
				2000  	&  498.368  &  488.311 &   439.159  &  384.060 &   338.106  &  301.793		\\
				\hline
			\end{tabular}
			\begin{tabular}{l r r r r r r r r }
				\noalign{\bigskip}
				\noalign{\bigskip}
				& \multicolumn{6}{c}{Exponent $\alpha_\gamma$ (multiplier $\rho_\gamma=1$)}  	  \\
				\cline{2-9} \noalign{\smallskip}
				No.\ of Covariates: 	&  0.125 &  0.25  &  0.5  &  .75  &  1  &  2    &  3   &  4		\\
				\hline \noalign{\smallskip}
				1000  	&  456.373   &  420.073  &   341.358   &  312.818  &   305.861  &   301.793   &  302.134  &  302.764		\\
				1500  	&  462.646  &   426.711  &   342.527    &  312.544   &   305.965   &   301.924   &    301.943   &   302.229  		\\
				2000  	&  462.646  &   426.711  &   342.527   &   312.544   &   305.965   &   301.924   &    301.943   &   302.229  		\\
				\hline
			\end{tabular}
		\end{threeparttable}
	\end{center}
\end{table}

\begin{figure}
	\begin{center}
		\caption{Empirical Coverage of 95\% Confidence Intervals, Varying Signal Strength and Sparsity of $p_\t(x)$ and $\mut(x)$, 2000 Covariates}
		\label{supp-fig-2000}
		\includegraphics[scale=1]{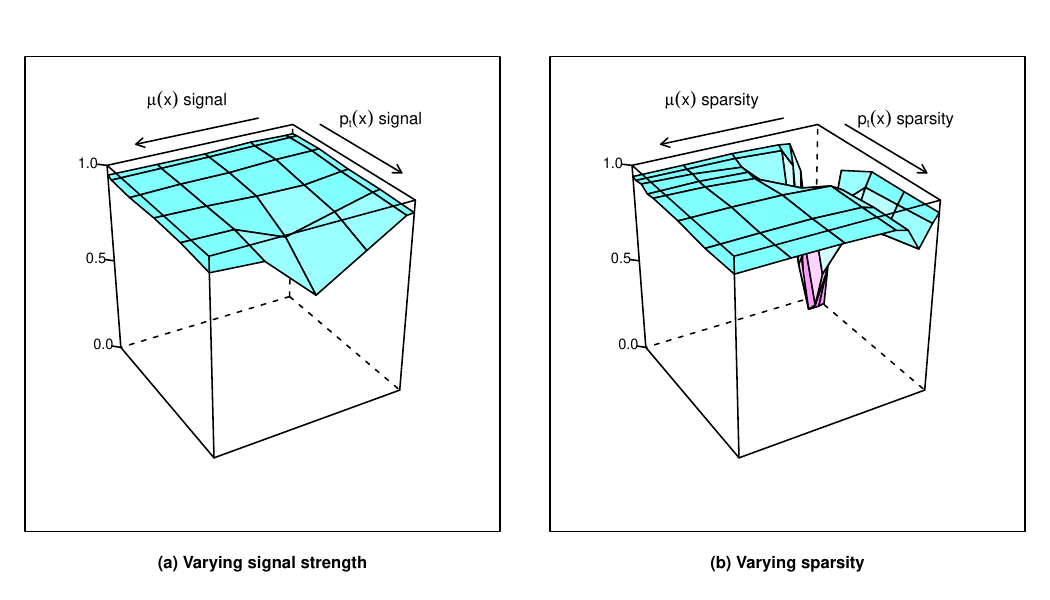}
	\end{center}
\end{figure}

\begin{figure}
	\begin{center}
		\caption{Empirical Coverage of 95\% Confidence Intervals, Penalty Chosen with Cross-Validation, Varying Signal Strength and Sparsity of $p_\t(x)$ and $\mut(x)$, 1000 Covariates}
		\label{supp-fig-cross-p1000}
		\includegraphics[scale=1]{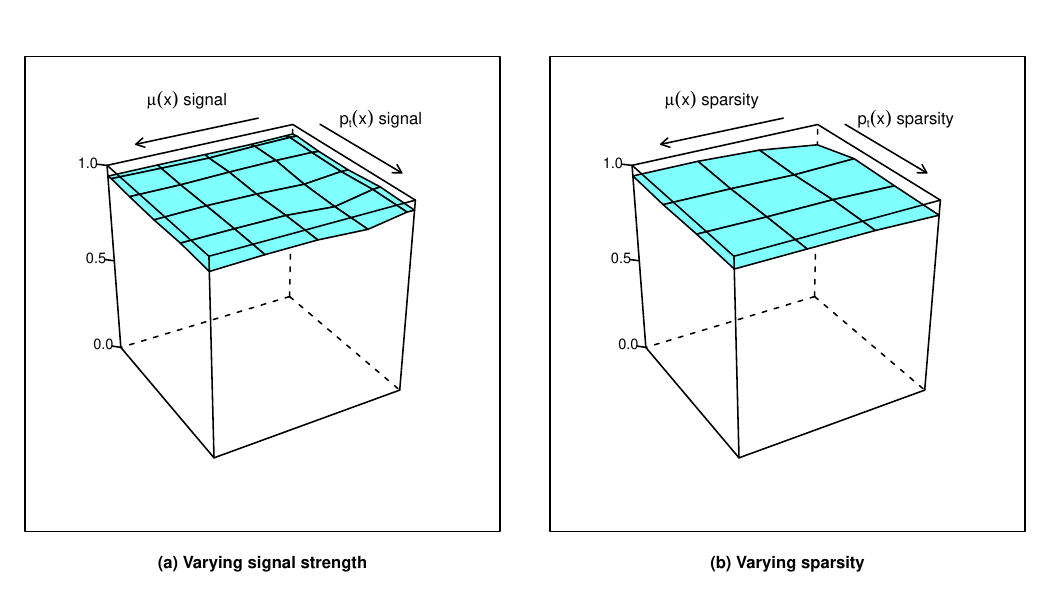}
	\end{center}
\end{figure}

\begin{figure}
	\begin{center}
		\caption{Empirical Coverage of 95\% Confidence Intervals, Penalty Chosen with Cross-Validation, Varying Signal Strength and Sparsity of $p_\t(x)$ and $\mut(x)$, 1500 Covariates}
		\label{supp-fig-cross-p1500}
		\includegraphics[scale=1]{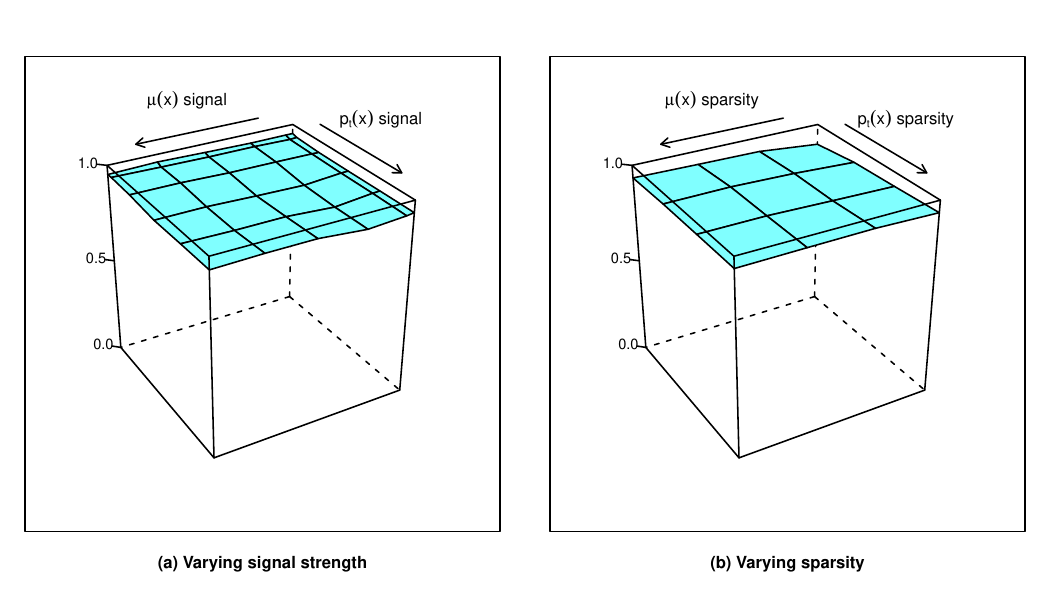}
	\end{center}
\end{figure}

\begin{figure}
	\begin{center}
		\caption{Empirical Coverage of 95\% Confidence Intervals, Penalty Chosen with Cross-Validation, Varying Signal Strength and Sparsity of $p_\t(x)$ and $\mut(x)$, 2000 Covariates}
		\label{supp-fig-cross-p2000}
		\includegraphics[scale=1]{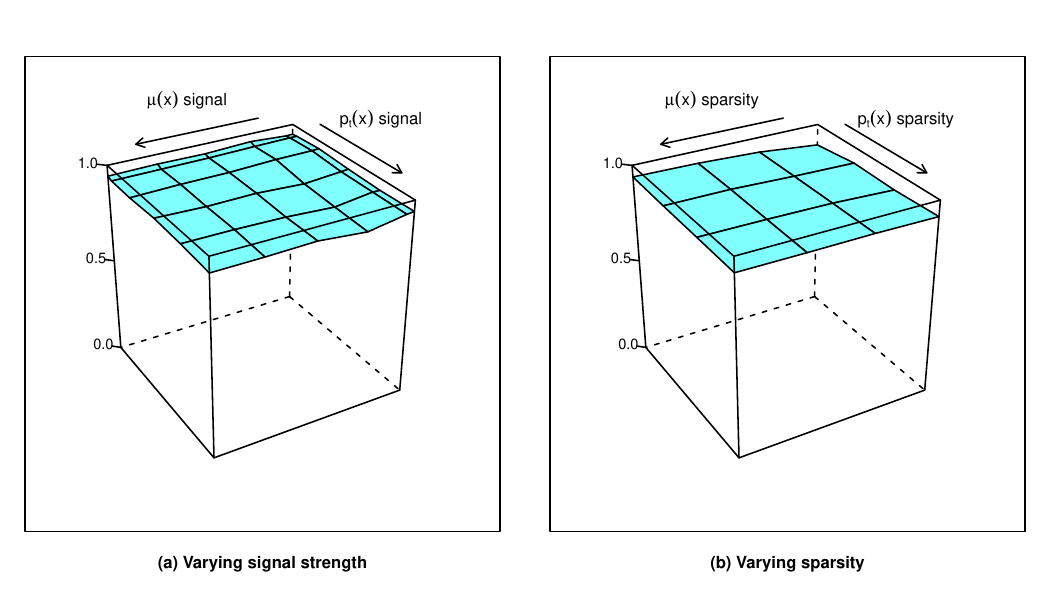}
	\end{center}
\end{figure}

\end{appendices}

\end{document}